\documentclass[12pt,a4paper]{amsart}

\usepackage{amsmath,amsfonts,amssymb}
\usepackage{color}
\usepackage[hmargin=2cm,vmargin=2cm]{geometry}

\usepackage{hyperref}
\usepackage{nameref,zref-xr}                    
\usepackage[all]{xy}           

\usepackage{cancel}

\newcommand{\mbP}{\mathbb P}
\newcommand{\mbZ}{\mathbb Z}
\newcommand{\mbC}{\mathbb C}

\newcommand{\oM}{\overline{\mathcal M}}

\def\cM{{\mathcal{M}}}
\def\oM{{\overline{\mathcal{M}}}}

\def\d{{\partial}}

\newcommand{\eps}{\varepsilon}

\newcommand{\even}{\mathrm{even}}
\newcommand{\ct}{\mathrm{ct}}

\newcommand{\Coef}{\mathrm{Coef}}

\DeclareMathOperator{\Deg}{Deg}

\newcommand{\tv}{\widetilde v}
\renewcommand{\top}{\mathrm{top}}

\DeclareMathOperator{\Aut}{Aut}

\newcommand{\tOmega}{\widetilde\Omega}

\newcommand{\gl}{\mathrm{gl}}

\newcommand{\tR}{\widetilde{R}}

\newcommand{\ext}{\mathrm{ext}}

\newcommand{\mcC}{\mathcal{C}}

\newcommand{\mcF}{\mathcal{F}}

\newcommand{\mcL}{\mathcal{L}}

\newcommand{\tT}{\widetilde{T}}

\newcommand{\tpi}{\widetilde{\pi}}

\newcommand{\mcO}{\mathcal{O}}

\def\t{\mathfrak{t}}

\def\d{\partial}
\def\n{\noindent}
\def\f{\frac}

\newcommand{\beq}{\begin{equation}}
\newcommand{\eeq}{\end{equation}}
\newcommand{\orig}{\mathrm{orig}}
\newcommand{\Id}{\mathrm{Id}}
\newcommand{\mcR}{\mathcal{R}}
\newcommand{\mcH}{\mathcal{H}}

\newcommand{\End}{\mathrm{End}}
\newcommand{\tmcF}{\widetilde{\mathcal{F}}}
\newcommand{\tq}{\widetilde{q}}
\newcommand{\hmcF}{\widehat{\mathcal{F}}}
\newcommand{\omcF}{\overline{\mathcal{F}}}
\newcommand{\oF}{\overline{F}}
\renewcommand{\tt}{\widetilde{t}}
\newcommand{\oc}{{\overline{c}}}

\newcommand{\ot}{\overline{t}}
\newcommand{\tS}{\widetilde{S}}
\def\un{{1\!\! 1}}
\newcommand{\tp}{\widetilde{p}}
\newcommand{\tPsi}{\widetilde{\Psi}}
\newcommand{\diag}{\mathrm{diag}}
\newcommand{\tD}{\widetilde{D}}
\newcommand{\tGamma}{\widetilde{\Gamma}}

\newcommand{\tA}{\widetilde{A}}
\newcommand{\triv}{\mathrm{triv}}
\newcommand{\const}{\mathrm{const}}
\newcommand{\GL}{\mathrm{GL}}
\newcommand{\tgamma}{\widetilde{\gamma}}
\newcommand{\txi}{\widetilde{\xi}}

\newcommand{\oxi}{\overline{\xi}}

\newcommand{\Lie}{\mathrm{Lie}}
\newcommand{\tnabla}{\widetilde{\nabla}}
\newcommand{\Ext}{\mathrm{Ext}}
\newcommand{\otau}{{\overline{\tau}}}

\newcommand{\oH}{\overline{H}}
\newcommand{\ET}{\mathrm{ET}}
\newcommand{\Cont}{\mathrm{Cont}}
\newcommand{\te}{\widetilde{e}}
\newcommand{\tk}{\widetilde{k}}
\newcommand{\mcV}{\mathcal{V}}
\newcommand{\tcL}{\widetilde{\mathcal{L}}}
\newcommand{\mcE}{\mathcal{E}}

\newtheorem{theorem}{Theorem}[section]
\newtheorem{proposition}[theorem]{Proposition}
\newtheorem{lemma}[theorem]{Lemma}

\newtheorem{definition}[theorem]{Definition}

\newtheorem{remark}[theorem]{Remark}

\numberwithin{equation}{section}

\begin{document}

\title{Semisimple flat F-manifolds in higher genus}

\author{Alessandro Arsie}
\address{A.~Arsie:\newline Department of Mathematics and Statistics, The University of Toledo,\newline 2801 W. Bancroft St., 43606 Toledo, OH, USA}
\email{alessandro.arsie@utoledo.edu}

\author{Alexandr Buryak}
\address{A.~Buryak:\newline Faculty of Mathematics, National Research University Higher School of Economics,\newline 6 Usacheva str., Moscow, 119048, Russian Federation;\smallskip\newline
Center for Advanced Studies, Skolkovo Institute of Science and Technology,\newline 1 Nobel str., Moscow, 143026, Russian Federation}
\email{aburyak@hse.ru}

\author{Paolo Lorenzoni}
\address{P.~Lorenzoni:\newline Dipartimento di Matematica e Applicazioni, Universit\`a di Milano-Bicocca,\newline Via Roberto Cozzi 53, I-20125 Milano, Italy}
\email{paolo.lorenzoni@unimib.it}

\author{Paolo Rossi}
\address{P.~Rossi:\newline Dipartimento di Matematica ``Tullio Levi-Civita'', Universit\`a degli Studi di Padova,\newline Via Trieste 63, 35121 Padova, Italy}
\email{paolo.rossi@math.unipd.it}

\begin{abstract}
In this paper, we generalize the Givental theory for Frobenius manifolds and cohomological field theories to flat F-manifolds and F-cohomological field theories. In particular, we define a notion of Givental cone for flat F-manifolds, and we provide a generalization of the Givental group as a matrix loop group acting on them. We show that this action is transitive on semisimple flat F-manifolds. We then extend this action to F-cohomological field theories in all genera. We show that, given a semisimple flat F-manifold and a Givental group element connecting it to the constant flat F-manifold at its origin, one can construct a family of F-CohFTs in all genera, parameterized by a vector in the associative algebra at the origin, whose genus~$0$ part is the given flat F-manifold. If the flat F-manifold is homogeneous, then the associated family of F-CohFTs contains a subfamily of homogeneous F-CohFTs. However, unlike in the case of Frobenius manifolds and CohFTs, these homogeneous F-CohFTs can have different conformal dimensions, which are determined by the properties of a certain metric associated to the flat F-manifold.
\end{abstract}

\date{\today}

\maketitle

\tableofcontents

\section*{Introduction}
In a series of influential papers \cite{Giv01a,Giv01b,Giv04}, A. Givental, inspired by localization formulas in the Gromov--Witten theory of projective spaces, described a technique to express the all genera descendant Gromov--Witten potential of a target variety with semisimple quantum cohomology in terms of the action of a certain operator on $N$ copies (where $N$ is the dimension of the target variety's cohomology) of the descendant potential of a point (also known as the Witten--Kontsevich tau-function).\\

Restricting the attention to the genus $0$ descendant potential, the situation is described as the action of a certain loop group of matrices on the space of descendant potentials of calibrated Frobenius manifolds (i.e., solutions to certain differential equations inspired by Gromov--Witten theory and well-known in the theory of Frobenius manifolds, see \cite{DZ05})). This action is in fact transitive when further restricted to semisimple Frobenius manifolds. In this sense, taking as a starting point the $N$-fold product of the trivial $1$-dimensional Frobenius manifold, the descendant potential of any other fixed calibrated semisimple Frobenius manifold can be recovered by the action of an operator representing an element of the Givental group connecting these two semisimple Frobenius manifolds.\\

The geometric set-up for this result is interpreting the genus $0$ descendant potential as the generating function for a Lagrangian cone in an infinite dimensional symplectic vector space. The Givental group is then a group of symplectic transformations acting on the set of all such cones.\\

In \cite{Giv01b}, Givental conjectured how to extend such action to the potential at all genera as a canonical quantization of the above symplectic action. In particular, Givental's formula can be seen as a way to reconstruct higher genus descendant Gromov--Witten invariants of a target variety from its (genus $0$) quantum cohomology, as long as this is semisimple.\\

In \cite{Tel12}, Teleman proved Givental's reconstruction formula by reformulating the problem in the language of cohomological field theories (CohFTs), families of cohomology classes on the moduli space of stable algebraic curves introduced in \cite{KM94} to axiomatize the properties of Gromov--Witten invariants. In this context, the Givental symplectic loop group is seen as acting directly on the space of all CohFTs, and this action restricts to a transitive action on the space of semisimple CohFTs. The explicit form of this action had been known to experts for a while before being first accurately described in \cite{PPZ15}.\\

In this language, the Givental--Teleman reconstruction theorem relies on two results: the transitivity of the action of the Givental group in genus $0$ (i.e., on descendant potentials of semisimple calibrated Frobenius manifolds) and the fact that a semisimple CohFT is essentially uniquely reconstructable from its genus $0$ part (in fact, up to insertions of Hodge classes in the general case, and uniquely for homogeneous CohFTs). This shows, in particular, that the Givental group acts transitively on semisimple CohFTs.\\

In this paper, we consider a generalization of Frobenius manifolds called flat F-manifolds. F-manifolds were introduced by Hertling and Manin in \cite{HM99} (see also the book \cite{Man99}). They are generalizations of Frobenius manifolds, for which one drops the potentiality condition and the presence of a metric. Flat F-manifolds were first studied by Getzler~\cite{Get04} and Manin~\cite{Man05} (in Getzler's paper they are called Dubrovin manifolds) and they are often useful to capture interesting structures in singularity theory and algebraic geometry.\\

The relevance of F-manifolds stems also from the fact that in this framework it is possible to capture phenomena that do not always have a natural set-up in the theory of Frobenius manifolds. Moreover, many important constructions in the theory of Frobenius manifolds admit a natural generalization in this framework.
\begin{itemize}
\item {\it Painlev\'e transcendents}. Three dimensional semisimple flat F-manifolds equipped with a linear Euler vector field are parameterized by solutions of the full family of Painlev\'e~VI (PVI) equations (\cite{Lor14,AL15}) and in the non-semisimple regular case by the full family of PIV and PV equations \cite{AL15}. The same results were proved later with different methods in \cite{KMS15} (semisimple case) and \cite{KM19} (non-semisimple case). In the last paper, the authors proved that PII--PVI appear, as special cases, in dimension~$4$. In the case of Frobenius manifolds, no results of this type are known for the non-semisimple situation.
\item {\it Reflection groups}. The Dubrovin construction of a Frobenius manifold structure on the orbit space of a Coxeter group can be generalized to well-generated complex reflection groups \cite{KMS15,AL17,KMS18}. Moreover, the flat coordinates of the flat F-manifold are distinguished basic invariant polynomials of the group that generalize the Saito flat coordinates. 
\item {\it Open Gromov--Witten theory}. Flat F-manifolds appear in genus $0$ open Gromov--Witten theory, as remarked in \cite{BB19}. In particular, the open WDVV equations introduced by Horev and Solomon in \cite{HS12} can be thought as a particular case of the oriented WDVV equations that play the role of the usual WDVV's in this  more general setting.
\item {\it Higher genus extension}. The notion of a Frobenius manifold can be naturally extended
to higher genus considering CohFTs. It was proved in \cite{BR18} that a similar extension can be defined also in the case of flat F-manifolds and leads to the notion of an F-CohFT. In the same paper, the first, all genera, explicit example of an F-CohFT, relevant for open Gromov--Witten theory, was constructed.
\end{itemize}  
Due to the variety of applications ranging from integrable systems to Gromov--Witten theory, it is natural to ask whether Givental's theory and the Givental--Teleman reconstruction can be generalized to flat F-manifolds and F-CohFTs. In this paper, we give an answer to this question in the semisimple case.\\

The paper is organized as follows.\\

In Section \ref{section:flat F-manifolds}, we develop the theory of semisimple flat F-manifolds in canonical coordinates, reformulating some known facts and introducing stronger or more precise technical results needed for Section \ref{section:Givental theory for flat F-manifolds}, where we introduce a suitably generalized version of the Givental theory for flat F-manifolds.\\

In particular, we prove that to a calibrated flat F-manifold one can univocally associate a sequence of descendant vector potentials describing a cone in an infinite dimensional vector space. When the flat F-manifold is Frobenius, such cone is Lagrangian with respect to the symplectic structure constructed via the flat metric, as proved in \cite{Giv04}. We then introduce a larger Givental-type loop group (which is not symplectic anymore) and a corresponding action, which is defined on the space of such descendant cones. We prove that the action is transitive for semisimple flat F-manifolds, thereby completely generalizing the genus $0$ Givental theory. We also recall the definition of a homogeneous flat F-manifold given in \cite{BB19} and the related notions of a Saito structure without metric (introduced in \cite{Sab98}) and of a bi-flat F-manifold (introduced in \cite{AL13}). For homogeneous flat F-manifolds, we show that the $R$-matrix defining an element of the generalized Givental group is uniquely determined. \\

In Section \ref{section:F-CohFT}, we study F-CohFTs. They are in fact generalizations of partial CohFTs: the gluing axiom at a nonseparating node is dropped and, moreover, the complete equivariance of the classes with respect to permutation of marked points is broken, as one of them carries a co-vector, instead of a vector. This removes the necessity of a metric, which is then also dropped. Indeed, partial CohFTs in genus $0$ still reduce to Frobenius manifolds, while F-CohFTs give flat F-manifolds.\\

Finally, in Section \ref{section:group action of F-CohFTs}, we extend our generalized Givental group action to F-CohFTs in all genera and we show our main result: given any semisimple flat F-manifold, we can construct an F-CohFT with that F-manifold as its genus $0$ part. In fact, because of the absence of the gluing axiom at nonseparating nodes, some genus $1$ information is needed to fix the nonzero genus part. This is done by specifying the degree $0$ part of the F-CohFT on $\oM_{1,1}$, which amounts to a vector $G_0$ in the F-CohFT phase space $V$. For each choice of such a vector, we construct a different F-CohFT with the given flat F-manifold as its genus $0$ part. If the flat F-manifold is homogeneous, then we construct a decomposition $V=\oplus_{i\in I} V_i$ and prove that for $G_0\in V_i$ the resulting F-CohFT is homogeneous of conformal dimension $\gamma_i$. The collection of numbers~$\gamma_i$, $i\in I$ is determined by the properties of a certain metric that is associated to the flat F-manifold.\\

It is worth underlying that, despite in the definition of a flat F-manifold one drops the requirement of a metric, typical instead for Frobenius structures, a metric compatible with the product can nonetheless be constructed \cite{AL13} around any semisimple point. This hidden metric plays a crucial role in our proof of the transitivity of the Givental action on the set of semisimple flat F-manifolds. Moreover, despite the definition of homogeneous flat F-manifold does not involve the notion of a conformal dimension either, the same hidden metric is at the origin of the appearence of a tuple of conformal dimensions for the homogeneous higher genus extensions of a given homogeneous semisimple flat F-manifold.\\ 

To put these results in perspective, let us point out the following. When the unit axiom is dropped, flat F-manifolds are essentially equivalent to Hycomm algebras. In this case, in a slightly more general set-up, a generalization of the Givental group action for the case of genus zero curves is given in Section 6 of~\cite{KMS13}. A conceptual explanation for this action as the choice of a homotopy trivialisation of the circle action in BV algebras is also given in~\cite{KMS13} and, from a different perspective, in~\cite{DSV15}. A Lagrangian-cone-type description of the Givental theory for Losev--Manin CohFTs (which exist only in genus $0$) is given in~\cite{SZ11}. However, we want to stress that the present paper is, to the best of our knowledge, the first one to consider and fully develop the higher genus Givental theory for flat F-manifolds and F-CohFTs.\\

We remark that, beside the aforementioned applications to singularity theory and moduli spaces of curves, another motivation for studying F-CohFTs and developing a corresponding Givental-type theory comes from integrable systems. Indeed, as shown in \cite{BR18}, to any F-CohFT one can associate, via a suitable generalization of the double ramification hierarchy construction of \cite{Bur15,BR16}, an infinite hierarchy of compatible evolutionary PDEs (in particular, systems of conservation laws). The dispersionless limit of this hierarchy is the principal hierarchy associated with the corresponding flat F-manifold, see Section \ref{section:Givental theory for flat F-manifolds}. In the literature, there are interesting examples of integrable hierarchies like those studied by Antonowicz and Fordy in \cite{AF87} whose dispersionless limit can be interpreted as the principal hierarchy of a (homogeneous) flat F-manifold. \\

The results of the present paper then allow to construct a family of dispersive deformations of the principal hierarchy of any semisimple flat F-manifold, parametrized by a vector $G_0$ at its origin. In the homogeneous case, choices of $G_0$ exist for which the deformation is homogeneous. We will study in detail such dispersive deformations and the properties of the double ramification hierarchy of an F-CohFT in our next paper.\\

\subsection*{Acknowledgments}
We thank Sergey Shadrin for providing useful remarks and pointing out some references. The work of A.~B. is an output of a research project implemented as part of the Basic Research Program at the National Research University Higher School of Economics (HSE University). P.~L. is supported by MIUR - FFABR funds 2017 and by funds of H2020-MSCA-RISE-2017 Project No. 778010 IPaDEGAN.\\

%%%%%%%%%%%%%%%%%%%%%%%%%%%%%%%%%%%%%%%%%%%%%%%%%%%%%%%%%%%%%%%%%%%%%%%%%
%%%%%%%%%%%%%%%%%%%%%%%%%%%%%%%%%%%%%%%%%%%%%%%%%%%%%%%%%%%%%%%%%%%%%%%%%

\section{Flat F-manifolds around a semisimple point}\label{section:flat F-manifolds}

After recalling the definition of a flat F-manifold as a generalization of the notion of a Frobenius manifold, in this section, we show that a flat F-manifold around a semisimple point possesses a metric, and we construct rotation coefficients and a sequence of $R$-matrices. These objects will play an important role in our later construction of an F-cohomological field theory in all genera associated to a flat F-manifold.

\smallskip

\subsection{Flat F-manifolds and Frobenius manifolds}

We recall here the following facts and definitions from \cite{Get04,Man05}, see also \cite{AL18} and \cite{Dub96}.

\smallskip

\begin{definition}
A flat F-manifold $(M,\nabla,\circ,e)$ is the datum of an analytic manifold $M$, an analytic connection $\nabla$ in the tangent bundle $T M$, an algebra structure $(T_p M,\circ)$ with unit $e$ on each tangent space, analytically depending on the point $p\in M$, such that the one-parameter family of connections $\nabla_z=\nabla+z\circ$ is flat and torsionless for any $z\in\mbC$, and $\nabla e=0$.
\end{definition}

\smallskip

From the flatness and the torsionlessness of $\nabla_z$ one can deduce the commutativity and the associativity of the algebras $(T_pM,\circ)$. Moreover, if one choses flat coordinates $t^\alpha$, $1\le\alpha\le N$, $N=\dim M$, for the connection $\nabla$, then it is easy to see that locally there exist analytic functions $F^\alpha(t^1,\ldots,t^N)$, $1\leq\alpha\leq N$, such that the second derivatives 
\begin{gather}\label{eq:structure constants of flat F-man}
c^\alpha_{\beta\gamma}=\frac{\d^2 F^\alpha}{\d t^\beta \d t^\gamma}
\end{gather}
are the structure constants of the algebras $(T_p M,\circ)$,
\begin{gather*}
\frac{\d}{\d t^\beta}\circ\frac{\d}{\d t^\gamma}=c^\alpha_{\beta\gamma}\frac{\d}{\d t^\alpha}.
\end{gather*}
Also, in the coordinates~$t^\alpha$ the unit $e$ has the form $e=A^\alpha\frac{\d}{\d t^\alpha}$ for some constants $A^\alpha\in\mbC$. Note that we use Einstein's convention of sum over repeated Greek indices. From the associativity of the algebras $(T_p M,\circ)$ and the fact that the vector field $A^\alpha\frac{\d}{\d t^\alpha}$ is the unit it follows that
\begin{align}
A^\mu\frac{\d^2 F^\alpha}{\d t^\mu\d t^\beta} &= \delta^\alpha_\beta, && 1\leq \alpha,\beta\leq N,\label{eq:axiom1 of flat F-man}\\
\frac{\d^2 F^\alpha}{\d t^\beta \d t^\mu} \frac{\d^2 F^\mu}{\d t^\gamma \d t^\delta} &= \frac{\d^2 F^\alpha}{\d t^\gamma \d t^\mu} \frac{\d^2 F^\mu}{\d t^\beta \d t^\delta}, && 1\leq \alpha,\beta,\gamma,\delta\leq N.\label{eq:axiom2 of flat F-man}
\end{align}
The $N$-tuple of functions $\oF=(F^1,\ldots,F^N)$ is called the {\it vector potential} of the flat F-manifold.\\

Conversely, if $M$ is an open subset of $\mbC^N$ and $F^1,\ldots,F^N\in\mcO(M)$ are functions satisfying equations~\eqref{eq:axiom1 of flat F-man} and~\eqref{eq:axiom2 of flat F-man}, then these functions define a flat F-manifold $(M,\nabla,\circ,A^\alpha\frac{\d}{\d t^\alpha})$ with the connection~$\nabla$ given by $\nabla_{\frac{\d}{\d t^\alpha}}\frac{\d}{\d t^\beta}=0$, and the multiplication $\circ$ given by the structure constants~\eqref{eq:structure constants of flat F-man}.

\smallskip

\begin{definition}
Consider a flat F-manifold $(M,\nabla,\circ,e)$ and a symmetric nondegenerate bilinear form~$g$ (often called a metric) on the tangent spaces~$T_pM$ analytically depending on the point $p\in M$. We say that $g$ is compatible with the product $\circ$ if
\begin{gather*}
g(X\circ Y, Z)=g(X,Y\circ Z),
\end{gather*}
for any local vector fields $X,Y,Z $ on $M$.
\end{definition}

\smallskip

A point $p\in M$ of an $N$-dimensional flat F-manifold $(M,\nabla,\circ,e)$ is called \textit{semisimple} if $T_pM$ has a basis of idempotents $\pi_1,\dots,\pi_N$ satisfying $\pi_k \circ \pi_l = \delta_{k,l} \pi_k$. Moreover, locally around such a point one can choose coordinates $u^i$ such that $\frac{\d}{\d u^k}\circ\frac{\d}{\d u^l}=\delta_{k,l}\frac{\d}{\d u^k}$. These coordinates are called {\it canonical coordinates}. In particular, this means that semisimplicity at a point is an open property on $M$. In canonical coordinates we have $e=\sum_i\frac{\d}{\d u^i}$.\\

A flat F-manifold given by a vector potential $(F^1,\ldots,F^N)$ is called {\it homogeneous} if there exists a vector field $E$ of the form
\begin{equation}\label{Euler}
E=\sum_{\alpha=1}^N(\underbrace{(1-q_\alpha)t^\alpha+r^\alpha}_{=:E^\alpha})\frac{\d}{\d t^\alpha},\quad q_\alpha,r^\alpha\in\mbC,
\end{equation}
satisfying $[e,E]=e$ and such that
\begin{gather*}
E^\mu\frac{\d F^\alpha}{\d t^\mu}=(2-q_\alpha)F^\alpha+A^\alpha_\beta t^\beta+B^\alpha,\quad\text{for some $A^\alpha_\beta,B^\alpha\in\mbC$}.
\end{gather*}
Note that this equation can be written more invariantly as $\Lie_E(\circ)=\circ$, where $\Lie_E$ denotes the Lie derivative. The vector field $E$ is called the {\it Euler vector field}.

\smallskip

\begin{definition}
A flat F-manifold $(M,\nabla,\circ,e)$ is called a Frobenius manifold if it is equipped with a metric $\eta$ compatible with the product $\circ$ and such that $\nabla \eta = 0$. The connection $\nabla$ is then the Levi-Civita connection associated to~$\eta$. A Frobenius manifold will be denoted by a tuple $(M,\eta,\circ,e)$. 
\end{definition}

\smallskip

Homogeneous Frobenius manifolds are sometimes called {\it conformal} Frobenius manifolds.\\

In case a flat F-manifold is actually Frobenius, the vector potential $\oF$ can be shown to descend locally from a Frobenius potential $F(t^*)$ as $F^\alpha(t^*) = \eta^{\alpha\mu} \frac{\d F(t^*)}{\d t^\mu}$ and the Frobenius potential $F(t^*)$ satisfies
\begin{align}
A^\mu\frac{\d^3 F}{\d t^\mu\d t^\alpha \d t^\beta} &= \eta_{\alpha\beta}, && 1\leq \alpha,\beta\leq N,\label{eq:axiom1 for Frobenius manifolds}\\
\frac{\d^3 F}{\d t^\alpha \d t^\beta \d t^\mu} \eta^{\mu \nu}\frac{\d^3 F}{\d t^\nu \d t^\gamma \d t^\delta} &= \frac{\d^3 F}{\d t^\alpha \d t^\gamma \d t^\mu} \eta^{\mu \nu}\frac{\d^3 F}{\d t^\nu \d t^\beta \d t^\delta}, && 1\leq \alpha,\beta,\gamma,\delta\leq N.\label{eq:axiom2 for Frobenius manifolds}
\end{align}
In particular, the structure functions $c^\alpha_{\beta\gamma}$ of the algebras $(T_pM,\circ)$ can be written as $c^\alpha_{\beta\gamma} =\eta^{\alpha\mu} \frac{\d^3 F}{\d t^\mu \d t^\beta \d t^\gamma}$, $1\leq \alpha,\beta,\gamma\leq N$.

\smallskip

\subsection{Metric, rotation coefficients and $R$-matrices}\label{subsection:metric for a flat F-manifold}

Consider a flat F-manifold $(M,\nabla,\circ,e)$ around a semisimple point. Let $u^1,\ldots,u^N$ be the canonical coordinates. By $t^1,\ldots,t^N$ we denote the flat coordinates.\\

In general, our flat F-manifold is not Frobenius and so it doesn't possess a metric which is covariantly constant with respect to $\nabla$ and compatible with the product $\circ$. However, there is a natural metric compatible with the product $\circ$, which was first constructed in~\cite{AL13}. Introduce a matrix~$\tPsi$ by
$$
\tPsi:=\left(\frac{\d u^i}{\d t^\alpha}\right).
$$
Note that, in canonical coordinates, the connection $\nabla_z=\nabla+z\circ$ is given by
$$
\nabla+z\circ=d-d\tPsi\cdot\tPsi^{-1}+z dU,
$$
where $U:=\diag(u^1,\ldots,u^N)$.

\smallskip

\begin{proposition}\label{proposition:properties of tGamma}
1. The matrix $d\tPsi\cdot\tPsi^{-1}$ has the form 
$$
d\tPsi\cdot\tPsi^{-1}=\tD+[\tGamma,dU],
$$
where $\tD$ is a diagonal matrix consisting of one-forms and $\tGamma$ is a matrix with vanishing diagonal entries. \\
2. We have $d\tD=0$ and 
\begin{gather}\label{eq:system for tGamma}
d[\tGamma,dU]=\tD\wedge [\tGamma,dU]+[\tGamma,dU]\wedge\tD+[\tGamma,dU]\wedge [\tGamma,dU].
\end{gather}
\end{proposition}
\begin{proof}
Denote $M:=d\tPsi\cdot\tPsi^{-1}$. The flatness of $\nabla_z$ is equivalent to the equation
\begin{gather}\label{eq:equations for M}
-d M+(-M+zdU)\wedge(-M+zdU)=0\Leftrightarrow 
\left\{
\begin{aligned}
&M\wedge dU+dU\wedge M=0,\\
&dM=M\wedge M.
\end{aligned}
\right.
\end{gather}
Part 1 of the proposition follows from the equation $M\wedge dU+dU\wedge M=0$. For Part 2 we write
\begin{multline*}
d\left(\tD+[\tGamma,dU]\right)=\left(\tD+[\tGamma,dU]\right)\wedge\left(\tD+[\tGamma,dU]\right)\Leftrightarrow\\
\Leftrightarrow d\tD+d[\tGamma,dU]=\tD\wedge [\tGamma,dU]+[\tGamma,dU]\wedge\tD+[\tGamma,dU]\wedge [\tGamma,dU],
\end{multline*}
and it remains to note that the diagonal parts of the matrices $[\tGamma,dU]$ and $[\tGamma,dU]\wedge[\tGamma,dU]$ are equal to zero. The proposition is proved.
\end{proof}

\smallskip

Let $\Gamma^i_{jk}$ be the Christoffel symbols of the connection $\nabla$ in canonical coordinates, $\tGamma=(\tgamma^i_j)$ and $\tD=\diag(\tD_1,\ldots,\tD_N)$, where $\tD_i$ are one-forms $\tD_i=\sum_j\tD_{ij}du^j$. Note that $\tD_{ij}=-\Gamma^i_{ji}$. Proposition \ref{proposition:properties of tGamma} together with the fact $\sum_k\frac{\d\tPsi}{\d u^k}=0$ implies that
\begin{align}
&\Gamma^i_{jk}=0,&&k\ne i\ne j\ne k,\label{eq:Gamma in terms of tgamma1}\\
&\Gamma^i_{ij}=\Gamma^i_{ji}=-\Gamma^i_{jj}=\tgamma^i_j,&& i\ne j,\label{eq:Gamma in terms of tgamma2}\\
&\Gamma^i_{ii}=-\sum_{k\ne i}\tgamma^i_k,\label{eq:Gamma in terms of tgamma3}
\end{align}
and that the functions $\tgamma^i_j$ satisfy the following system:
\begin{align}
\frac{\d\tgamma^i_j}{\d u^k}=&-\tgamma^i_j\tgamma^i_k+\tgamma^i_j\tgamma^j_k+\tgamma^i_k\tgamma^k_j, && i\ne k\ne j\ne i,\label{eq:Darboux-Egoroff for tgamma,1}\\
\sum_k\frac{\d\tgamma^i_j}{\d u^k}=&0,&& i\ne j.\label{eq:Darboux-Egoroff for tgamma,2}
\end{align}  

\smallskip

Since $d\tD=0$, there exists a nondegenerate diagonal matrix $H=\diag(H_1,\ldots,H_N)$ satisfying 
$$
dH\cdot H^{-1}=-\tD.
$$
The functions $H_i$ are defined by this property uniquely up to rescalings $H_i\mapsto\lambda_i H_i$, $\lambda_i\in\mbC^*$. Define a metric $g=\sum_i g_i(d u^i)^2$ on our flat F-manifold by $g_i:=H_i^2$. It is clearly compatible with the product $\circ$. If our flat F-manifold is Frobenius, then there exist numbers $\lambda_i\in\mbC^*$ such that the metric $\sum_i \lambda_i g_i(d u^i)^2$ coincides with the metric $\eta$.\\

Define matrices $\Psi$ and $\Gamma$ by
$$
\Psi:=H\tPsi,\qquad\Gamma=(\gamma^i_j):=H\tGamma H^{-1}.
$$
Let us call the coefficients $\gamma^i_j$ the {\it rotation coefficients}. 

\smallskip

\begin{proposition}
We have
\begin{gather}\label{eq:formulas for dPsi and dGammadU}
d\Psi=[\Gamma,dU]\Psi,\qquad d[\Gamma,dU]=[\Gamma,dU]\wedge[\Gamma,dU].
\end{gather}
\end{proposition}
\begin{proof}
We compute
\begin{align*}
d\Psi\cdot\Psi^{-1}=&d\left(H\tPsi\right)\tPsi^{-1}H^{-1}=dH\cdot H^{-1}+\tD+H[\tGamma,dU]H^{-1}=[\Gamma,dU],\\
d[\Gamma,dU]=&d\left(H[\tGamma,dU]H^{-1}\right)=\\
=&-\tD\wedge[\Gamma,dU]+\left(\tD\wedge[\Gamma,dU]+[\Gamma,dU]\wedge\tD+[\Gamma,dU]\wedge[\Gamma,dU]\right)-[\Gamma,dU]\wedge\tD=\\
=&[\Gamma,dU]\wedge[\Gamma,dU].
\end{align*}
\end{proof}

\smallskip

Note that the matrix equation $d[\Gamma,dU]=[\Gamma,dU]\wedge[\Gamma,dU]$ is equivalent to the system 
\begin{gather*}
\left\{\begin{aligned}
&\frac{\d\gamma^i_j}{\d u^k}=\gamma^i_k\gamma^k_j,\quad k\ne i\ne j\ne k,\\
&\sum_{k=1}^N\frac{\d \gamma^i_j}{\d u^k}=0,
\end{aligned}\right.
\end{gather*}
which is the classical Darboux--Egorov system. In the case of a Frobenius manifold, the coefficients $\gamma^i_j$ are the rotation coefficients of the metric and satisfy the additional symmetry property~$\gamma^i_j=\gamma^j_i$.\\

Note also that we have the system
$$
\frac{\d H_i}{\d u^j}=
\begin{cases}
\gamma^i_j H_j,&\text{if $i\ne j$},\\
-\sum_{k\ne i}\gamma^i_k H_k,&\text{if $i=j$}.
\end{cases}
$$
Introducing a column-vector $\oH:=(H_1,\ldots,H_N)$, this system can be equivalently written as
\begin{gather}\label{eq:equation for oH}
d\oH=[\Gamma,dU]\oH.
\end{gather}

\smallskip 

\begin{proposition}\label{proposition:matrices R_k}
1. There exists a sequence of matrices $R_0=\Id,R_1,R_2,\ldots$ satisfying the equations
\begin{gather}\label{eq:equations for R_k}
d R_{k-1}+R_{k-1}[\Gamma,dU]=[R_k,dU],\quad k\ge 1.
\end{gather}
2. The matrices $R_i$ are determined uniquely up to a transformation
\begin{gather}\label{eq:ambiguity of R_i}
\Id+\sum_{i\ge 1}R_i z^i\mapsto \left(\Id+\sum_{i\ge 1}D_iz^i\right)\left(\Id+\sum_{i\ge 1}R_i z^i\right),
\end{gather}
where $D_i$, $i\ge 1$, are arbitrary diagonal matrices with constant entries.
\end{proposition}
\begin{proof}
1. The matrices $R_i$ can be recursively constructed in the following way. Suppose that the matrices $R_0=\Id,R_1,\ldots,R_m$, $m\ge 0$, are already constructed. We define the nondiagonal entries of~$R_{m+1}$ by 
\begin{gather}\label{eq:non-diagonal part of R_m}
(R_{m+1})^i_j:=(R_m)^i_i\gamma^i_j-\frac{\d(R_m)^i_j}{\d u^i},\quad i\ne j,
\end{gather}
and then determine the diagonal entries by the equation
\begin{gather}\label{eq:diagonal part of R_m}
d(R_{m+1})^i_i=-\sum_{j\ne i}(R_{m+1})^i_j\gamma^j_i(du^i-du^j),
\end{gather}
where the integration constants can be arbitrary. Let us check that this procedure is well defined and gives a solution of equations~\eqref{eq:equations for R_k}.\\

Suppose that $n\ge 0$ steps of our procedure are well defined and produce matrices $R_0=\Id,R_1,\ldots,R_n$ satisfying equations~\eqref{eq:equations for R_k} with $k\le n$. Let us first check that
$$
d[R_n,dU]=[R_n,dU]\wedge[\Gamma,dU].
$$
For $n=0$ this is trivial and for $n\ge 1$ we compute
\begin{align}
d[R_n,dU]=&d\left(dR_{n-1}+R_{n-1}[\Gamma,dU]\right)=dR_{n-1}\wedge[\Gamma,dU]+R_{n-1}d[\Gamma,dU]=\label{eq:computation with R_n}\\
=&dR_{n-1}\wedge[\Gamma,dU]+R_{n-1}[\Gamma,dU]\wedge[\Gamma,dU]=[R_n,dU]\wedge[\Gamma,dU].\notag
\end{align}

\smallskip

Then note that equation~\eqref{eq:diagonal part of R_m} with $m=n-1$ implies that the diagonal part of the matrix $dR_n+R_n[\Gamma,dU]$ is equal to zero. Moreover, we have
\begin{gather*}
(dR_n+R_n[\Gamma,dU])\wedge dU+dU\wedge(dR_n+R_n[\Gamma,dU])=d[R_n,dU]-[R_n,dU]\wedge[\Gamma,dU]=0,
\end{gather*}
and, therefore, $dR_n+R_n[\Gamma,dU]=[R,dU]$ for some matrix $R$ whose nondiagonal entries are given exactly by formula~\eqref{eq:non-diagonal part of R_m}, $R^i_j:=(R_n)^i_i\gamma^i_j-\frac{\d(R_n)^i_j}{\d u^i}$, $i\ne j$. In order to check that the diagonal part of $R_{n+1}$ can be defined by equation~\eqref{eq:diagonal part of R_m} with $m=n$, we have to check that
\begin{gather*}
d(R[\Gamma,dU])^\diag=0,
\end{gather*}
where $(\cdot)^\diag$ denotes the diagonal part of a matrix. We compute
$$
d(R[\Gamma,dU])^\diag=\left((dR+R[\Gamma,dU])\wedge[\Gamma,dU]\right)^\diag,
$$
and it remains to check that the expression
$$
(dR+R[\Gamma,dU])\wedge dU+dU\wedge(dR+R[\Gamma,dU])=d[R,dU]-[R,dU]\wedge[\Gamma,dU]
$$
is equal to zero, which is true by the same computation as in~\eqref{eq:computation with R_n}. This completes the proof of Part 1 of the proposition.\\

2. Clearly, transformations~\eqref{eq:ambiguity of R_i} preserve the space of solutions of equations~\eqref{eq:equations for R_k}. Suppose that a sequence of matrices $R_0=\Id,R_1,\ldots$ satisfies equations~\eqref{eq:equations for R_k}. For a fixed~$k$, equation~\eqref{eq:equations for R_k} determines the nondiagonal entries of the matrix $R_k$ in terms of the matrix~$R_{k-1}$ and this gives formula~\eqref{eq:non-diagonal part of R_m}. Since $[R_k,dU]^\diag=0$, equation~\eqref{eq:equations for R_k} determines the differential of the diagonal part of $R_{k-1}$ in terms of the nondiagonal part of $R_{k-1}$. This gives formula~\eqref{eq:diagonal part of R_m}. So all solutions of equations~\eqref{eq:equations for R_k} are given by the procedure described in the proof of the first part of the proposition. At each step of this procedure, the integration constants for the diagonal entries of $R_i$ are totally ambiguous. It is easy to check by induction that, fixing some choice of integration constants, any other choice can be obtained by a transformation of the form~\eqref{eq:ambiguity of R_i}. This completes the proof of the proposition.
\end{proof}

\smallskip

Note that equation~\eqref{eq:equations for R_k} for~$k=1$ implies that $R_1-\Gamma$ is a diagonal matrix.\\

Note also that, introducing the generating series $R(z):=\Id+\sum_{i\ge i}R_i z^i$, the system of equations~\eqref{eq:equations for R_k} can be equivalently written as
\begin{gather}\label{eq:equation for R(z)}
z(dR(z)+R(z)[\Gamma,dU])=[R(z),dU].
\end{gather}

\smallskip

In order to explain the meaning of relations \eqref{eq:equations for R_k}, let us consider the system $\nabla_{z^{-1}}\txi=0$ for $1$-forms $\txi=\sum_{i=1}^N\txi_i(u^*,z) du^i$ depending on~$z$ that are covariantly constant with respect to the connection $\nabla_{z^{-1}}=\nabla+\frac{1}{z}\circ$. In canonical coordinates, this system reads
\begin{align*}
\frac{\d\txi_i}{\d u^j}=&\Gamma^i_{ji}\txi_i+\Gamma^j_{ji}\txi_j,\quad j\ne i,\\
\frac{\d\txi_i}{\d u^i}=&\sum_{l\ne i}\left(-\Gamma^i_{il}\txi_i+\Gamma^l_{ii}\txi_l\right)+\frac{1}{z}\txi_i.
\end{align*}
Let us rewrite the above system for the unknown functions $\xi_i$ defined as $\txi_i=H_i\xi_i$. We obtain the system
\begin{align*}
\frac{\d\xi_i}{\d u^j}=&\gamma^j_i\xi_j,\quad j\ne i\\
\frac{\d\xi_i}{\d u^i}=&-\sum_{l\ne i}\gamma^l_i\xi_l+\frac{1}{z}\txi_i,
\end{align*}
which, introducing a row-vector $\oxi:=(\xi_1,\ldots,\xi_N)$, is equivalent to the equation
\begin{gather}\label{eq:equation for xi}
d\oxi=\oxi[dU,\Gamma]+\frac{1}{z}\oxi dU.
\end{gather}
For $\gamma^i_j=0$ (trivial flat F-manifold), a fundamental matrix of solutions of this equation is $\Xi^0=e^{U/z}$. It is straightforward to check that, looking for a fundamental matrix $\Xi$ of solutions of equation~\eqref{eq:equation for xi} in the form
$$
\Xi=e^{U/z}\left({\rm Id}+\sum_{k\ge 1}R_k z^k\right),
$$ 
one obtains system~\eqref{eq:equations for R_k}. We will see that the formal series in the brackets can be interpreted as an element of a group acting on the space of flat F-manifolds.

\smallskip

\subsection{Flat F-manifolds and Riemannian F-manifolds}

In the previous section, we saw that, around a semisimple point, a flat F-manifold possesses a metric. In this section, we show that there is a correspondence between semisimple (at any point) flat F-manifolds and Riemannian F-manifolds.\\

Let us first recall the local description of flat F-manifolds around a semisimple point in canonical coordinates.

\smallskip

\begin{theorem}[\cite{AL15}]\label{theorem:local description of semisimple flat F-manifolds}
Let $\tgamma^i_j$, $i\ne j$, be a solution of equations~\eqref{eq:Darboux-Egoroff for tgamma,1},~\eqref{eq:Darboux-Egoroff for tgamma,2}. Then the connection $\nabla$, given by equations~\eqref{eq:Gamma in terms of tgamma1}--\eqref{eq:Gamma in terms of tgamma3}, the structure constants, given by $c^i_{jk}:=\delta^i_j\delta^i_k$, and the vector field $e:=\sum_i\frac{\d}{\d u^i}$ define a semisimple flat F-manifold structure, where the coordinates~$u^i$ are canonical. Moreover, any flat F-manifold around a semisimple point can be obtained in this way.
\end{theorem}
\begin{proof}
The fact that any flat F-manifold around a semisimple point can be obtained in this way was proved in the previous section. The converse statement is a direct computation (see \cite{AL15} for details).
\end{proof}

\smallskip

In the previous section, we associated to any flat F-manifold around a semisimple point a metric $g=\sum_i g_i(d u^i)^2$ satisfying the condition 
\begin{gather}\label{eq:defining property for gi}
\frac{1}{g_i}\frac{\d g_i}{\d u^j}=2\Gamma^i_{ij},\quad 1\le i,j\le N. 
\end{gather}
This condition determines a metric uniquely up to rescalings $g_i\mapsto\lambda_ig_i$, $\lambda_i\in\mbC^*$. Thus, we actually get a family of metrics parameterized by a vector $(\lambda_1,\ldots,\lambda_N)\in(\mbC^*)^N$. This family can be described in a more invariant way, without going to the canonical coordinates.

\smallskip

\begin{proposition}\label{proposition:two descriptions of a metric}
The family of metrics $g$ on a flat F-manifold $M$ around a semisimple point given, in canonical coordinates, by~\eqref{eq:defining property for gi} coincides with the family of metrics $g$ compatible with the product $\circ$ and satisfying the condition
\begin{gather}\label{eq:invariant definition of a metric}
\nabla_kg_{ij}=\frac{1}{2}\sum_l\left(c^l_{ik}(d\theta)_{lj}+c^l_{jk}(d\theta)_{li}\right),\quad 1\le i,j,k\le N,
\end{gather}
where $X,Y,Z$ are local vector fields on $M$ and $\theta$ is the counit, $\theta(\cdot)=g(e,\cdot)$.
\end{proposition}
\begin{proof}
Suppose that a metric $g=\sum_i g_i(du^i)^2$ satisfies condition~\eqref{eq:defining property for gi}. Note that $\theta=\sum_i g_i du^i$. Then the proof of property~\eqref{eq:invariant definition of a metric} becomes a simple direct computation based on the expression of the Christoffel symbols $\Gamma^i_{jk}$ in terms of the functions~$g_i$.\\

Suppose now that a metric $g=\frac{1}{2}\sum_{i,j} g_{ij}du^idu^j$ is compatible with the product $\circ$ and satisfies condition~\eqref{eq:invariant definition of a metric}. The compatibility with the product~$\circ$ immediately implies that $g_{ij}=0$ for $i\ne j$. For $i=j$, the right-hand side of~\eqref{eq:invariant definition of a metric} is zero, while the left-hand side is $\frac{\d g_{ii}}{\d u^k}-2\Gamma^i_{ik}g_{ii}$, which gives equation~\eqref{eq:defining property for gi} for $g_{ii}=g_i$.
\end{proof}

\smallskip

\begin{definition}\cite{HM99}
An F-manifold $(M,\circ,e)$ is the datum of an analytic manifold $M$, a commutative associative algebra structure $(T_pM,\circ)$ on each tangent space analytically depending on the point $p\in M$, and a unit vector field $e$ such that the condition
\begin{gather}\label{eq:HM condition,invariant}
\Lie_{X\circ Y}(\circ)=X\circ\Lie_Y(\circ)+Y\circ\Lie_X(\circ)
\end{gather} 
is satisfied, for any local vector fields $X,Y$ on $M$. The above condition is called the Hertling--Manin condition.
\end{definition}

\smallskip

If $t^1,\ldots,t^N$ are some coordinates on $M$, then condition~\eqref{eq:HM condition,invariant} is equivalent to the following condition for the structure constants $c^i_{jk}$ of the multiplication~$\circ$:
\begin{gather*}
\sum_{s=1}^N\left(\frac{\d c^k_{jl}}{\d t^s}c^s_{im}+\frac{\d c^s_{im}}{\d t^j}c^k_{sl}-\frac{\d c^k_{im}}{\d t^s}c^s_{jl}-\frac{\d c^s_{jl}}{\d t^i}c^k_{sm}-\frac{\d c^s_{jl}}{\d t^m}c^k_{si}+\frac{\d c^s_{mi}}{\d t^l}c^k_{js}\right)=0,\quad 1\le i,j,k,l,m\le N.
\end{gather*}

\smallskip

In the remaining part of this section, we will focus on the semisimple case describing the relation between flat F-manifolds and a special class of F-manifolds called Riemannian F-manifolds.

\smallskip

\begin{definition}
A semisimple (pseudo-)Riemannian F-manifold $(M,g,\circ,e)$ is the datum of a semisimple F-manifold $(M,\circ,e)$ equipped with a metric~$g$ compatible with the product $\circ$ and such that
\begin{gather}\label{eq:condition for Riemannian F-manifold 2}
R(Y,Z)(X\circ W)+R(X,Y)(Z\circ W)+R(Z,X)(Y\circ W)=0,
\end{gather}
where $R$ is the curvature operator for the Levi-Civita connection $\tnabla$ associated to $g$ and $X,Y,W,Z$ are local vector fields on $M$. If also the condition 
\begin{gather*}
\Lie_e g=0
\end{gather*}
is satisfied, then the manifold is called a Riemannian F-manifold with Killing unit vector field.
\end{definition}

\smallskip

\begin{remark}
Using semisimplicity, it is easy to check that condition \eqref{eq:condition for Riemannian F-manifold 2} can be replaced by the equivalent condition
\begin{gather}\label{eq:condition for Riemannian F-manifold}
Z\circ R(W,Y)(X)+W\circ R(Y,Z)(X)+Y\circ R(Z,W)(X)=0.
\end{gather}
\end{remark} 

\smallskip

\begin{remark} The notion of a Riemannian F-manifold appears also in \cite{LPR09} and \cite{DS11}. In both cases, the definition involves some extra conditions. In \cite{LPR09}, the connection and the product are required to satisfy the condition $\tnabla_l c^i_{jk}=\nabla_j c^i_{lk}$, while in \cite{DS11} the counit $\theta$ is required to be closed.
\end{remark} 

\smallskip

We have the following theorem. 

\smallskip

\begin{theorem}
1. Consider a flat F-manifold $(M,\nabla,\circ,e)$ around a semisimple point and a metric~$g$ compatible with the product $\circ$ and satisfying~\eqref{eq:invariant definition of a metric}. Then the tuple $(M,g,\circ,e)$ is a Riemannian F-manifold with Killing unit vector field.\\
2. Consider a Riemannian F-manifold $(M,g,\circ,e)$ with Killing unit vector field around a semisimple point. Then there exists a unique torsionless connection satisfying~\eqref{eq:invariant definition of a metric}. With this connection the tuple $(M,\nabla,\circ,e)$ is a flat F-manifold.
\end{theorem}
\begin{proof}
Consider a Riemannian F-manifold $(M,g,\circ,e)$ around a semisimple point. Let $u^1,\ldots,u^N$ be the canonical coordinates. The compatibility with the product~$\circ$ implies that the metric $g$ has the form $g=\sum_i g_i(du^i)^2$ and, therefore, the Christoffel symbols $\tGamma^i_{jk}$ of the connection $\tnabla$ are given by
\begin{align*}
\tGamma^i_{jk}&=0,&& i\ne j\ne k\ne i,\\
\tGamma^i_{ij}&=\frac{1}{2g_i}\frac{\d g_i}{\d u^j},&& 1\le i,j\le N,\\
\tGamma^j_{ii}&=-\frac{1}{2g_j}\frac{\d g_i}{\d u^j},&& i\ne j.
\end{align*}

\smallskip

Equation \eqref{eq:condition for Riemannian F-manifold} reads
\begin{gather}\label{eq:Riemann tensor condition,1}
R^m_{lij}\delta^m_{k}+R^m_{lki}\delta^m_{j}+R^m_{ljk}\delta^m_{i}=0,\quad 1\le i,j,k,l,m\le N.
\end{gather}
Note that, because of the skewsymmetry of the Riemann tensor with respect to the second and the third lower indices, the condition above is satisfied if some of the indices $i,j,k$ coincide. Note also that condition~\eqref{eq:Riemann tensor condition,1} is trivially satisfied when $m$ is distinct from $i,j,k$. Therefore, condition~\eqref{eq:Riemann tensor condition,1} is nontrivial only if the indices $i,j,k$ are distinct and $m$ coincides with one of them, which gives the system
$$
R^k_{lij}=0,\quad k\ne i\ne j\ne k.
$$

\smallskip

The vanishing of the Christoffel symbols $\tGamma^i_{jk}$ with distinct indices $i,j,k$ implies that $R^k_{lij}$ vanishes if all the indices $i,j,k,l$ are distinct. Therefore, we come to the system
\begin{gather*}
R^k_{kij}=0,\qquad R^k_{iij}=0,\qquad k\ne i\ne j\ne k.
\end{gather*}
The skewsymmetry of the tensor $R_{ijkl}=g_i R^i_{jkl}$ with the respect to the first two lower indices implies that $R^k_{kij}=0$. Thus, condition~\eqref{eq:condition for Riemannian F-manifold} is equivalent to the system
$$
R^k_{iij}=0,\qquad k\ne i\ne j\ne k.
$$
Since
\begin{gather*}
R^k_{iij}=-\frac{\d\tGamma^k_{ii}}{\d u^j}+\tGamma^k_{ii}\tGamma^i_{ij}-\tGamma^k_{jj}\tGamma^j_{ii}-\tGamma^k_{kj}\tGamma^k_{ii},\quad i\ne j\ne k\ne i,
\end{gather*}
we conclude that the datum of a Riemannian F-manifold in canonical coordinates is equivalent to a diagonal metric $g=\sum_i g_i(du^i)^2$ satisfying the Darboux-Tsarev system
\begin{gather}\label{eq:Darboux-Tsarev system}
-\frac{\d\tGamma^k_{ii}}{\d u^j}+\tGamma^k_{ii}\tGamma^i_{ij}-\tGamma^k_{jj}\tGamma^j_{ii}-\tGamma^k_{kj}\tGamma^k_{ii}=0,\quad i\ne j\ne k\ne i.
\end{gather}

\smallskip

Let us prove Part 1 of the theorem. Consider a flat F-manifold $(M,\nabla,\circ,e)$ and a metric~$g$ compatible with the product $\circ$ and satisfying~\eqref{eq:condition for Riemannian F-manifold}. By Proposition~\ref{proposition:two descriptions of a metric}, in canonical coordinates we have $g=\sum_i g_i(du^i)^2$, where the $g_i$'s satisfy~\eqref{eq:defining property for gi}. Note that
\begin{align*}
&\tGamma^i_{ij}=\Gamma^i_{ij}=\tgamma^i_j,\qquad\tGamma^j_{ii}=-\frac{g_i}{g_j}\tGamma^i_{ij}=-\frac{g_i}{g_j}\tgamma^i_j,&& i\ne j,\\
&\tGamma^i_{ii}=\Gamma^i_{ii}=-\sum_{j\ne i}\tgamma^i_j.
\end{align*}
Using these formulas, it is now easy to check that system~\eqref{eq:Darboux-Tsarev system} follows from system~\eqref{eq:Darboux-Egoroff for tgamma,1}. Since $\sum_j\frac{1}{g_i}\frac{\d g_i}{\d u^j}=\sum_j\Gamma^i_{ij}=0$, the unit vector field is Killing.\\

Let us prove Part 2 of the theorem. Consider a Riemannian F-manifold $(M,g,\circ,e)$ with Killing unit vector field. Let $\tgamma^i_j:=\tGamma^i_{ij}$, $i\ne j$. It is easy to check that the Darboux--Tsarev system~\eqref{eq:Darboux-Tsarev system} implies that the functions~$\tgamma^i_j$ satisfy system~\eqref{eq:Darboux-Egoroff for tgamma,1}. The fact that the unit vector field $e$ is Killing implies that 
$$
\sum_k\frac{\d g_i}{\d u^k}=0\Rightarrow\sum_k\frac{\d\tGamma^i_{ij}}{\d u^k}=0.
$$
Therefore, equation~\eqref{eq:Darboux-Egoroff for tgamma,2} is also satisfied and, thus, by Theorem~\ref{theorem:local description of semisimple flat F-manifolds}, the connection~$\nabla$ given by equations~\eqref{eq:Gamma in terms of tgamma1}--\eqref{eq:Gamma in terms of tgamma3} defines a flat F-manifold $(M,\nabla,\circ,e)$. The fact that $\nabla$ satisfies condition~\eqref{eq:invariant definition of a metric} or, equivalently, condition~\eqref{eq:defining property for gi} is obvious.\\

It remains to check that condition~\eqref{eq:invariant definition of a metric} determines a connection~$\nabla$ uniquely. Denote the tensor on the right-hand side of this equation by~$\Delta_{kij}$. Let us write three equations~\eqref{eq:invariant definition of a metric} corresponding to the cyclic permutations of the indices $i,j,k$:
\begin{align*}
&\frac{\d g_{ij}}{\d u^k}-\Gamma_{kij}-\Gamma_{kji}=\Delta_{kij},\\
&\frac{\d g_{ki}}{\d u^j}-\Gamma_{jki}-\Gamma_{jik}=\Delta_{jki},\\
&\frac{\d g_{jk}}{\d u^i}-\Gamma_{ijk}-\Gamma_{ikj}=\Delta_{ijk},
\end{align*} 
where $\Gamma_{kij}:=\sum_s\Gamma^s_{ki}g_{sj}$. Summing the first and the third equations and subtracting the second one, we get
$$
\Gamma_{kij}=\frac{1}{2}\left(\frac{\d g_{ij}}{\d u^k}-\frac{\d g_{ik}}{\d u^j}+\frac{\d g_{jk}}{\d u^i}-\Delta_{kij}+\Delta_{jik}-\Delta_{ijk}\right),
$$
which completes the proof of the theorem.
\end{proof}

\smallskip

\subsection{Homogeneous flat F-manifolds}

In this section, we will prove that the metric, constructed in Section~\ref{subsection:metric for a flat F-manifold}, in the case of homogeneous flat F-manifolds satisfies an additional homogeneity property. We will also show that an additional homogeneity property allows to fix uniquely the choice of $R$-matrices from Proposition~\ref{proposition:matrices R_k}.\\

Consider a flat F-manifold $(M,\nabla,\circ,e)$ around a semisimple point. Let $t^1,\ldots,t^N$ be flat coordinates. Suppose that our flat F-manifold is
 homogeneous with an Euler vector field
$$
E=\sum_{\alpha=1}^N((1-q_\alpha) t^\alpha+r^\alpha)\frac{\d}{\d t^\alpha},\quad q_\alpha,r^\alpha\in\mbC.
$$
In canonical coordinates, we have $E=\sum_i(u^i+a^i)\frac{\d}{\d u^i}$ for some $a^i\in\mbC$. After an appropriate shift of the coordinates, we can assume that
$$
E=\sum_i u^i\frac{\d}{\d u^i}.
$$
Consider the matrices $\tPsi,\tGamma=(\tgamma^i_j),\Gamma=(\gamma^i_j)$, the diagonal matrix $\tD$ (consisting of one-forms) and the metric $g=\sum_i g_i(du^i)^2$, $g_i=H_i^2$, constructed in Section~\ref{subsection:metric for a flat F-manifold}.

\smallskip

\begin{proposition}[\cite{Lor14}]\label{proposition:homogeneity of H and gamma}
1. The diagonal matrix $i_E\tD$ is constant, $i_E\tD=-\diag(\delta_1,\ldots,\delta_N)$, $\delta_i\in\mbC$.\\
2. We have 
\begin{gather}\label{eq:homogeneity for H and Gamma}
\sum_j u^j\frac{\d H_i}{\d u^j}=\delta_i H_i,\qquad \sum_k u^k\frac{\d\gamma^i_j}{\d u^k}=(\delta_i-\delta_j-1)\gamma^i_j.
\end{gather}
\end{proposition}
\begin{proof}
Define a diagonal matrix $Q$ by $Q:=\diag(q_1,\ldots,q_N)$. In \cite{BB19} the authors introduced a family of connections $\tnabla^\lambda$ on $M\times\mbC^*$, depending on a complex parameter $\lambda$, by
\begin{align}
\tnabla^\lambda_X Y:=&\nabla_X Y+z X\circ Y,\notag\\
\tnabla^\lambda_{\frac{\d}{\d z}}Y:=&\frac{\d Y}{\d z}+E\circ Y+\frac{\lambda-Q}{z}Y,\label{eq:connection tnabla,2}\\
\tnabla^\lambda_X\frac{\d}{\d z}=\tnabla^\lambda_{\frac{\d}{\d z}}\frac{\d}{\d z}:=&0,\notag
\end{align}
where $z$ is the coordinate on $\mbC^*$ and $X,Y$ are local vector fields on $M\times\mbC^*$ having zero component along $\mbC^*$. The authors proved that the connection $\tnabla^\lambda$ is flat for any value of $\lambda$. Let us show how to derive the proposition from the flatness of $\tnabla^\lambda$.\\

Note that equation~\eqref{eq:connection tnabla,2} can be rewritten as
$$
\tnabla^\lambda_{\frac{\d}{\d z}}Y=\frac{\d Y}{\d z}+E\circ Y+\frac{1}{z}\nabla_Y E+\frac{\lambda-1}{z}Y.
$$
In canonical coordinates, we have
$$
\nabla_Y E=\sum_i Y^i\frac{\d}{\d u^i}+\left(\tPsi\sum_iY^i\frac{\d(\tPsi^{-1})}{\d u^i}\right)E,
$$
and it is easy to check that
$$
\left(\tPsi\sum_iY^i\frac{\d(\tPsi^{-1})}{\d u^i}\right)E=\left(\tPsi\sum_iu^i\frac{\d(\tPsi^{-1})}{\d u^i}\right)Y.
$$
For an $N\times N$ matrix $A=(A^i_j)$ denote by $\Ext(A)$ the $(N+1)\times(N+1)$ matrix defined by
$$
\Ext(A)^i_j:=
\begin{cases}
A^i_j,&\text{if $1\le i,j\le N$},\\
0,&\text{otherwise}.
\end{cases}
$$
We see that in canonical coordinates the connection $\tnabla^\lambda$ is given by
$$
\tnabla^\lambda=d+\Ext\left(-d\tPsi\cdot\tPsi^{-1}+zdU+\left(U+\frac{\lambda}{z}-\frac{1}{z}\left(\sum_iu^i\frac{\d\tPsi}{\d u^i}\right)\tPsi^{-1}\right)dz\right).
$$

\smallskip

The flatness of the connection $\tnabla^\lambda$ for any value of $\lambda$ is equivalent to equation~\eqref{eq:equations for M} together with the equations
\begin{align}
dB+[B,M]=&0,\label{eq:B and M}\\
[U,M]+[B,dU]=&0,\notag
\end{align}
where $M:=d\tPsi\cdot\tPsi^{-1}$ and $B:=i_E M=\left(\sum_iu^i\frac{\d\tPsi}{\d u^i}\right)\tPsi^{-1}$. Equation~\eqref{eq:B and M} implies that
$$
dB+i_E(M\wedge M)=0\stackrel{\text{by eq.~\eqref{eq:equations for M}}}{\Rightarrow}dB+i_E(dM)=0.
$$
The diagonal part of $dM$ is $d\tD=0$. Therefore, the diagonal part of $dB$, which is $d(i_E\tD)$, vanishes. This proves the first part of the proposition.\\

Denote $\Delta:=-i_E\tD$. Since $H$ is defined by $dH\cdot H^{-1}=-\tD$, we get $\sum_i u^i\frac{\d H}{\d u^i}=\Delta H$, which is exactly the first equation in~\eqref{eq:homogeneity for H and Gamma}. We see that $B=-\Delta+[\tGamma,U]$. Therefore, equation~\eqref{eq:B and M} implies that
$$
d[\tGamma,U]+[-\Delta+[\tGamma,U],\tD+[\tGamma,dU]]=0.
$$
Substituting $\tGamma=H^{-1}\Gamma H$, we get
$$
d[\Gamma,U]=[[\Delta,\Gamma],dU]+[[\Gamma,dU],[\Gamma,U]].
$$
Applying the contraction $i_E$ to both sides of this equation, we get
$$
\left[\sum_iu^i\frac{\d\Gamma}{\d u^i},U\right]+[\Gamma,U]=[[\Delta,\Gamma],U]\Rightarrow\sum_iu^i\frac{\d\Gamma}{\d u^i}=[\Delta,\Gamma]-\Gamma,
$$
which is exactly the second equation in~\eqref{eq:homogeneity for H and Gamma}. This completes the proof of the proposition.
\end{proof}

\smallskip

\begin{remark}
If our homogeneous flat F-manifold is a conformal Frobenius manifold of conformal dimension $\delta$, meaning that $\Lie_E\eta=(2-\delta)\eta$, then $\delta_i=-\frac{\delta}{2}$, $1\le i\le N$.
\end{remark}

\smallskip

\begin{proposition}\label{proposition:unique R-matrix in the homogeneous case}
There exists a unique sequence of matrices $R_0=\Id,R_1,R_2,\ldots$ satisfying the differential equation~\eqref{eq:equations for R_k} and the homogeneity condition
\begin{gather}\label{eq:homogeneity for Rk}
\sum_s u^s\frac{\d(R_k)^i_j}{\d u^s}=(\delta_i-\delta_j-k)(R_k)^i_j,\quad k\ge 0.
\end{gather}
\end{proposition}
\begin{proof}
As we know from the proof of Proposition~\ref{proposition:matrices R_k}, without requiring property~\eqref{eq:homogeneity for Rk} a sequence of matrices $R_k$ is recursively determined in the following way. Suppose that matrices $R_0,R_1,\ldots,R_m$, $m\ge 0$, are already constructed. Then the nondiagonal entries of $R_{m+1}$ are given by
\begin{gather}\label{eq:nondiagonal entries of R}
(R_{m+1})^i_j=(R_m)^i_i\gamma^i_j-\frac{\d(R_m)^i_j}{\d u^i},\quad i\ne j.
\end{gather}
The diagonal entries of $R_{m+1}$ are determined by the equation
\begin{gather}\label{eq:diagonal entries of R}
\frac{\d(R_{m+1})^i_i}{\d u^j}=
\begin{cases}
(R_{m+1})^i_j\gamma^j_i,&\text{if $j\ne i$},\\
-\sum_{k\ne i}(R_{m+1})^i_k\gamma^k_i,&\text{if $j=i$},
\end{cases}\quad 1\le i,j\le N,
\end{gather}
uniquely up to constants, which can be arbitrary. Suppose that the matrix $R_m$ satisfies the homogeneity condition~\eqref{eq:homogeneity for Rk}. From formula~\eqref{eq:nondiagonal entries of R} and Proposition~\ref{proposition:homogeneity of H and gamma} it follows that the nondiagonal entries of $R_{m+1}$ also satisfy the homogeneity condition~\eqref{eq:homogeneity for Rk}. Formula~\eqref{eq:diagonal entries of R} implies that $\sum_s u^s\frac{\d}{\d u^s}\frac{\d(R_{m+1})^i_i}{\d u^j}=-(m+2)\frac{\d(R_{m+1})^i_i}{\d u^j}$ for any $i$ and $j$. Thus, unique functions~$(R_{m+1})^i_i$ satisfying equations~\eqref{eq:homogeneity for Rk} and~\eqref{eq:diagonal entries of R} are given by
$$
(R_{m+1})^i_i=-\frac{1}{m+1}\sum_{j\ne i}(u^j-u^i)(R_{m+1})^i_j\gamma^j_i.
$$
The proposition is proved.
\end{proof}

\smallskip

\begin{remark}
In~\cite{AL13}, the authors introduced the notion of a bi-flat F-manifold, which is the datum of two different flat F-manifold structures $(\nabla,\circ,e)$ and $(\nabla^{*},*,E)$ on the same manifold~$M$ intertwined by the following conditions:
\begin{enumerate}
\item $[e,E]=e$, $\Lie_E(\circ)=\circ$;
\item $X*Y=(E\circ)^{-1}\,X\circ Y$ (or $X\circ Y=(e*)^{-1}X*Y$) for all local vector fields $X,Y$ on~$M$; 
\item $(d_{\nabla}-d_{\nabla^{*}})(X\,\circ)=0$ for all local vector fields $X$ on $M$, where $d_{\nabla}$ is the exterior covariant derivative.  
\end{enumerate}
Using the fact that the connection $\nabla^*$ can be expressed in terms of the other data, in~\cite{AL17}, the authors proved that in the semisimple case the structure of a bi-flat F-manifold is equivalent to the datum of a flat F-manifold, equipped with an invertible vector field $E$ satisfying condition~(1) from the list above and also the property $\nabla\nabla E=0$ (see~\cite{KMS18} for the discussion of the regular case). In other words, on the complement of the discriminant (i.e., at the points where $E$ is invertible) any regular homogeneous flat F-manifold is equipped with a bi-flat structure.
\end{remark}

\smallskip

\begin{remark}
The structure $(\nabla,\circ,e,E)$ where $E$ is a linear Euler vector field can be also characterized in terms of a flat meromorphic connection on the bundle $\pi^*TM$ on $\mathbb{P}\times M$ called Saito structure without metric (see \cite{Sab98} for details). In particular, assuming the existence of flat coordinates diagonalizing the matrix $\nabla E$, it turns out that Saito structures without metric are equivalent to homogeneous flat structures.
\end{remark}

\smallskip

\subsubsection{Example: extended $2$-spin theory}\label{subsubsection:example1}

An example of a flat F-manifold is given by the extended $r$-spin theory, constructed in~\cite{JKV01} and then studied in details in~\cite{BCT19,Bur18,BR18}. Let us consider the case $r=2$ and compute the metric, the rotation coefficients, and the $R$-matrices.\\

The vector potential of the flat F-manifold of the extended $2$-spin theory is given by (see, e.g.,~\cite[Section~4.3]{BR18})
$$
F^1(t^1,t^2)=\frac{(t^1)^2}{2},\qquad F^2(t^1,t^2)=t^1t^2-\frac{(t^2)^3}{12}.
$$
The unit is $\frac{\d}{\d t^1}$. The flat F-manifold is homogeneous with the Euler vector field
$$
E=t^1\frac{\d}{\d t^1}+\frac{1}{2}t^2\frac{\d}{\d t^2}.
$$
The set of non-semisimple points coincides with the $t^1$-axis. Consider our flat F-manifold around a point $(0,\tau)$, $\tau\ne 0$. The canonical coordinates satisfy the system of differential equations
$$
\frac{\d u^i}{\d t^\alpha}\frac{\d u^i}{\d t^\beta}=c_{\alpha\beta}^\gamma\frac{\d u^i}{\d t^\gamma},\quad 1\le\alpha,\beta,i\le 2,
$$
from which we find
$$
\left\{
\begin{aligned}
u^1=&t^1,\\
u^2=&t^1-\frac{(t^2)^2}{2},
\end{aligned}
\right.\qquad
\left\{
\begin{aligned}
t^1=&u^1,\\
t^2=&\sqrt{2(u^1-u^2)},
\end{aligned}
\right.\qquad
\tPsi=\begin{pmatrix}
1 & 0\\
1 & -t^2
\end{pmatrix}.
$$
We then compute
$$
d\tPsi\cdot\tPsi^{-1}=\begin{pmatrix}
0 & 0\\
-\frac{dt^2}{t^2} & \frac{dt^2}{t^2}
\end{pmatrix},
\qquad
dU=\begin{pmatrix}
dt^1 & 0\\
0 & dt^1-t^2dt^2
\end{pmatrix},
$$
and find
$$
\tGamma=\begin{pmatrix}
0 & 0\\
-\frac{1}{(t^2)^2} & 0
\end{pmatrix},
\qquad
H=\begin{pmatrix}
\lambda_1 & 0\\
0 & \frac{\lambda_2}{t^2}
\end{pmatrix},\quad\lambda_1,\lambda_2\in\mbC^*.
$$
Let us choose the parameters $\lambda_1,\lambda_2$ such that $H|_{(t^1,t^2)=(0,\tau)}=\Id$. We obtain
$$
H=
\begin{pmatrix}
1 & 0\\
0 & \frac{\tau}{t^2}
\end{pmatrix},\qquad
\Gamma=
\begin{pmatrix}
0 & 0\\
-\frac{\tau}{(t^2)^3} & 0
\end{pmatrix},\qquad
\Psi=
\begin{pmatrix}
1 & 0\\
\frac{\tau}{t^2} & -\tau
\end{pmatrix}.
$$
We see that 
$$
\delta_1=0,\qquad\delta_2=-\frac{1}{2}.
$$
It is easy to check that a unique sequence of $R$-matrices given by Proposition~\ref{proposition:unique R-matrix in the homogeneous case} is the following:
$$
R_m=\begin{pmatrix}
0 & 0\\
(-1)^m\tau\frac{(2m-1)!!}{(t^2)^{2m+1}} & 0
\end{pmatrix},\quad m\ge 1.
$$
In particular, we have
$$
\Psi\big|_{(t^1,t^2)=(0,\tau)}=
\begin{pmatrix}
1 & 0\\
1 & -\tau
\end{pmatrix},\qquad
R_m\big|_{(t^1,t^2)=(0,\tau)}=\begin{pmatrix}
0 & 0\\
(-1)^m\frac{(2m-1)!!}{\tau^{2m}} & 0
\end{pmatrix},\quad m\ge 1.
$$
\\

%%%%%%%%%%%%%%%%%%%%%%%%%%%%%%%%%%%%%%%%%%%%%%%%%%%%%%%%%%%%
%%%%%%%%%%%%%%%%%%%%%%%%%%%%%%%%%%%%%%%%%%%%%%%%%%%%%%%%%%%%

\section{Givental-type theory for flat F-manifolds}\label{section:Givental theory for flat F-manifolds}

In this section, we recall the notion of a calibrated flat F-manifold, introduced in~\cite{BB19}, and interpret such an object as a certain infinite dimensional submanifold in the space $\mbC^N[[z,z^{-1}]]$. This generalizes a similar result of A. Givental~\cite{Giv04} about Frobenius manifolds and allows to introduce a group action on the space of calibrated flat F-manifolds. We then prove that this group action, combined with linear changes of coordinates, is transitive on the space of semisimple calibrated flat F-manifolds.\\

Let us fix a point $\ot_\orig=(t^1_\orig,\ldots,t^N_\orig)\in\mbC^N$ and denote by $\mcR^{\ot_{\orig}}$ the ring of formal power series in the shifted variables $t^\alpha-t^\alpha_\orig$. In this section, we consider flat F-manifolds defined on a formal neighbourhood of $\ot_\orig$, which means that the functions describing the structure of our flat F-manifolds belong to the ring $\mcR^{\ot_{\orig}}$. 

\smallskip

\subsection{Calibrated flat F-manifolds and descendant vector potentials}\label{subsection:calibrated flat F-manifold}

Consider a flat F-manifold structure given by a vector potential $\oF=(F^1,\ldots,F^N)$, $F^\alpha\in\mcR^{\ot_{\orig}}$, and a unit $A^\alpha\frac{\d}{\d t^\alpha}$, $A^\alpha\in\mbC$. We will often denote the unit by $\frac{\d}{\d t^\un}$. A {\it calibration} of our flat F-manifold is a collection of functions~$\Omega^{\alpha,d}_{\beta,0}\in\mcR^{\ot_{\orig}}$, $1\le\alpha,\beta\le N$, $d\ge -1$, satisfying $\Omega^{\alpha,-1}_{\beta,0}=\delta^\alpha_\beta$ and the property
\begin{gather}\label{eq:property of a calibration}
\frac{\d\Omega^{\alpha,d}_{\beta,0}}{\d t^\gamma}=c^\mu_{\gamma\beta}\Omega^{\alpha,d-1}_{\mu,0}, \quad d\ge 0,
\end{gather}
where $c^\alpha_{\beta\gamma}=\frac{\d^2 F^\alpha}{\d t^\beta\d t^\gamma}$. Introduce $N\times N$ matrices $\Omega^d_0$, $d\ge -1$, by $(\Omega^d_0)^\alpha_\beta:=\Omega^{\alpha,d}_{\beta,0}$. Equation~\eqref{eq:property of a calibration} implies that $c^\alpha_{\beta\gamma}=\frac{\d\Omega^{\alpha,0}_{\beta,0}}{\d t^\gamma}$ and, thus, equation~\eqref{eq:property of a calibration} can be written in the matrix form as
$$
d\Omega^p_0=\Omega^{p-1}_0\cdot d\Omega^0_0,\quad p\ge 0,
$$
where $d(\cdot)$ denotes the full differential. A calibration is determined uniquely up to a transformation
\begin{gather}\label{eq:transformation of calibration}
\left(\Id+\sum_{d\ge 1}\Omega^{d-1}_0 z^d\right)\mapsto G(z)\left(\Id+\sum_{d\ge 1}\Omega^{d-1}_0 z^d\right),\quad G(z)\in\End(\mbC^N)[[z]],\quad G(0)=\Id.
\end{gather}
A flat F-manifold together with a calibration is called a {\it calibrated flat F-manifold}.\\

To our calibrated flat F-manifold one can associate a two-parameter family of matrices $\Omega^p_q$, $p,q\ge 0$, in the following way. Let us introduce matrices $\Omega^0_d=(\Omega^{\alpha,0}_{\beta,d})$, $d\ge 0$, by the equation
\begin{gather}\label{eq:upper-lower relation}
\left(\Id+\sum_{d\ge 1}(-1)^d\Omega^0_{d-1}z^d\right)\left(\Id+\sum_{d\ge 1}\Omega^{d-1}_0 z^d\right)=\Id.
\end{gather}
By definition, we put $\Omega^0_{-1}:=\Id$. The matrices $\Omega^0_p$ satisfy the equation
\begin{gather}\label{eq:TRR with differential}
d\Omega^0_p=d\Omega^0_0\cdot\Omega^0_{p-1},\quad p\ge 0.
\end{gather}
We define matrices $\Omega^p_q=(\Omega^{\alpha,p}_{\beta,q})$, $p,q\ge 0$, by
\begin{gather}\label{eq:definition of Omegapq}
\Omega^p_q:=\sum_{i=0}^q(-1)^{q-i}\Omega^{p+q-i}_0\Omega^0_{i-1}\stackrel{\scriptsize{\text{eq. \eqref{eq:upper-lower relation}}}}{=}\sum_{i=0}^p(-1)^{p-i}\Omega^{i-1}_0\Omega^0_{p+q-i},
\end{gather}

\smallskip

Let us present the construction of the descendant vector potentials associated to our calibrated flat F-manifold. We will use the notation $\Omega^{\alpha,p}_{\un,q}:=A^\beta\Omega^{\alpha,p}_{\beta,q}$. Equation~\eqref{eq:property of a calibration} implies that $\Omega^{\alpha,0}_{\un,0}-t^\alpha$ is a constant, $\Omega^{\alpha,0}_{\un,0}=t^\alpha+c^\alpha$, $\oc=(c^1,\ldots,c^N)\in\mbC^N$. Let
$$
(\Omega^\oc)^{\alpha,p}_{\beta,q}:=\left.\Omega^{\alpha,p}_{\beta,q}\right|_{t^\gamma\mapsto t^\gamma-c^\gamma}\in\mcR^{\ot_\orig+\oc}.
$$
Consider the {\it principal hierarchy} associated to our calibrated flat F-manifold:
\begin{gather}\label{eq:principal hierarchy for F-man}
\frac{\d v^\alpha}{\d t^\beta_d}=\d_x\left(\left.(\Omega^\oc)^{\alpha,0}_{\beta,d}\right|_{t^\gamma=v^\gamma}\right),\quad 1\le\alpha,\beta\le N,\quad d\ge 0.
\end{gather}
The flows of the principal hierarchy pairwise commute. Since $(\Omega^\oc)^{\alpha,0}_{\un,0}=t^\alpha$, we can identify the flows $\frac{\d}{\d x}$ and $\frac{\d}{\d t^\un_0}=A^\alpha\frac{\d}{\d t^\alpha_0}$. Clearly, the functions $v^\alpha=t^\alpha_0$ satisfy the subsystem of system~\eqref{eq:principal hierarchy for F-man} given by the flows~$\frac{\d}{\d t^\beta_0}$. Denote by $(v^\top)^\alpha\in\mcR^{\ot_{\orig}+\oc}[[t^*_{\ge 1}]]$ the solution of the principal hierarchy specified by the initial condition
$$
\left.(v^\top)^\alpha\right|_{t^*_{\ge 1}=0}=t^\alpha_0.
$$
It is often called the {\it topological solution}.\\

Let 
$$
(\Omega^\top)^{\alpha,p}_{\beta,q}:=\left.(\Omega^\oc)^{\alpha,p}_{\beta,q}\right|_{t^\gamma\mapsto(v^\top)^\gamma}\in\mcR^{\ot_\orig+\oc}[[t^*_{\ge 1}]].
$$
The {\it descendant vector potentials} $\omcF^a=(\mcF^{1,a},\ldots,\mcF^{N,a})$, $a\ge 0$, of our calibrated flat F-manifold are defined by
$$
\mcF^{\alpha,a}:=\sum_{b\ge 0}(\Omega^\top)^{\alpha,a}_{\beta,b}q^\beta_b\in\mcR^{\ot_\orig+\oc}[[t^*_{\ge 1}]],\quad a\ge 0,
$$
where $q^\beta_b:=t^\beta_b-A^\beta\delta_{b,1}$. We have the property
\begin{gather}\label{eq:derivative of descendant vector potential}
\frac{\d\mcF^{\alpha,a}}{\d t^\beta_b}=(\Omega^\top)^{\alpha,a}_{\beta,b}.
\end{gather}
The function $F^\alpha$ coincides with the function $\left.\mcF^{\alpha,0}\right|_{\substack{t^\beta\mapsto t^\beta+c^\beta\\t^*_{\ge 1}=0}}$ up to an affine function in the variables $t^\gamma$. 

\smallskip

\begin{remark}\label{remark:about shifts}
We see that in the construction of the descendant vector potentials we have to go from the functions $\Omega^{\alpha,p}_{\beta,q}$ to the shifted functions $(\Omega^\oc)^{\alpha,p}_{\beta,q}$. Actually, if we start from the shifted vector potential $((F^\oc)^1,\ldots,(F^\oc)^N)$ given by $(F^\oc)^\alpha:=F^\alpha|_{t^\beta\mapsto t^\beta-c^\beta}$, then the functions $(\Omega^\oc)^{\alpha,p}_{\beta,0}$ define a calibration and in the further construction of the descendant vector potentials we don't have to do any shifts and get the same functions $\mcF^{\alpha,a}$. That is why, in~\cite{BB19} the authors don't write explicitly the shifts needed in the construction of the descendant vector potetentials. We do it because we want to study the action of transformations~\eqref{eq:transformation of calibration} on descendant vector potentials and such a transformation shifts a point around which formal functions $\mcF^{\alpha,a}$ are defined.
\end{remark}

\smallskip

Let us adopt the convention
$$
\mcF^{\alpha,a}:=(-1)^{a+1}q^\alpha_{-a-1},\quad\text{if $a<0$}.
$$

\smallskip

\begin{proposition}\label{proposition:descendant vector potentials,equivalent approach}
A sequence of $N$-tuples of functions~$(\mcF^{1,a},\ldots,\mcF^{N,a})$, $\mcF^{\alpha,a}\in\mcR^{\ot_{\orig}}[[t^*_{\ge 1}]]$, $a\ge 0$, is a sequence of descendant vector potentials of a flat F-manifold if and only if the following equations are satisfied:
\begin{align}
\sum_{b\ge 0}q^\beta_{b+1}\frac{\d\mcF^{\alpha,a}}{\d q^\beta_b}=&-\mcF^{\alpha,a-1},&& a\in\mbZ,\label{eq:vector string equation}\\
\sum_{b\ge 0}q^\beta_b\frac{\d\mcF^{\alpha,a}}{\d q^\beta_b}=&\mcF^{\alpha,a},&& a\in\mbZ,\label{eq:vector dilaton equation}\\
\frac{\d^2\mcF^{\alpha,0}}{\d q^\beta_{b+1}\d q^\gamma_c}=&\frac{\d\mcF^{\mu,0}}{\d q^\beta_b}\frac{\d^2\mcF^{\alpha,0}}{\d q^\mu_0\d q^\gamma_c},&& b,c\ge 0,\label{eq:basic vector TRR}\\
\frac{\d\mcF^{\alpha,a+1}}{\d q^\beta_b}+\frac{\d\mcF^{\alpha,a}}{\d q^\beta_{b+1}}=&\frac{\d\mcF^{\alpha,a}}{\d q^\mu_0}\frac{\d\mcF^{\mu,0}}{\d q^\beta_b},&& a,b\ge 0.\label{eq:important vector relation}
\end{align}
\end{proposition}

\smallskip

It is not hard to check that equations~\eqref{eq:basic vector TRR} and~\eqref{eq:important vector relation} imply the following generalizations of equation~\eqref{eq:basic vector TRR}:
\begin{align}
\frac{\d^2\mcF^{\alpha,a}}{\d q^\beta_{b+1}\d q^\gamma_c}=&\frac{\d\mcF^{\mu,0}}{\d q^\beta_b}\frac{\d^2\mcF^{\alpha,a}}{\d q^\mu_0\d q^\gamma_c},&& a,b,c\ge 0,\label{eq:vector TRR1}\\
\frac{\d^2\mcF^{\alpha,a+1}}{\d q^\beta_b\d q^\gamma_c}=&\frac{\d\mcF^{\alpha,a}}{\d q^\mu_0}\frac{\d^2\mcF^{\mu,0}}{\d q^\beta_b\d q^\gamma_c},&& a,b,c\ge 0.\label{eq:vector TRR2}
\end{align}

\smallskip

\begin{proof}[Proof of Proposition~\ref{proposition:descendant vector potentials,equivalent approach}]
Suppose that $N$-tuples $(\mcF^{1,a},\ldots,\mcF^{N,a})$, $a\ge 0$, are descendant vector potentials of a flat F-manifold. Equation~\eqref{eq:vector string equation} was proved in~\cite{BB19}. Also we have $\sum_{b\ge 0}q^\beta_b\frac{\d (v^{\top})^\alpha}{\d q^\beta_b}=0$~\cite{BB19}, which implies that $\sum_{b\ge 0}q^\beta_b\frac{\d(\Omega^\top)^{\alpha,a}_{\gamma,c}}{\d q^\beta_b}=0$. Therefore, equation~\eqref{eq:vector dilaton equation} is true. Equation~\eqref{eq:basic vector TRR} follows from equations~\eqref{eq:TRR with differential} and~\eqref{eq:derivative of descendant vector potential}. The last equation~\eqref{eq:important vector relation} follows from equation~\eqref{eq:definition of Omegapq}.\\

Suppose now that functions $\mcF^{\alpha,a}\in\mcR^{\ot_{\orig}}[[t^*_{\ge 1}]]$ satisfy equations~\eqref{eq:vector string equation}--\eqref{eq:important vector relation}. Then it is straightforward to check that the functions $F^\alpha$ and $\Omega^{\alpha,d}_{\beta,0}$ given by
$$
F^\alpha:=\left.\mcF^{\alpha,0}\right|_{t^*_{\ge 1}=0},\qquad \Omega^{\alpha,d}_{\beta,0}:=\left.\frac{\d\mcF^{\alpha,d}}{\d t^\beta_0}\right|_{t^*_{\ge 1}=0},
$$
define a calibrated flat F-manifold such that the $N$-tuples $(\mcF^{1,a},\ldots,\mcF^{N,a})$ are its descendant vector potentials.
\end{proof}

\smallskip

Let $\omcF^a(t^*_*)$, $a\ge 0$, be a collection of descendant vector potentials and consider a linear change of variables 
\begin{gather*}
t^\alpha_a\mapsto\tt^\alpha_a(t^*_*)=M^\alpha_\mu t^\mu_a,
\end{gather*}
where $M=(M^\alpha_\beta)\in\GL(\mbC^N)$. Then it is easy to see that the collection of functions $(M.\mcF)^{\alpha,a}(\tt^*_*)$ defined by
\begin{gather}\label{eq:transformation of descendant potentials from linear changes}
(M.\mcF)^{\alpha,a}:=\left.M^\alpha_\mu\mcF^{\mu,a}\right|_{t^\beta_b=(M^{-1})^\beta_\gamma\tt^\gamma_b},\quad a\ge 0,
\end{gather}
satisfies all the equations~\eqref{eq:vector string equation}--\eqref{eq:important vector relation} and, thus, gives a collection of descendant vector potentials. The unit vector field $\tA^\alpha\frac{\d}{\d\tt^\alpha}$ of the associated flat F-manifold is given by $\tA^\alpha=M^\alpha_\mu A^\mu$. Clearly, this defines a $\GL(\mbC^N)$-action on collections of descendant vector potentials. We will use the notation
$$
\overline{M.\mcF}^a:=((M.\mcF)^{1,a},\ldots,(M.\mcF)^{N,a}).
$$

\smallskip

\subsection{Ancestor vector potentials}

Consider a collection of descendant vector potentials $\omcF^a$, $\mcF^{\alpha,a}\in\mcR^{\ot_{\orig}}[[t^*_{\ge 1}]]$. They are called {\it ancestor} if $\ot_{\orig}=0$ and the following property is satisfied:
\begin{equation}\label{eq:ancestor vector potential}
\left.\frac{\d\mcF^{\alpha,a}}{\d t^\beta_b}\right|_{t^*_*=0}=0,\quad a,b\ge 0.
\end{equation}

\smallskip

\begin{lemma}
Consider a flat F-manifold structure given by a vector potential $\oF$ with $F^\alpha\in\mbC[[t^1,\ldots,t^N]]$. Then there exists a unique calibration giving a collection of ancestor potentials. 
\end{lemma}
\begin{proof}
It is easy to check that such a calibration is uniquely determined by equation~\eqref{eq:property of a calibration} and the properties $\Omega^{\alpha,0}_{\beta,0}=\frac{\d F^\alpha}{\d t^\beta}-\left.\frac{\d F^\alpha}{\d t^\beta}\right|_{t^*=0}$ and $\left.\Omega^{\alpha,a}_{\beta,0}\right|_{t^*=0}=0$, $a\ge 0$.
\end{proof}

\smallskip

A unique calibration described by this lemma will be called the {\it ancestor calibration}.\\

Let us assign to the variable $t^\beta_b$ degree $b-1$, $\deg t^\beta_b:=b-1$. 

\smallskip

\begin{lemma}\label{lemma:technical property of ancestor potentials}
Consider a collection of ancestor vector potentials $\omcF^a$, $a\ge 0$. Then 
\begin{gather}\label{eq:degree of Falphaa}
\deg\mcF^{\alpha,a}\le -a-2,\quad a\ge 0,
\end{gather}
which means that all the monomials that form the power series $\mcF^{\alpha,a}$ have degree less or equal to~$-a-2$.
\end{lemma}
\begin{proof}
By equation~\eqref{eq:vector dilaton equation}, we have $\left.\mcF^{\alpha,a}\right|_{t^*_*=0}=0$. Therefore, property~\eqref{eq:degree of Falphaa} follows from the property
\begin{gather*}
\deg\left(\frac{\d\mcF^{\alpha,a}}{\d t^\beta_b}\right)\le -a-b-1,\quad a,b\ge 0,
\end{gather*}
which can be easily checked by induction using relations~\eqref{eq:vector TRR1} and~\eqref{eq:vector TRR2}. The lemma is proved.
\end{proof}

\smallskip

\begin{remark}\label{remark:polynomiality of ancestor potentials}
Lemma~\ref{lemma:technical property of ancestor potentials} implies that for a collection of ancestor vector potentials $\omcF^a$ we have
$$
\mcF^{\alpha,a}\in\mbC[t^*_{\ge 2}][[t^*_0,t^*_1]].
$$
Therefore, for any collection of constants $c^\beta_b\in\mbC$ such that $c^\beta_0=c^\beta_1=0$ the formal power series~$\mcF^{\alpha,a}$ can be expressed as a formal power series in the shifted variables $(t^\beta_b+c^\beta_b)$, for which property~\eqref{eq:degree of Falphaa} still holds.
\end{remark}

\smallskip

\subsection{Constant flat F-manifolds}

A flat F-manifold given by a vector potential $\oF=(F^1,\ldots,F^N)$ is called {\it constant} if its structure constants $c^\alpha_{\beta\gamma}=\frac{\d^2 F^\alpha}{\d t^\beta\d t^\gamma}$ are constants.\\ 

In any rank $N\ge 1$, the simplest constant flat F-manifold has the vector potential $\oF=\left(\frac{(t^1)^2}{2},\ldots,\frac{(t^N)^2}{2}\right)$ and the unit $\sum_{\alpha=1}^N\frac{\d}{\d t^\alpha}$. Its ancestor vector potentials are given by 
$$
\mcF^{\alpha,a}=\sum_{n\ge a+2}\sum_{\substack{d_1,\ldots,d_n\ge 0\\\sum d_i=n-2-a}}\frac{1}{n(n-1)}\frac{\prod t^\alpha_{d_i}}{a!\prod d_i!},\quad a\ge 0.
$$
This flat F-manifold is called the {\it trivial flat F-manifold of dimension $N$}. Applying to it the $\GL(\mbC^N)$-action~\eqref{eq:transformation of descendant potentials from linear changes}, we can get the ancestor vector potentials of any semisimple constant flat F-manifold. Clearly, in canonical coordinates the Christoffel symbols of the connection $\nabla$ and the rotation coefficients of any semisimple constant flat F-manifold vanish.

\smallskip

\subsection{Flat F-manifolds and the geometry of $\mbC^N[[z,z^{-1}]]$}

Let $\mcH:=\mbC^N[[z,z^{-1}]]$. Denote by $\phi_1,\ldots,\phi_N$ the standard basis in $\mbC^N$. Any element $f(z)\in\mcH$ can be expressed as
$$
f(z)=\sum_{d\ge 0}\frac{(-1)^{d+1}p^\alpha_d\phi_\alpha}{z^{d+1}}+\sum_{d\ge 0}q^\alpha_d\phi_\alpha z^d,\quad p^\alpha_d,q^\alpha_d\in\mbC,
$$
in a unique way. We view the coefficients $p^\alpha_d$ and $q^\alpha_d$ as coordinates on the space $\mcH$. We have the decomposition $\mcH=\mcH_+\oplus\mcH_-$, where $\mcH_+=\mbC^N[[z]]$ and $\mcH_-=z^{-1}\mbC^N[[z^{-1}]]$.\\

Consider a collection of functions $\mcF^{\alpha,a}\in\mcR^{\ot_\orig}[[t^*_{\ge 1}]]$, $1\le\alpha\le N$, $a\ge 0$. Recall that we relate the variables $t^\beta_b$ and $q^\beta_b$ by $q^\beta_b=t^\beta_b-A^\beta\delta_{b,1}$. We view~$\mcF^{\alpha,a}$ as formal functions on the space $\mcH_+$ near the point $\ot_\orig-\phi_\un z$, where $\phi_\un:=A^\alpha\phi_\alpha$. Consider the graph of the collection of functions $\mcF^{\alpha,a}$ in the space $\mcH$:
$$
\mcC:=\left\{(p^*_*,q^*_*)\in\mcH|p^\alpha_a=\mcF^{\alpha,a}\right\}\subset\mcH.
$$

\smallskip

\begin{theorem}\label{theorem:flat F-manifold as a cone}
The functions $\mcF^{\alpha,a}$ satisfy equations~\eqref{eq:vector string equation}--\eqref{eq:important vector relation} if and only if the following three conditions are satisfied:
\begin{itemize}
\item[1.] $\mcC$ is a cone with the vertex at the origin;
\item[2.] for any point $f\in\mcC$, we have $f\in zT_f\mcC$, where $T_f\mcC\subset\mcH$ denotes the tangent space to $\mcC$ at the point $f$;
\item[3.] for any point $f\in\mcC$, the tangent space $T_f\mcC$ is tangent to $\mcC$ along $zT_f\mcC$.
\end{itemize}
\end{theorem}

\smallskip

Note that, comparing to a similar result about Frobenius manifolds~\cite[Theorem~1]{Giv04}, we just drop the condition that the cone is Lagrangian and the condition that any tangent space~$L$ is tangent to $\mcC$ {\it exactly} along $zL$. Actually, the second of these two conditions follows automatically from the conditions of Theorem~\ref{theorem:flat F-manifold as a cone}. We will show it in~Lemma~\ref{lemma:last condition} after the proof of the theorem.

\smallskip

\begin{proof}[Proof of Theorem~\ref{theorem:flat F-manifold as a cone}]
Clearly, $\mcC$ is a cone with the vertex at the origin if and only if the vector field $\sum_{a\ge 0}q^\alpha_a\frac{\d}{\d q^\alpha_a}+\sum_{a\ge 0}p^\alpha_a\frac{\d}{\d p^\alpha_a}$ is tangent to $\mcC$, which is equivalent to equation~\eqref{eq:vector dilaton equation}.\\

Assume now that equation~\eqref{eq:vector dilaton equation} is satisfied. Therefore, for any $f\in\mcC$ the tangent space $T_f\mcC$ passes through the origin. Since $\mcC$ is the graph of the collection of functions~$\mcF^{\alpha,a}$, we can lift the coordinate vector fields $\frac{\d}{\d q^\alpha_a}$ on $\mcH_+$ to vector fields~$e_{\alpha,a}$ on $\mcC$, which give a natural basis in the tangent spaces to $\mcC$. We have
$$
e_{\alpha,a}=\frac{\d}{\d q^\alpha_a}+\sum_{b\ge 0}\frac{\d\mcF^{\beta,b}}{\d q^\alpha_a}\frac{\d}{\d p^\beta_b},\quad1\le\alpha\le N,\quad a\ge 0.
$$
Consider now a point $f=(p^\beta_b,q^\alpha_a)\in\mcC$. We have $z^{-1}f=(\tp^\beta_b,\tq^\alpha_a)$, where
$$
\tq^\alpha_a=q^\alpha_{a+1},\qquad 
\tp^\beta_b=
\begin{cases}
-q^\beta_0,&\text{if $b=0$},\\
-\mcF^{\beta,b-1},&\text{if $b\ge 1$}.
\end{cases}
$$
The condition $f\in zT_f\mcC\Leftrightarrow z^{-1}f\in T_f\mcC$ is satisfied if and only if the vector $$
\sum_{a\ge 0}q^\alpha_{a+1}\frac{\d}{\d q^\alpha_a}-q^\alpha_0\frac{\d}{\d p^\alpha_0}-\sum_{a\ge 0}\mcF^{\alpha,a}\frac{\d}{\d p^\alpha_{a+1}}\in T_f\mcH
$$
belongs to $T_f\mcC$. The last property is equivalent to the equation
$$
\sum_{a\ge 0}q^\alpha_{a+1}\frac{\d\mcF^{\beta,b}}{\d q^\alpha_a}=
\begin{cases}
-\mcF^{\beta,b-1},&\text{if $b\ge 1$},\\
-q^\beta_0,&\text{if $b=0$},
\end{cases}
$$
which is exactly equation~\eqref{eq:vector string equation}.\\

Let us assume that equations~\eqref{eq:vector string equation} and~\eqref{eq:vector dilaton equation} are satisfied. For a tangent space $T_f\mcC$, a basis in $z T_f\mcC$ is given by the vectors
$$
z e_{\alpha,a}=\frac{\d}{\d q^\alpha_{a+1}}-\frac{\d\mcF^{\beta,0}}{\d q^\alpha_a}\frac{\d}{\d q^\beta_0}-\sum_{b\ge 1}\frac{\d\mcF^{\beta,b}}{\d q^\alpha_a}\frac{\d}{\d p^\beta_{b-1}},\quad 1\le\alpha\le N,\quad a\ge 0.
$$
Thus,
\begin{gather*}
zT_f\mcC\subset T_f\mcC\Leftrightarrow ze_{\alpha,a}=e_{\alpha,a+1}-\frac{\d\mcF^{\mu,0}}{\d q^\alpha_a}e_{\mu,0}\Leftrightarrow \frac{\d\mcF^{\beta,b+1}}{\d q^\alpha_a}+\frac{\d\mcF^{\beta,b}}{\d q^\alpha_{a+1}}=\frac{\d\mcF^{\mu,0}}{\d q^\alpha_a}\frac{\d\mcF^{\beta,b}}{\d q^\mu_0},
\end{gather*}
where the last equation coincides with~\eqref{eq:important vector relation}.\\

Let us assume now that equations~\eqref{eq:vector string equation},~\eqref{eq:vector dilaton equation} and~\eqref{eq:important vector relation} are satisfied. We already know that for any point $f\in\mcC$ we have $f\in zT_f\mcC\subset T_f\mcC$. Therefore, the space $T_f\mcC$ is tangent to~$\mcC$ along~$zT_f\mcC$ if and only if the Lie derivative of the coefficient of $\frac{\d}{\d p^\alpha_a}$ in $e_{\gamma,c}$ along the vector field~$z e_{\beta,b}$ is zero. The last property is equivalent to the equation
$$
\left(\frac{\d}{\d q^\beta_{b+1}}-\frac{\d\mcF^{\mu,0}}{\d q^\beta_b}\frac{\d}{\d q^\mu_0}\right)\frac{\d\mcF^{\alpha,a}}{\d q^\gamma_c}=0.
$$
Since we have assumed that equation~\eqref{eq:important vector relation} is satisfied, the last equation is equivalent to equation~\eqref{eq:basic vector TRR}. The theorem is proved. 
\end{proof}

\smallskip

The cone $\mcC\subset\mcH$ obtained from a calibrated flat F-manifold by the construction described above will be called a {\it generalized Givental cone}. The generalized Givental cone corresponding to an ancesor calibration will be called an {\it ancestor cone}. The ancestor cone corresponding to the trivial flat F-manifold of dimension $N$ will be denoted by $\mcC^\triv_N$.

\smallskip

\begin{lemma}\label{lemma:last condition}
Consider functions $\mcF^{\alpha,a}$ satisfying equations~\eqref{eq:vector string equation}--\eqref{eq:important vector relation} and the associated cone $\mcC\subset\mcH$. Then each tangent space $L$ to the cone $\mcC$ is tangent to it exactly along~$zL$.
\end{lemma}
\begin{proof}
The subspace $zL\subset L$ has codimension~$N$ and a basis in $L/zL$ is given by the vectors~$e_{\alpha,0}$. The coefficients of $\frac{\d}{\d p^\beta_b}$ in the vector fields $e_{\alpha,a}$ give natural functions on the space of tangent spaces to $\mcC$. Consider the coefficient $\Coef_{\d/\d p^\beta_0}e_{\un,0}=\frac{\d\mcF^{\beta,0}}{\d q^\un_0}$ and the Lie derivative $L_{e_{\alpha,0}}$ of it along the vector field $e_{\alpha,0}$. Since
$$
\left.\left(L_{e_{\alpha,0}}\frac{\d\mcF^{\beta,0}}{\d q^\un_0}\right)\right|_{q^\gamma_c=\delta_{c,0}t^\gamma_\orig-\delta_{c,1}A^\gamma}=\delta^\beta_\alpha,
$$
we see that any tangent space $L$ to $\mcC$ is tangent to it exactly along $zL$.
\end{proof}

\smallskip

Consider a map $M\in\GL(\mbC^N)$, $M\phi_{\alpha}=M^\beta_\alpha\phi_\beta$. It induces a linear map $\mcH\mapsto\mcH$ and, obviously, the conditions for a cone~$\mcC$ formulated in Theorem~\ref{theorem:flat F-manifold as a cone} are preserved by this map. It is easy to see that the transformation of the corresponding descendant vector potentials is described by formula~\eqref{eq:transformation of descendant potentials from linear changes}.

\smallskip

\subsection{$J$-function}

Consider a generalized Givental cone $\mcC\subset\mcH$. In the same way as for the case of Frobenius manifolds~\cite{Giv04}, one can introduce the $J$-function associated to~$\mcC$. For this we consider the intersection of the cone $\mcC$ with the affine space $-z\phi_\un+z\mcH_-$. Via the projection to $-z\phi_\un+\mbC^N$ along $\mcH_-$, the intersection becomes the graph of a function from $\mbC^N$ to $\mcH$ called the {\it $J$-function}:
$$
\mbC^N\ni t^\alpha\phi_\alpha\mapsto J(-z,t^*)=-z\phi_\un+t^\alpha\phi_\alpha+\sum_{j\ge 1}J_k(t^*)(-z)^{-k}.
$$
If $\omcF^a$ are the descendant vector potentials, corresponding to the cone $\mcC$, then, clearly, 
$$
J_k(t^*)=\left.\mcF^{\alpha,k-1}\phi_\alpha\right|_{t^*_{\ge 1}=0},\quad k\ge 1.
$$
The cone $\mcC$ is uniquely determined by its $J$-function. This fact can be derived, for example, from equation~\eqref{eq:important vector relation}.\\

The $J$-function of the cone $\mcC^\triv_N$ is 
$$
J(z,t^*)=\sum_{\alpha=1}^N ze^{t^\alpha/z}\phi_\alpha.
$$

\smallskip

\subsection{Generalized Givental group action}

Consider two groups
\begin{align*}
G_+:=&\left\{R(z)=\Id+\sum_{i\ge 1}R_i z^i\in\End(\mbC^N)[[z]]\right\},\\
G_-:=&\left\{S(z)=\Id+\sum_{i\ge 1}S_i z^{-i}\in\End(\mbC^N)[[z^{-1}]]\right\}.
\end{align*}
Let us call the groups $G_+$ and $G_-$ the {\it upper triangular} and the {\it lower triangular group}, respectively. If we ignore for a moment potential problems caused by infinite summation, we can say that the groups $G_+$ and $G_-$ act on the space $\mcH$ by the left multiplication. Moreover, the conditions for a cone~$\mcC$ formulated in Theorem~\ref{theorem:flat F-manifold as a cone} are preserved by these actions. Thus, we get $G_{\pm}$-actions on the space of calibrated flat F-manifolds. Let us analyze these actions more carefully.

\smallskip

\subsubsection{Upper triangular group}\label{subsubsection:upper triangular group}

\begin{proposition}\label{proposition:R-action is well-defined}
The $G_+$-action on the space of ancestor cones is well defined.
\end{proposition}
\begin{proof}
Let $\mcC\subset\mcH$ be the cone corresponding to a collection $\omcF^a$ of ancestor vector potentials. Consider a point 
$$
f(z)=\sum_{a\ge 0}\frac{(-1)^{a+1}\mcF^{\alpha,a}\phi_\alpha}{z^{a+1}}+\sum_{a\ge 0}q^\alpha_a\phi_\alpha z^a\in\mcC.
$$
For $R(z)\in G_+$, we have 
$$
R(z)f(z)=\sum_{a\ge 0}\frac{(-1)^{a+1}\tmcF^{\alpha,a}\phi_\alpha}{z^{a+1}}+\sum_{a\ge 0}\tq^\alpha_a\phi_\alpha z^a,
$$
where
\begin{align}
\tq^\alpha_a=&q^\alpha_a+\sum_{i=1}^a(R_i)^\alpha_\mu q^\mu_{a-i}+\sum_{i\ge a+1}(R_i)^\alpha_\mu(-1)^{i-a}\mcF^{\mu,i-a-1},&& a\ge 0,\label{eq:R-transformation of q}\\
\tmcF^{\alpha,a}=&\mcF^{\alpha,a}+\sum_{i\ge 1}(-1)^i(R_i)^\alpha_\mu\mcF^{\mu,a+i},&& a\ge 0.\label{eq:R-transformation of mcF}
\end{align}
Lemma~\ref{lemma:technical property of ancestor potentials} implies that the infinite sums on the right-hand sides of these two equations are well defined. Lemma~\ref{lemma:technical property of ancestor potentials} also implies that the transformation $q^\alpha_a\mapsto\tq^\alpha_a$ can be considered as a change of variables. Therefore, we can express the functions $\tmcF^{\alpha,a}(q^*_*)$ as functions of the variables $\tq^\beta_b$, $\tmcF^{\alpha,a}(q^*_*)=\hmcF^{\alpha,a}(\tq^*_*)$. Note that the function $\hmcF^{\alpha,a}$ is a formal function around the point 
$$
R(z)(-\phi_\un z)=-\phi_\un z-\sum_{i\ge 1}(R_i)^\mu_\un\phi_\mu z^{i+1}\in\mcH_+,
$$
but, by Remark~\ref{remark:polynomiality of ancestor potentials}, it can be expressed as a formal function around the point $-\phi_\un z$.\\

It remains to check that the descendant vector potentials $(\hmcF^{1,a},\ldots,\hmcF^{N,a})$ are ancestor. Let us express the functions $\hmcF^{\alpha,a}$ as formal power series in the variables $\tt^\beta_b=\tq^\beta_b+A^\beta\delta_{b,1}$. Formulas~\eqref{eq:R-transformation of q} and~\eqref{eq:R-transformation of mcF} imply that property~\eqref{eq:degree of Falphaa} holds for the functions $\hmcF^{\alpha,a}$. Therefore, $\left.\frac{\d\hmcF^{\alpha,a}}{\d\tt^\beta_b}\right|_{\tt^*_*=0}=0$ and the descendant vector potentials $(\hmcF^{1,a},\ldots,\hmcF^{N,a})$ are ancestor.
\end{proof}

\smallskip

Consider a collection of ancestor vector potentials $\omcF^a$ and the corresponding cone $\mcC\subset\mcH$. For $R(z)\in G_+$, denote by $\overline{R(z).\mcF}^a=\left((R(z).\mcF)^{1,a},\ldots,(R(z).\mcF)^{N,a}\right)$ the ancestor vector potentials, corresponding to the cone $R(z)\mcC\subset\mcH$. Recall that we use the convention
$$
\mcF^{\alpha,a}=(-1)^{a+1}q^\alpha_{-a-1},\quad\text{if $a<0$}.
$$

\smallskip

\begin{proposition}\label{proposition:infinitesimal R-action}
The infinitesimal action of the group $G_+$ on the space of ancestor vector potentials is given by
\begin{gather}\label{eq:infinitesimal R-action}
\left.\frac{d}{d\eps}(e^{\eps r(z)}.\mcF)^{\alpha,a}\right|_{\eps=0}=\sum_{i\ge 1}(-1)^i(r_i)^\alpha_\mu\mcF^{\mu,a+i}+\sum_{i\ge 1,\,j\ge 0}(-1)^{i-j-1}(r_i)^\mu_\nu\frac{\d\mcF^{\alpha,a}}{\d t^\mu_j}\mcF^{\nu,i-j-1},\quad a\ge 0,
\end{gather}
where $r(z)=\sum_{i\ge 1}r_i z^i$, $r_i\in\End(\mbC^N)$.
\end{proposition}
\begin{proof}
We have
$$
\left.e^{\eps r(z)}[\mcF]^{\alpha,a}\right|_{q^\beta_b\mapsto\tq^\beta_b}=\tmcF^{\alpha,a},
$$
where $\tq^\beta_b$ and $\tmcF^{\alpha,a}$ are given by formulas~\eqref{eq:R-transformation of q} and~\eqref{eq:R-transformation of mcF} with $R(z)=e^{\eps r(z)}$. Differentiating both sides of this equation with respect to  $\eps$ and setting $\eps=0$, we get formula~\eqref{eq:infinitesimal R-action}.
\end{proof}

\smallskip

\subsubsection{Lower triangular group}\label{subsubsection:lower triangular group}

Consider a flat F-manifold with a vector potential $\oF$, $F^\alpha\in\mcR^{\ot_\orig}$, its calibration given by matrices $\Omega^p_0$, and the associated descendant vector potentials~$\omcF^a$, $\mcF^{\alpha,a}\in\mcR^{\ot_\orig+\oc}[[t^*_{\ge 1}]]$, where $c^\alpha=\Omega^{\alpha,0}_{\un,0}-t^\alpha$. Let $\mcC\subset\mcH$ be the associated cone.

\smallskip

\begin{proposition}\label{proposition:lower triangular group action}
Consider an arbitrary element $S(z)=\Id+\sum_{i\ge 1}S_i z^{-i}\in G_-$ and let $\log S(z)=s(z)=\sum_{i\ge 1}s_i z^{-i}$.\\
1. The cone $S(z)\mcC$ is well defined. It corresponds to a collection of descendant vector potentials, which we denote by $\overline{S(z).\mcF}^a=\left((S(z).\mcF)^{1,a},\ldots,(S(z).\mcF)^{N,a}\right)$, with $(S(z).\mcF)^{\alpha,a}\in\mcR^{\ot_\orig+\oc-S_1\phi_\un}[[t^*_{\ge 1}]]$ given by
\begin{gather}\label{eq:S-action on mcF}
(S(z).\mcF)^{\alpha,a}=e^{-\widehat{s(z)}}\left(\mcF^{\alpha,a}+\sum_{i=1}^a(-1)^i(S_i)^\alpha_\mu\mcF^{\mu,a-i}+(-1)^{a+1}\sum_{i\ge a+1}(S_i)^\alpha_\mu q^\mu_{i-a-1}\right),
\end{gather}
where $\widehat{s(z)}=\sum_{\substack{i\ge 1\\j\ge 0}}(s_i)^\alpha_\beta q^\beta_{i+j}\frac{\d}{\d q^\alpha_j}$.\\
2. The descendant vector potentials $\overline{S(z).\mcF}^a$ correspond to the same flat F-manifold, but with the different calibration given by the matrices $\tOmega^d_0$ defined by
$$
\Id+\sum_{d\ge 1}\tOmega^{d-1}_0 z^d=S(-z^{-1})\left(\Id+\sum_{d\ge 1}\Omega^{d-1}_0 z^d\right).
$$
\end{proposition}
\begin{proof}
1. For a point 
$$
f(z)=\sum_{a\ge 0}\frac{(-1)^{a+1}\mcF^{\alpha,a}\phi_\alpha}{z^{a+1}}+\sum_{a\ge 0}q^\alpha_a\phi_\alpha z^a\in\mcC,
$$
we have
$$
S(z)f(z)=\sum_{a\ge 0}\frac{(-1)^{a+1}(\mcF')^{\alpha,a}\phi_\alpha}{z^{a+1}}+\sum_{a\ge 0}(q')^\alpha_a\phi_\alpha z^a,
$$
where
\begin{align*}
(q')^\alpha_a=&q^\alpha_a+\sum_{i\ge 1}(S_i)^\alpha_\mu q^\mu_{a+i}=e^{\widehat{s(z)}} q^\alpha_a,&& a\ge 0,\\
(\mcF')^{\alpha,a}=&\mcF^{\alpha,a}+\sum_{i=1}^a(-1)^i(S_i)^\alpha_\mu\mcF^{\mu,a-i}+(-1)^{a+1}\sum_{i\ge a+1}(S_i)^\alpha_\mu q^\mu_{i-a-1},&& a\ge 0.
\end{align*}
Expressing the function $(\mcF')^{\alpha,a}$ as a function of the variables $(q')^{\beta}_b$, we get formula~\eqref{eq:S-action on mcF}. This proves Part 1 of the proposition.\\

2. Denote $S(z)^{-1}=\Id+\sum_{i\ge 1}\tS_i z^{-i}$. Then we have $\tOmega^0_p=\sum_{i=0}^{p+1}\Omega^0_{p-i}\tS_i$, $p\ge -1$. From this, it is easy to see that the topological solution $(\tv^\top)^\alpha$ and the matrices $(\tOmega^\top)^0_b$ corresponding to the new calibration are given by
$$
(\tv^\top)^\alpha=e^{-\widehat{s(z)}}(v^\top)^\alpha,\qquad (\tOmega^\top)^0_b=e^{-\widehat{s(z)}}\left(\sum_{i=0}^{b+1}(\Omega^\top)^0_{b-i}\tS_i\right).
$$
Then for the descendant vector potentials $(\tmcF^{1,a},\ldots,\tmcF^{N,a})$ corresponding to the new calibration we get
\begin{align*}
\tmcF^{\alpha,0}=&\sum_{b\ge 0}q^\beta_b(\tOmega^\top)^{\alpha,0}_{\beta,b}=\sum_{b\ge 0}q^\beta_b e^{-\widehat{s(z)}}\left(\sum_{i=0}^{b+1}(\Omega^\top)^{\alpha,0}_{\mu,b-i}(\tS_i)^\mu_\beta\right)=\\
=&e^{-\widehat{s(z)}}\left(\sum_{b,j\ge 0}(S_j)^\beta_\nu q^\nu_{b+j}\sum_{i=0}^{b+1}(\Omega^\top)^{\alpha,0}_{\mu,b-i}(\tS_i)^\mu_\beta\right)=
\end{align*}
\begin{align*}
=&e^{-\widehat{s(z)}}\left(\sum_{i,j,l\ge 0}(\Omega^\top)^{\alpha,0}_{\mu,l}(\tS_i)^\mu_\beta(S_j)^\beta_\nu q^\nu_{i+j+l}+\sum_{b,j\ge 0}(\tS_{b+1})^\alpha_\beta(S_j)^\beta_\nu q^\nu_{b+j}\right)=\\
=&e^{-\widehat{s(z)}}\left(\mcF^{\alpha,0}-\sum_{l\ge 0}(S_{l+1})^\alpha_\nu q^\nu_l\right)=(S(z).\mcF)^{\alpha,0}.
\end{align*}
Knowing that $\tmcF^{\alpha,0}=(S(z).\mcF)^{\alpha,0}$ is enough to conclude that $\tmcF^{\alpha,a}=(S(z).\mcF)^{\alpha,a}$ for all $a\ge 0$. This completes the proof of part 2 of the proposition. 
\end{proof}

\smallskip

\begin{proposition}\label{proposition:all from ancestor}
Any generalized Givental cone can be obtained from some ancestor cone by the action of an element from the group~$G_-$. 
\end{proposition}
\begin{proof}
Choosing a matrix $S_1$ such that $S_1\phi_\un=\ot_\orig+\oc$ and using Proposition~\ref{proposition:lower triangular group action}, we get
$$
((\Id+S_1z).\mcF)^{\alpha,a}\in\mcR^0[[t^*_{\ge 1}]].
$$
So, without loss of generality, we can assume that $\ot_\orig=\oc=0$. Using again Proposition~\ref{proposition:lower triangular group action}, we see that the cone
$$
\left(\Id+\sum_{j\ge 1}(-1)^jz^{-j}\Omega^{j-1}_0(0)\right)^{-1}\mcC
$$
is ancestor. The proposition is proved.
\end{proof}

\smallskip

Let $\oF=(F^1,\ldots,F^N)$, $F^\alpha\in\mbC[[t^1,\ldots,t^N]]$, be a vector potential of a flat F-manifold. Denote by $\mcC$ the associated ancestor cone and by $\Omega^j_{0}$ the matrices defining the ancestor calibration. Consider a family of vector potentials $\oF_{\otau}$, depending on formal parameters $\tau^1,\ldots,\tau^N$, $\otau=(\tau^1,\ldots,\tau^N)$, defined by
$$
F^\alpha_{\otau}:=F^\alpha|_{t^\beta\mapsto t^\beta+\tau^\beta}\in\mbC[[\tau^*]][[t^*]].
$$
Denote by $\mcC_{\otau}$ the associated family of ancestor cones.

\smallskip

\begin{lemma}\label{lemma:ancestor cones at shifted points}
We have
\begin{gather}\label{eq:ancestor cones at shifted points}
\mcC_{\otau}=\left(\Id+\sum_{j\ge 1}(-1)^j\Omega^{j-1}_0(\tau^*)z^{-j}\right)^{-1}\mcC.
\end{gather}
\end{lemma}

\smallskip

According to Proposition~\ref{proposition:lower triangular group action}, the object on the right-hand side of equation~\eqref{eq:ancestor cones at shifted points} is a germ of a graph over the point $-\tau^\alpha\phi_\alpha-\phi_\un z\in\mcH_+$, while $\mcC_{\otau}$ is a germ of a graph over the point $-\phi_\un z\in\mcH_+$. However, since the parameters $\tau^\alpha$ are formal, the cone on the right-hand side can be considered as a germ of a graph over the point $-\phi_\un z\in\mcH_+$. 

\smallskip

\begin{proof}[Proof of Lemma~\ref{lemma:ancestor cones at shifted points}]
By Proposition~\ref{proposition:lower triangular group action}, the cone on the right-hand side of equation~\eqref{eq:ancestor cones at shifted points} corresponds to the same flat F-manifold given by the vector potential $\oF$ together with the calibration
$$
\Id+\sum_{j\ge 1}\tOmega^{j-1}_0(t^*)z^j=\left(\Id+\sum_{j\ge 1}\Omega^{j-1}_0(\tau^*)z^j\right)^{-1}\left(\Id+\sum_{j\ge 1}\Omega^{j-1}_0(t^*)z^j\right).
$$
By Remark~\ref{remark:about shifts}, this cone also corresponds to the flat F-manifold given by the vector potential~$\oF_{\otau}$ together with the calibration $\Id+\sum_{j\ge 1}(\tOmega^{-\otau})^{j-1}_0z^j$. Since
$$
\left.\left(\Id+\sum_{j\ge 1}(\tOmega^{-\otau})^{j-1}_0z^j\right)\right|_{t^\alpha=0}=\left.\left(\Id+\sum_{j\ge 1}\tOmega^{j-1}_0z^j\right)\right|_{t^\alpha=\tau^\alpha}=\Id,
$$
the cone on the right-hand side of~\eqref{eq:ancestor cones at shifted points} is indeed the ancestor cone corresponding to the flat F-manifold with the vector potential~$\oF_{\otau}$.
\end{proof}

\smallskip

\subsection{Reconstruction of semisimple flat F-manifolds}\label{section:reconstruction}

Consider a flat F-manifold given by a vector potential $\oF=(F^1,\ldots,F^N)$, $F^\alpha\in\mbC[[t^1,\ldots,t^N]]$. Consider the constant flat F-manifold given by the structure constants of our flat F-manifold at the origin. We call this constant flat F-manifold the {\it constant part of our flat F-manifold}. 

\smallskip

\begin{theorem}\label{theorem:reconstruction}
Suppose that our flat F-manifold is semisimple at the origin. Then the corresponding ancestor cone $\mcC$ can be obtained from the ancestor cone $\mcC^\const$ of the constant part by some element $R(z)$ of the group $G_+$.
\end{theorem}

\smallskip

Before proving the theorem, let us present an important technical result. Consider a calibrated flat F-manifold given by a vector potential $\oF$ and matrices $\Omega^p_0$, $F^\alpha,\Omega^{\alpha,p}_{\beta,0}\in\mcR^{\ot_\orig}$, and the associated descendant vector potentials $\omcF^a$, $\mcF^{\alpha,a}\in\mcR^{\ot_{\orig}+\oc}[[t^*_{\ge 1}]]$. Recall that the constants~$c^\alpha$ are given by $\Omega^{\alpha,0}_{\un,0}=t^\alpha+c^\alpha$. Consider the associated cone $\mcC\subset\mcH$.

\smallskip

\begin{proposition}\label{proposition:another parameterization}
The cone $\mcC$ has the following parameterization:
$$
\mcC=\left\{\left(\sum_{j\ge 0}(-1)^j\Omega^{j-1}_0 z^{-j}\right)\sum_{i\ge 1}q^\alpha_i\phi_\alpha z^i\right\}.
$$
\end{proposition}
\begin{proof}
During the proof of the proposition we will denote the coordinates on $\mcH$ by $\tp^\beta_b$ and $\tq^\alpha_a$. We have
$$
\left(\sum_{j\ge 0}(-1)^j\Omega^{j-1}_0 z^{-j}\right)\sum_{i\ge 1}q^\alpha_i\phi_\alpha z^i=\sum_{a\ge 0}\frac{(-1)^{a+1}\tmcF^{\alpha,a}(q^*_*)\phi_\alpha}{z^{a+1}}+\sum_{a\ge 0}\tq^\alpha_a(q^*_*)\phi_\alpha z^a,
$$
where
\begin{align*}
\tq^\alpha_a(q^*_*)=&
\begin{cases}
t^\alpha_0+c^\alpha+\sum_{i\ge 1}(-1)^i\Omega^{\alpha,i-1}_{\mu,0}t^\mu_i,&\text{if $a=0$},\\
q^\alpha_a+\sum_{i\ge 1}(-1)^i\Omega^{\alpha,i-1}_{\mu,0} q^\mu_{i+a},&\text{if $a\ge 1$},
\end{cases}\\
\tmcF^{\alpha,a}(q^*_*)=&\sum_{i\ge 1}(-1)^i\Omega^{\alpha,i+a}_{\mu,0}q^\mu_i,\hspace{3.1cm} a\ge 0.
\end{align*}
We see that the transformation $q^\alpha_a\mapsto\tq^\alpha_a(q^*_*)$ can be considered as a change of variables and, therefore, we can express the functions $\tmcF^{\alpha,a}(q^*_*)$ as functions of the variables $\tq^\beta_b$, $\tmcF^{\alpha,a}(q^*_*)=\left.\hmcF^{\alpha,a}(\tq^*_*)\right|_{\tq^\beta_b=\tq^\beta_b(q^*_*)}$, where $\hmcF^{\alpha,a}\in\mbC[[\tq^\beta_0-(t^\beta_\orig+c^\beta),\tq^\beta_1+A^\beta,\tq^\beta_{\ge 2}]]$.\\

Let us first check that the submanifold $\mcC'\subset\mcH$ given by
$$
\mcC':=\left\{\left(\sum_{j\ge 0}(-1)^j\Omega^{j-1}_0 z^{-j}\right)\sum_{i\ge 1}q^\alpha_i\phi_\alpha z^i\right\}=\left\{\sum_{i\ge 1,\,j\ge 0}(-1)^j\Omega^{\mu,j-1}_{\nu,0}q^\nu_i\phi_\mu z^{i-j}\right\}
$$
satisfies the conditions from Theorem~\ref{theorem:flat F-manifold as a cone}. Regarding the first condition, we see that a vector
$$
f=\sum_{i\ge 1,\,j\ge 0}(-1)^j\Omega^{\mu,j-1}_{\nu,0}q^\nu_i\phi_\mu z^{i-j}\in\mcC'\subset\mcH
$$
depends linearly on the variables $q^\nu_i$ with $i\ge 1$. Therefore, $\mcC'$ is a cone with the vertex at the origin.\\ 

For $f\in\mcH$, denote by $f^\alpha_i$ the coefficient of $\phi_\alpha z^i$ in $f$ and let us use these coefficients as coordinates on $\mcH$. For $f\in\mcC'$ a basis in the tangent space $T_f\mcC'$ is given by the vectors
$$
e_{\alpha,a}:=
\begin{cases}
\sum_{i\ge 1,\,j\ge 0}(-1)^j\frac{\d\Omega^{\mu,j-1}_{\nu,0}}{\d q^\alpha_0}q^\nu_i\frac{\d}{\d f^\mu_{i-j}}=-\sum_{i\ge 1,\,j\ge 0}(-1)^j q^\nu_i c^\theta_{\nu\alpha}\Omega^{\mu,j-1}_{\theta,0}\frac{\d}{\d f^\mu_{i-j-1}},&\text{if $a=0$},\\
\sum_{j\ge 0}(-1)^j\Omega^{\mu,j-1}_{\alpha,0}\frac{\d}{\d f^\mu_{a-j}},&\text{if $a\ge 1$}.
\end{cases}
$$
One can immediately see that after the obvious identification of the vector spaces $\mcH$ and $T_f\mcH$ the vectors $z^{-1}f$ and $-e_{\un,0}$ become equal. Thus, the second condition from Theorem~\ref{theorem:flat F-manifold as a cone} is satisfied.\\

Let us check the third condition. We compute 
$$
z e_{\alpha,a}=
\begin{cases}
e_{\alpha,a+1},&\text{if $a\ge 1$},\\
-\sum_{i\ge 1}q^\nu_i c^\theta_{\nu\alpha}e_{\theta,i},&\text{if $a=0$},
\end{cases}
$$
which implies that $z T_f\mcC'\subset T_f\mcC'$. Define a matrix $M=(M^\alpha_\beta)$ by $M^\alpha_\beta:=-q^\theta_1 c^\alpha_{\theta\beta}$. Note that, since $M^\alpha_\beta=\delta^\alpha_\beta-t^\theta_1 c^\alpha_{\theta\beta}$, the matrix~$M$ is invertible. Therefore, the vectors $e_{\alpha,a}$ with $a\ge 1$ give a basis in $z T_f\mcC'$. We see that in order to check the third condition from Theorem~\ref{theorem:flat F-manifold as a cone} it is sufficient to check that $\frac{\d e_{\alpha,a}}{\d q^\beta_b}\in T_f\mcC'$ for $a\ge 0$ and $b\ge 1$. Clearly, $\frac{\d e_{\alpha,a}}{\d q^\beta_b}=0$ if $a,b\ge 1$. For $b\ge 1$ we have
$$
\frac{\d e_{\alpha,0}}{\d q^\beta_b}=-\sum_{j\ge 0}(-1)^jc^\theta_{\alpha\beta}\Omega^{\mu,j-1}_{\theta,0}\frac{\d}{\d f^\mu_{b-1-j}}=
\begin{cases}
-c^\theta_{\alpha\beta}e_{\theta,b-1},&\text{if $b\ge 2$},\\
-\sum_{j\ge 0}(-1)^jc^\theta_{\alpha\beta}\Omega^{\mu,j-1}_{\theta,0}\frac{\d}{\d f^\mu_{-j}},&\text{if $b=1$},
\end{cases}
$$
and it remains to check that 
$$
h_\alpha:=\sum_{j\ge 0}(-1)^j\Omega^{\mu,j-1}_{\alpha,0}\frac{\d}{\d f^\mu_{-j}}\in T_f\mcC'.
$$
We can easily see that
$$
M^\nu_\alpha h_\nu=e_{\alpha,0}+\sum_{i\ge 2}q^\nu_i c^\theta_{\nu\alpha}e_{\theta,i-1},
$$
and, since the matrix $M$ is invertible, we get $h_\alpha\in T_f\mcC'$.\\ 

Let us finally prove that $\mcC=\mcC'$. Note that $\left.\tq^\beta_b(q^*_*)\right|_{t^*_{\ge 1}=0}=\delta_{b,0}(t^\beta_0+c^\beta)-\delta_{b,1}A^\beta$ and
\begin{gather*}
\left.\tmcF^{\alpha,a}(q^*_*)\right|_{t^*_{\ge 1}=0}=\Omega^{\alpha,a+1}_{\un,0}=\mcF^{\alpha,a}\Big|_{\substack{t^\gamma_0\mapsto t^\gamma_0+c^\gamma\\t^*_{\ge 1}=0}},\quad a\ge 0.
\end{gather*}
Thus, the $J$-functions of the cones $\mcC$ and $\mcC'$ coincide, which implies that $\mcC=\mcC'$. The proposition is proved.
\end{proof}

\smallskip

\begin{proof}[Proof of Theorem~\ref{theorem:reconstruction}]
Consider the canonical coordinates $u^i(t^*)$ of our flat F-manifold. Making an appropriate shift, we can assume that $u^i(0)=0$. Consider the matrices $H$ and $\Psi$ constructed in Section~\ref{subsection:metric for a flat F-manifold}, and the matrices $R_i$ given by Proposition~\ref{proposition:matrices R_k}. We consider these matrices as functions of the variables $t^\alpha$, $H=H(t^*)$, $\Psi=\Psi(t^*)$, $R_i=R_i(t^*)$. Denote
$$
R(z,t^*):=1+\sum_{i\ge 1}R_i(t^*)z^i.
$$
Let us prove that 
$$
\mcC=\Psi^{-1}(0)R^{-1}(-z,0)\Psi(0)\mcC^\const.
$$
Since $\mcC_N^\triv=H^{-1}(0)\Psi(0)\mcC^\const$, it is sufficient to prove that
$$
\mcC^\triv_N=H^{-1}(0)R(-z,0)\Psi(0)\mcC.
$$

\smallskip

We say that a Laurent series $\sum_{i\in\mbZ}f_i(t^*)z^i$, $f_i\in\mbC[[t^1,\ldots,t^N]]$, is {\it admissible} if for each $i\le 0$ the formal power series $f_i(t^*)$ consists of monomials $t^{\alpha_1}\cdots t^{\alpha_k}$ with $k\ge -i$. Note that the product of any two admissible Laurent series is well defined and is also an admissible Laurent series. Any Laurent series $f$ with matrix coefficients can be considered as a matrix, whose entries are Laurent series, and we say that $f$ is admissible, if all the entries are admissible.\\  

Consider the following Laurent series with matrix coefficients:
\begin{gather*}
S_1(z,t^*):=\sum_{j\ge 0}\Omega^{j-1}_0(t^*)z^{-j},\qquad S_2(z,t^*):=e^{U(t^*)/z}R(z,t^*)\Psi(t^*),
\end{gather*}
where $U(t^*)=\diag(u^1(t^*),\ldots,u^N(t^*))$ and the matrices $\Omega^{j-1}_0(t^*)$ define the ancestor calibration for our flat F-manifold. Since $U(0)=0$, the Laurent series $S_2(z,t^*)$ is well defined and is admissible. By Lemma~\ref{lemma:technical property of ancestor potentials}, the Laurent series $S_1(z,t^*)$ is also admissible and we can consider the product 
$$
S_2(-z,t^*)S_1^{-1}(-z,t^*).
$$
Clearly, $R(-z,0)\Psi(0)=S_2(-z,0)S_1^{-1}(-z,0)$. Define an ancestor cone $\mcC_0$ by
$$
\mcC_0:=H^{-1}(0)S_2(-z,0)S_1^{-1}(-z,0)\mcC.
$$
We have to prove that $\mcC_0=\mcC^\triv_N$.\\

Both Laurent series $S_1$ and $S_2$ satisfy the same differential equation
$$
d S_{1,2}=z^{-1}S_{1,2}d\Omega^0_0.
$$
The equation $d S_1=z^{-1}S_1 d\Omega^0_0$ is a part of the definition of a calibration. In order to prove the equation $d S_2=z^{-1}S_2 d\Omega^0_0$, one should use the formula $d\Omega^0_0=\Psi^{-1}dU\Psi$ together with formulas~\eqref{eq:formulas for dPsi and dGammadU} and~\eqref{eq:equation for R(z)}. This implies that the product $S_2(z,t^*)S_1^{-1}(z,t^*)$ does not depend on the variables~$t^\alpha$. By Proposition~\ref{proposition:another parameterization}, the cone $\mcC_0$ can be parameterized as follows:
$$
\mcC_0=\left\{H^{-1}(0)S_2(-z,0)S_1^{-1}(-z,0)S_1(-z,t^*)\left(-zA^\alpha\phi_\alpha+\sum_{i\ge 1}t^\alpha_i\phi_\alpha z^i\right)\right\}.
$$
Since $S_2(-z,0)S_1^{-1}(-z,0)=S_2(-z,t^*)S_1^{-1}(-z,t^*)$, we get the following parameterization of the cone $\mcC_0$:
$$
\mcC_0=\left\{e^{-U(t^*)/z}H^{-1}(0)R(-z,t^*)\Psi(t^*)\left(-z A^\alpha\phi_\alpha+\sum_{i\ge 1}t^\alpha_i\phi_\alpha z^i\right)\right\}.
$$

\smallskip

Consider a family of vectors $f(z,t^*)\in z\mcH_+$, depending on $t^1,\ldots,t^N$, defined by
$$
f(z,t^*):=\Psi^{-1}(t^*)R^{-1}(-z,t^*)H(0)\left(-z\sum_{\alpha=1}^N\phi_\alpha\right).
$$
This family has the form
$$
f(z,t^*)=-zA^\alpha\phi_\alpha+\sum_{i\ge 1}f^\alpha_i(t^*)\phi_\alpha z^i,\quad f^\alpha_i\in\mbC[[t^*]],\quad f^\alpha_i(0)=0,
$$
and we clearly have
$$
H^{-1}(0)R(-z,t^*)\Psi(t^*)f(z,t^*)=-z\sum_{\alpha=1}^N\phi_\alpha.
$$
Therefore,
$$
e^{-U(t^*)/z}\left(-z\sum_{\alpha=1}^N\phi_\alpha\right)\in\mcC_0.
$$

\smallskip

We see that the cone $\mcC_0$ contains the family of vectors $e^{-U(t^*)/z}\left(-z\sum_{\alpha=1}^N\phi_\alpha\right)$ parameterized by $t^1,\ldots,t^N$. This implies that the $J$-function of the cone $\mcC_0$ is $\sum_{\alpha=1}^N ze^{t^\alpha/z}\phi_\alpha$. Therefore, $\mcC_0=\mcC^{\triv}_N$. This completes the proof of the theorem.
\end{proof}

\smallskip

Combining this theorem with Proposition~\ref{proposition:all from ancestor}, we get the following result.

\smallskip

\begin{theorem}
Any generalized Givental cone $\mcC$ such that the algebra structure of the flat F-manifold at $\ot_{\orig}$ is semisimple can be obtained from the cone $\mcC^\triv_N$ by the following composition of operators:
$$
\mcC=S(z)R(z)M\mcC^\triv_N,
$$
for some $S(z)\in G_-$, $R(z)\in G_+$ and $M\in\GL(\mbC^N)$.\\
\end{theorem}

%%%%%%%%%%%%%%%%%%%%%%%%%%%%%%%%%%%%%%%%%%%%%%%%%%%%%%%%%%%%
%%%%%%%%%%%%%%%%%%%%%%%%%%%%%%%%%%%%%%%%%%%%%%%%%%%%%%%%%%%%

\section{F-cohomological field theories}\label{section:F-CohFT}

F-cohomological field theories (F-CohFTs for short) were introduced in \cite{BR18} as a generalization of the notion of a cohomological field theory (or CohFT) \cite{KM94} and of a partial cohomological field theory \cite{LRZ15}. We recall here their definition and their relation with Frobenius and flat F-manifolds. In what follows, we denote by $\oM_{g,n}$ the Deligne--Mumford moduli space of genus $g$ stable curves with $n$ marked points, where $g,n\geq 0$ and $2g-2+n>0$.

\smallskip

\subsection{F-CohFTs, partial CohFTs and CohFTs}

We will denote by $H^*(X)$ the cohomology ring with coefficients in $\mbC$ of a topological space~$X$. When considering the moduli space of stable curves, $X=\oM_{g,n}$, the cohomology ring $H^{2k}(\oM_{g,n})$ can optionally be replaced by the Chow ring $A^k(\oM_{g,n})$, $k\geq 0$. 

\smallskip

\begin{definition}\label{definition:F-CohFT}
An F-cohomological field theory (or F-CohFT) is a system of linear maps 
$$
c_{g,n+1}\colon V^*\otimes V^{\otimes n} \to H^\even(\oM_{g,n+1}),\quad 2g-1+n>0,
$$
where $V$ is an arbitrary finite dimensional vector space, together with a special element $e\in V$, called the unit, such that, chosen any basis $e_1,\ldots,e_{\dim V}$ of $V$ and the dual basis $e^1,\ldots,e^{\dim V}$ of $V^*$, the following axioms are satisfied:
\begin{itemize}
\item[(i)] the maps $c_{g,n+1}$ are equivariant with respect to the $S_n$-action permuting the $n$ copies of~$V$ in $V^*\otimes V^{\otimes n}$ and the last $n$ marked points in $\oM_{g,n+1}$, respectively.
\item[(ii)] $\pi^* c_{g,n+1}(e^{\alpha_0}\otimes \otimes_{i=1}^n e_{\alpha_i}) = c_{g,n+2}(e^{\alpha_0}\otimes \otimes_{i=1}^n  e_{\alpha_i}\otimes e)$ for $1 \leq\alpha_0,\alpha_1,\ldots,\alpha_n\leq \dim V$, where $\pi\colon\oM_{g,n+2}\to\oM_{g,n+1}$ is the map that forgets the last marked point.\\
Moreover, $c_{0,3}(e^{\alpha}\otimes e_\beta \otimes e) = \delta^\alpha_\beta$ for $1\leq \alpha,\beta\leq \dim V$.
\item[(iii)] $\gl^* c_{g_1+g_2,n_1+n_2+1}(e^{\alpha_0}\otimes \otimes_{i=1}^{n_1+n_2} e_{\alpha_i}) = c_{g_1,n_1+2}(e^{\alpha_0}\otimes \otimes_{i\in I} e_{\alpha_i} \otimes e_\mu)\otimes c_{g_2,n_2+1}(e^{\mu}\otimes \otimes_{j\in J} e_{\alpha_j})$ for $1 \leq\alpha_0,\alpha_1,\ldots,\alpha_{n_1+n_2}\leq \dim V$, where $I \sqcup J = \{2,\ldots,n_1+n_2+1\}$, $|I|=n_1$, $|J|=n_2$, and $\gl\colon\oM_{g_1,n_1+2}\times\oM_{g_2,n_2+1}\to \oM_{g_1+g_2,n_1+n_2+1}$ is the corresponding gluing map.
\end{itemize}
\end{definition}

\smallskip

There is an obvious generalization of the notion of F-CohFT, where the maps $c_{g,n+1}$ take value in $H^\even(\oM_{g,n+1})\otimes K$, where $K$ is a $\mbC$-algebra. We will call such objects {\it F-cohomological field theories with coefficients in $K$}.\\

Given an F-CohFT $c_{g,n+1}\colon V^*\otimes V^{\otimes n} \to H^\even(\oM_{g,n+1})$, $\dim V=N$, and a basis $e_1,\ldots,e_N\in V$, an $N$-tuple of functions $(F^1,\ldots,F^N)$ satisfying equations~\eqref{eq:axiom1 of flat F-man} and~\eqref{eq:axiom2 of flat F-man} can be constructed as the following generating functions:
\begin{equation}\label{eq:genus 0 vector potential of an F-CohFT}
F^\alpha(t^1,\ldots,t^N):=\sum_{n\geq 2}\frac{1}{n!}\sum_{1\leq\alpha_1,\ldots,\alpha_n\leq N}\left(\int_{\oM_{0,n+1}}c_{0,n+1}(e^\alpha\otimes\otimes_{i=1}^n e_{\alpha_i})\right)\prod_{i=1}^n t^{\alpha_i},
\end{equation}
thus yielding an associated flat F-manifold structure on a formal neighbourhood of $0$ in $V$. Note that the unit vector field of the flat F-manifold is $\frac{\d}{\d t^\un}=A^\alpha\frac{\d}{\d t^\alpha}$, where $A^\alpha e_\alpha=e$. More in general, we have the following result involving genus $0$ intersection numbers of the F-CohFT with psi classes, where $\psi_i \in H^2(\oM_{g,n})$, $1\leq i\leq n$, is the first Chern class of the $i$-th tautological bundle on $\oM_{g,n}$ whose fiber at a point representing the class of a marked stable curve $(C,x_1,\ldots,x_n)$ is the cotangent line to $C$ at $x_i\in C$.

\smallskip

\begin{proposition}
For $1\leq \alpha\leq N$ and $a\geq 0$, the formal power series 
\begin{equation}\label{eq:genus 0 ancestor potential of an F-CohFT}
\mcF^{\alpha,a}(t^*_*):=\sum_{n\geq 2}\frac{1}{n!}\sum_{\substack{1\leq\alpha_1,\ldots,\alpha_n\leq N\\ a_1,\ldots,a_n\geq 0}}\left(\int_{\oM_{0,n+1}}c_{0,n+1}(e^\alpha\otimes\otimes_{i=1}^n e_{\alpha_i})\psi_1^{a}\prod_{i=1}^n\psi_{i+1}^{a_i}\right)\prod_{i=1}^n t^{\alpha_i}_{a_i}
\end{equation}
form a sequence of ancestor vector potentials of a flat F-manifold.
\end{proposition}
\begin{proof}
We are going to use the characterization of ancestor vector potentials given by Proposition \ref{proposition:descendant vector potentials,equivalent approach} and equation \eqref{eq:ancestor vector potential}. Therefore, we need to show that equations \eqref{eq:ancestor vector potential}, \eqref{eq:vector string equation}, \eqref{eq:vector dilaton equation}, \eqref{eq:basic vector TRR} and \eqref{eq:important vector relation} are valid for the power series \eqref{eq:genus 0 ancestor potential of an F-CohFT}.\\

Equation \eqref{eq:ancestor vector potential} follows from the fact that the formal power series $\mcF^{\alpha,a}$ are always at least quadratic in the variables $t^*_*$, by definition \eqref{eq:genus 0 ancestor potential of an F-CohFT}.\\

Equations \eqref{eq:vector string equation} and \eqref{eq:vector dilaton equation} can be proved by computing $\frac{\d \mcF^{\alpha,a}}{\d t^\un_0}$ and $\frac{\d \mcF^{\alpha,a}}{\d t^\un_1}$ using Axiom (ii) of Definition~\ref{definition:F-CohFT} together with the equations
\begin{gather}\label{eq:monomial in psi's and forgetful map}
\psi_1^{a_1}\ldots\psi_n^{a_n} =\pi^* (\psi_1^{a_1}\ldots\psi_n^{a_n})+ \sum_{i=1}^n \pi^*(\psi_1^{a_1}\ldots\psi_i^{a_i-1}\ldots\psi_n^{a_n}) \delta^{\{i,n+1\}}_0
\end{gather}
and
$$
\pi_*\psi_{n+1} = 2g-2+n,
$$
respectively, where $\pi\colon\oM_{g,n+1}\to\oM_{g,n}$ is the morphism forgetting the last marked point and~$\delta^{\{i,n+1\}}_0$ is the class of the irreducible boundary divisor in $\oM_{g,n+1}$ that is the closure of the locus of marked stable curves $(C,x_1,\ldots,x_{n+1})$ with two irreducible components, one of genus $0$ carrying the marked points $x_i$ and $x_{n+1}$ and the other of genus $g$ carrying all the other markings.\\

Proving equation~\eqref{eq:basic vector TRR} requires Axiom (iii) of Definition \ref{definition:F-CohFT} together with the following formula for the psi class $\psi_i$, valid in genus $0$ for fixed $i,j,k \in \{1,\ldots,n+1\}$ with $i\neq j\neq k\neq i$:
\begin{gather}\label{eq:psi class in genus 0}
\psi_i = \sum_{\substack{I\subset \{1,\ldots,n+1\}\\|I|\geq 2,\, i\in I,\, j,k \in I^c}} \delta^I_0 \in H^2(\oM_{0,n+1}),
\end{gather}
where $\delta^I_0$ is the closure of the locus of genus $0$ marked stable curves $(C,x_1,\ldots,x_{n+1})$ with two irreducible components, one carrying the marked points with labels in $I$ and the other carrying the marked points with labels in the complement $I^c=\{1,\ldots,n+1\}\backslash I$. Applying the formula to the case $i,j>0$, $k=1$, one obtains equation \eqref{eq:basic vector TRR}.\\

Summing two equations~\eqref{eq:psi class in genus 0}, where $i$ and $j$ are swapped, we obtain
$$
\psi_i+\psi_j = \sum_{\substack{I\subset \{1,\ldots,n+1\}\\|I|,|I^c|\geq 2,\, i\in I,\, j\in I^c}} \delta^I_0 \in H^2(\oM_{0,n+1}),
$$
which implies, taking $i=1$, equation~\eqref{eq:important vector relation}.
\end{proof}

\smallskip
 
\begin{definition}\cite{LRZ15}\label{definition:pCohFT}
A partial CohFT is a system of linear maps 
$$
c_{g,n}\colon V^{\otimes n} \to H^\even(\oM_{g,n}),\quad 2g-2+n>0,
$$
where $V$ is an arbitrary finite dimensional vector space, together with a special element $e\in V$, called the unit, and a symmetric nondegenerate bilinear form $\eta\in (V^*)^{\otimes 2}$, called a metric, such that, chosen any basis $e_1,\ldots,e_{\dim V}$ of $V$, the following axioms are satisfied:
\begin{itemize}
\item[(i)] the maps $c_{g,n}$ are equivariant with respect to the $S_n$-action permuting the $n$ copies of~$V$ in $V^{\otimes n}$ and the $n$ marked points in $\oM_{g,n}$, respectively.
\item[(ii)] $\pi^* c_{g,n}( \otimes_{i=1}^n e_{\alpha_i}) = c_{g,n+1}(\otimes_{i=1}^n  e_{\alpha_i}\otimes e)$ for $1 \leq\alpha_1,\ldots,\alpha_n\leq \dim V$, where $\pi\colon\oM_{g,n+1}\to\oM_{g,n}$ is the map that forgets the last marked point.\\
Moreover $c_{0,3}(e_{\alpha}\otimes e_\beta \otimes e) =\eta(e_\alpha\otimes e_\beta) =:\eta_{\alpha\beta}$ for $1\leq \alpha,\beta\leq \dim V$.
\item[(iii)] $\gl^* c_{g_1+g_2,n_1+n_2}( \otimes_{i=1}^{n_1+n_2} e_{\alpha_i}) = \eta^{\mu \nu}c_{g_1,n_1+1}(\otimes_{i\in I} e_{\alpha_i} \otimes e_\mu)\otimes c_{g_2,n_2+1}(\otimes_{j\in J} e_{\alpha_j}\otimes e_\nu)$ for $1\leq\alpha_1,\ldots,\alpha_{n_1+n_2}\leq \dim V$, where $I \sqcup J = \{1,\ldots,n_1+n_2\}$, $|I|=n_1$, $|J|=n_2$, and $\gl\colon\oM_{g_1,n_1+1}\times\oM_{g_2,n_2+1}\to \oM_{g_1+g_2,n_1+n_2}$ is the corresponding gluing map and where $\eta^{\alpha\beta}$ is defined by $\eta^{\alpha \mu}\eta_{\mu \beta} = \delta^\alpha_\beta$ for $1\leq \alpha,\beta\leq \dim V$.
\end{itemize}
\end{definition}

\smallskip

\begin{remark}
Clearly, given a partial CohFT $c_{g,n}\colon V^{\otimes n} \to H^\even(\oM_{g,n})$, the system of linear maps $c^\bullet_{g,n+1}\colon V^*\otimes V^{\otimes n}\to H^\even(\oM_{g,n+1})$ 
defined as $c^\bullet_{g,n+1}(e^{\alpha_0}\otimes\otimes_{i=1}^n e_{\alpha_i}):=\eta^{\alpha_0\mu}c_{g,n+1}(e_\mu\otimes\otimes_{i=1}^n e_{\alpha_i})$ forms an F-CohFT, called the associated F-CohFT.
\end{remark}

\smallskip

Given a partial CohFT $c_{g,n}\colon V^{\otimes n} \to H^\even(\oM_{g,n})$, $\dim V=N$, and a basis $e_1,\ldots,e_N\in V$, a function $F(t^1,\ldots,t^N)$ satisfying equations~\eqref{eq:axiom1 for Frobenius manifolds} and~\eqref{eq:axiom2 for Frobenius manifolds} can be constructed as the following generating function:
$$
F(t^1,\ldots,t^N):=\sum_{n\geq 3}\frac{1}{n!}\sum_{1\leq\alpha_1,\ldots,\alpha_n\leq N}\left(\int_{\oM_{0,n}}c_{0,n}(\otimes_{i=1}^n e_{\alpha_i})\right)\prod_{i=1}^n t^{\alpha_i},
$$
thus yielding an associated Frobenius manifold structure on a formal neighbourhood of $0$ in $V$.

\smallskip

\begin{definition}[\cite{KM94}]\label{definition:CohFT}
A CohFT is a partial CohFT $c_{g,n}\colon V^{\otimes n} \to H^\even(\oM_{g,n})$ such that the following extra axiom is satisfied:
\begin{itemize}
\item[(iv)] $\gl^* c_{g+1,n}(\otimes_{i=1}^n e_{\alpha_i}) = c_{g,n+2}(\otimes_{i=1}^n e_{\alpha_i}\otimes e_{\mu}\otimes e_\nu) \eta^{\mu \nu}$ for $1 \leq\alpha_1,\ldots,\alpha_n\leq \dim V$, where  $\gl\colon\oM_{g,n+2}\to \oM_{g+1,n}$ is the gluing map, which increases the genus by identifying the last two marked points.
\end{itemize}
\end{definition}

\smallskip

\subsection{Formal shift of an F-CohFT}

Let $\tau^1,\ldots,\tau^N$ be formal variables. For an arbitrary F-CohFT $c_{g,n+1}\colon V^*\otimes V^{\otimes n}\to H^\even(\oM_{g,n+1})$ with a fixed basis $e_1,\ldots,e_N\in V$ consider a system of maps  
$$
c_{g,n+1}^\otau\colon V^*\otimes V^{\otimes n}\to H^\even(\oM_{g,n+1})\otimes\mbC[[\tau^1,\ldots,\tau^N]]
$$ 
defined by
\begin{gather}\label{eq:formal shift of an F-CohFT}
c^\otau_{g,n+1}(\omega\otimes \otimes_{i=1}^n v_i):=\sum_{m\geq 0} \frac{1}{m!}\pi_{m*} c_{g,n+m+1}(\omega\otimes \otimes_{i=1}^n v_i \otimes (\tau^\alpha e_\alpha)^{\otimes m}),
\end{gather}
where $\omega\in V^*$, $v_i\in V$, $\otau=(\tau^1,\ldots,\tau^N)$ and $\pi_m\colon\oM_{g,n+m+1}\to\oM_{g,n+1}$ is the map that forgets the last $m$ marked points. It is straightforward to check that the maps $c^\otau_{g,n+1}$ form an F-CohFT with coefficients in~$\mbC[[\tau^1,\ldots,\tau^N]]$. Moreover, if $\oF$ is the vector potential given by formula~\eqref{eq:genus 0 vector potential of an F-CohFT}, then the flat F-manifold of the F-CohFT $c^\otau$ is described by the vector potential~$\oF_\otau$ from Section~\ref{subsubsection:lower triangular group}. We will call the F-CohFT $c^\otau_{g,n+1}$ the {\it formal shift} of the F-CohFT $c_{g,n+1}$.

\smallskip

\subsection{Partial and F-topological field theories}
A CohFT $c_{g,n}\colon V^{\otimes n}\to H^\even(\oM_{g,n})$ is called a {\it topological field theory} (TFT) when $c_{g,n}(V^{\otimes n})\subset H^0(\oM_{g,n})$. In this case, each class $c_{g,n}(\otimes_{i=1}^n e_{\alpha_i})\in H^0(\oM_{g,n})$ is represented by a constant complex valued function on $\oM_{g,n}$. Evaluating such functions at a point of $\oM_{g,n}$ representing a maximally degenerate stable curve, i.e., a nodal curve whose irreducible components are all $\mbP^1$'s with three special (nodal or marked) points, using Axioms~(iii) and~(iv) of Definition~\ref{definition:CohFT}, we see that the entire CohFT can be reconstructed from the linear map $c_{0,3}\colon V^{\otimes 3}\to H^0(\oM_{0,3})$ only. This amounts to the datum of a Frobenius algebra with unit $(V,\bullet,\eta,e)$, where the metric is given by $\eta(e_\alpha,e_\beta)= c_{0,3}(e\otimes e_\alpha\otimes e_\beta)$ and the product is given by $e_{\alpha}\bullet e_{\beta} = c_{0,3}(e_\alpha\otimes e_\beta\otimes e_{\mu})\eta^{\mu \nu}e_{\nu}$.\\

Let us generalize these notions and observations to partial and F-CohFTs.

\smallskip

\begin{definition}
A partial CohFT $c_{g,n}\colon V^{\otimes n}\to H^\even(\oM_{g,n})$ is called a partial topological field theory (partial TFT) when $c_{g,n}(V^{\otimes n})\subset H^0(\oM_{g,n})$.

\smallskip

\noindent An F-CohFT $c_{g,n+1}\colon V^*\otimes V^{\otimes n}\to H^\even(\oM_{g,n+1})$ is called an F-topological field theory (F-TFT) when $c_{g,n+1}(V^*\otimes V^{\otimes n})\subset H^0(\oM_{g,n+1})$.
\end{definition}

\smallskip

\begin{proposition}\label{proposition:reconstruction of F-TFTs}
A partial TFT $c_{g,n}\colon V^{\otimes n}\to H^0(\oM_{g,n})$ can be uniquely reconstructed from the linear maps $c_{0,3}\colon V^{\otimes 3}\to H^0(\oM_{0,3})$ and $c_{1,1}\colon V\to H^0(\oM_{1,1})$ only. This amounts to the datum of a Frobenius algebra with unit $(V,\bullet,\eta,e)$, where the metric is given by $\eta(e_\alpha,e_\beta)= c_{0,3}(e\otimes e_\alpha\otimes e_\beta)$ and the product is given by $e_{\alpha}\bullet e_{\beta} = c_{0,3}(e_\alpha\otimes e_\beta\otimes e_{\mu})\eta^{\mu \nu}e_{\nu}$, together with a special covector $\omega\in V^*$ given by $\langle\omega,e_{\alpha}\rangle = c_{1,1}(e_\alpha)$.

\smallskip

\noindent An F-TFT $c_{g,n+1}\colon V^*\otimes V^{\otimes n}\to H^0(\oM_{g,n+1})$ can be uniquely reconstructed from the linear maps $c_{0,3}\colon V^*\otimes V^{\otimes 2}\to H^0(\oM_{0,3})$ and $c_{1,1}\colon V^*\to H^0(\oM_{1,1})$ only. This amounts to the datum of a commutative associative algebra with unit $(V,\bullet,e)$, where the product is given by $e_{\alpha}\bullet e_{\beta} = c_{0,3}(e^\mu \otimes e_\alpha\otimes e_\beta)e_{\mu}$, together with a special vector $w\in V$ given by $w=c_{1,1}(e^\alpha)e_\alpha$.
\end{proposition}
\begin{proof}
As for TFTs, the classes $c_{g,n}(\otimes_{i=1}^n e_{\alpha_i})\in H^0(\oM_{g,n})$ for partial TFTs and $c_{g,n+1}(e^{\alpha}\otimes \otimes_{i=1}^n e_{\alpha_i})\in H^0(\oM_{g,n+1})$ for F-TFTs are represented by constant complex valued functions on $\oM_{g,n}$ and $\oM_{g,n+1}$, respectively. Evaluating such functions at points of the moduli spaces representing stable curves with separating nodes only and whose irreducible components are either $\mbP^1$'s with three special (marked or nodal) points or elliptic curves with one special point, and using Axiom (iii) of Definition \ref{definition:pCohFT} or \ref{definition:F-CohFT} we obtain the desired result. 
\end{proof}

\smallskip

Given a CohFT, a partial CohFT or an F-CohFT, their degree zero parts are naturally a TFT, a partial TFT or an F-TFT, respectively.

\smallskip

\subsection{Homogeneous F-CohFTs}

Since $H^*(\oM_{g,n})$ is a graded $\mbC$-vector space, it is natural to consider the special case of F-CohFTs for which:
\begin{itemize}
\item the vector spaces $V$ and $V^*$ are also graded, $\deg e = 0$, and the pairing between $V$ and~$V^*$ has degree $0$, i.e., $\deg e^\alpha =-\deg e_\alpha$ for a homogeneous basis $e_1,\ldots,e_{\dim V}$ of $V$,
\item the maps $c_{g,n+1}\colon V^*\otimes V^{\otimes n} \to H^\even(\oM_{g,n})$ are homogeneous of degree $\deg c_{g,n+1}$.
\end{itemize}
Then, because of Axiom (ii) in Definition \ref{definition:F-CohFT}, $\deg c_{g,n+1}$ does not depend on $n$ and $\deg c_{0,n+1}=0$ for any $n\geq 2$. Moreover, because of Axiom (iii), $\deg c_{g,n+1}$ is a linear function of $g$, which implies that the general form of a grading compatible with the axioms of F-CohFT is
$$
\deg c_{g,n+1} = \gamma g , \quad \gamma \in \mbC.
$$
Thus, setting $q_\alpha:=\deg e_\alpha$, we get the following condition for the classes $c_{g,n+1}(e^{\alpha_0}\otimes\otimes_{i=1}^ne_{\alpha_i})$:
\begin{gather}\label{eq:homogeneous F-CohFT without correction}
\deg c_{g,n+1}(e^{\alpha_0}\otimes\otimes_{i=1}^ne_{\alpha_i})=\sum_{i=1}^n q_{\alpha_i}-q_{\alpha_0}+\gamma g,
\end{gather}
where $\deg$ denotes half of the cohomological degree. Now, in order to get a generalization of the usual notion of homogeneous CohFT, let us correct the left-hand side of~\eqref{eq:homogeneous F-CohFT without correction} by adding a term $\pi_{1*}c_{g,n+2}(e^{\alpha_0}\otimes\otimes_{i=1}^ne_{\alpha_i}\otimes r^\beta e_\beta)$, where $r^\beta\in\mbC$ and $\pi_1\colon\oM_{g,n+2}\to\oM_{g,n+1}$ is the map that forgets the last marked point. We finally arrive to the following definition.

\smallskip

\begin{definition}
An F-CohFT $c_{g,n+1}\colon V^*\otimes V^{\otimes n}\to H^\even(\oM_{g,n+1})$ is called homogeneous if~$V$ is a graded vector space with a homogeneous basis $e_1,\ldots,e_{\dim V}$, $\deg e=0$, and complex constants~$r^\alpha$, $1\leq \alpha\leq \dim V$, and $\gamma$ exist, such that the following condition is satisfied:
\begin{multline*}
\Deg c_{g,n+1}(e^{\alpha_0}\otimes\otimes_{i=1}^ne_{\alpha_i})+\pi_{1*}c_{g,n+2}(e^{\alpha_0}\otimes\otimes_{i=1}^ne_{\alpha_i}\otimes r^\beta e_\beta)=\\
=\left(\sum_{i=1}^n q_{\alpha_i}-q_{\alpha_0}+\gamma g\right)c_{g,n+1}(e^{\alpha_0}\otimes\otimes_{i=1}^ne_{\alpha_i}),
\end{multline*}
where $q_\alpha:=\deg e_\alpha$ and by $\Deg\colon H^*(\oM_{g,n})\to H^*(\oM_{g,n})$ we denote the operator that acts on $H^i(\oM_{g,n})$ by the multiplication by $\frac{i}{2}$. The constant $\gamma$ is called the conformal dimension of our F-CohFT.
\end{definition}

\smallskip

The flat F-manifold associated to a homogeneous F-CohFT is homogeneous with the Euler vector field given by
$$
E=\sum_{\alpha=1}^{\dim V}\left((1-q_\alpha)t^\alpha+r^\alpha\right)\frac{\d}{\d t^\alpha}.
$$

\smallskip

Suppose that a homogeneous F-CohFT comes from a partial CohFT $c_{g,n}\colon V^{\otimes n} \to H^\even(\oM_{g,n})$, with the metric $\eta$ on $V$, seen as the map $\eta\colon V^{\otimes 2}\to \mbC$, having degree $\deg \eta = -\delta$. Then our partial CohFT satisfies the condition
\begin{gather*}
\Deg c_{g,n}(\otimes_{i=1}^ne_{\alpha_i})+\pi_{1*}c_{g,n+1}(\otimes_{i=1}^ne_{\alpha_i}\otimes r^\beta e_\beta)=\left(\sum_{i=1}^n q_{\alpha_i}+\gamma g-\delta\right)c_{g,n}(\otimes_{i=1}^ne_{\alpha_i}).
\end{gather*}

\smallskip

Finally, if our partial CohFT is a CohFT, then the last property has to be compatible with the extra gluing axiom at nonseparating nodes and this imposes the further condition $\gamma=\delta$, i.e.,
\begin{gather*}
\Deg c_{g,n}(\otimes_{i=1}^ne_{\alpha_i})+\pi_{1*}c_{g,n+1}(\otimes_{i=1}^ne_{\alpha_i}\otimes r^\beta e_\beta)=\left(\sum_{i=1}^n q_{\alpha_i}+\delta(g-1)\right)c_{g,n}(\otimes_{i=1}^ne_{\alpha_i}).
\end{gather*}
This is exactly the homogeneity condition in the definition of homogeneous CohFTs (see, e.g.,~\cite[Definition~1.7]{PPZ15}). Note that the constant $\delta$ is the conformal dimension of the corresponding Frobenius manifold.\\
 
%%%%%%%%%%%%%%%%%%%%%%%%%%%%%%%%%%%%%%%%%%%%%%%%%%%%%%%%%
%%%%%%%%%%%%%%%%%%%%%%%%%%%%%%%%%%%%%%%%%%%%%%%%%%%%%%%%%

\section{Group action on F-CohFTs}\label{section:group action of F-CohFTs}

In this section, we define a generalization of the notion of a Givental group acting on the space of CohFTs (see \cite{PPZ15}) to a corresponding generalized Givental group acting on the space of F-CohFTs. Using this action, we then present a construction of a family of F-CohFTs associated to any given flat F-manifold that is semisimple at the origin. The family is parameterized by a vector $G_0\in\mbC^N$.

\smallskip

\subsection{$R$-matrices}
\begin{definition}
Given a vector space $V$, a system of linear maps 
$$
c_{g,n+1}\colon V^*\otimes V^{\otimes n} \to H^\even(\oM_{g,n+1}),\quad 2g-1+n>0,
$$
satisfying Axioms (i) and (iii) of Definition~\ref{definition:F-CohFT} is called an F-CohFT without unit.
\end{definition}

\smallskip

Consider now the group $G_+$ of $\End(V)$-valued power series of the form $R(z)=\Id+\sum_{i\geq 1} R_iz^i$, and let us denote by $R^{-1}(z)$ the inverse element to $R(z)$ and by $R(z)^t$ the transposed $\End(V^*)$-valued power series. We refer to such an element of $G_+$ as an {\it $R$-matrix}.\\

Let $\Gamma$ be a stable graph of genus $g$ with $n$ marked legs (see \cite[Section~0.2]{PPZ15} for the definition) and $V(\Gamma)$, $E(\Gamma)$ be its sets of vertices and edges, each vertex $v\in V(\Gamma)$ marked with a genus $g(v)$ and with valence $n(v)$. Let $\xi_{\Gamma}\colon\prod_{v\in V(\Gamma)} \oM_{g(v),n(v)} \to \oM_{g,n}$ be the natural map whose image is the closure of the locus of stable curves whose dual graph is $\Gamma$. The degree of~$\xi_\Gamma$ is~$|\Aut\Gamma|$, the number of automorphisms of the graph $\Gamma$.\\

Let $T_{g,n+1}$ be the set of stable trees of genus $g$ with $n+1$ marked legs. Then $\Gamma \in T_{g,n+1}$ can be seen as a stable rooted tree where the root is the vertex to which leg $1$ is attached and each edge $e\in E(\Gamma)$ is splitted into two half edges $e'$ and $e''$, where $e'$ is closer to the root and $e''$ is farther from the root.\\

The action of $R\in G_+$ on an F-CohFT without unit $c_{g,n+1}\colon V^*\otimes V^{\otimes n} \to H^\even(\oM_{g,n+1})$ is the system of maps
\begin{equation}\label{eq:R-action}
(Rc)_{g,n+1} := \sum_{\Gamma \in T_{g,n+1}}\xi_{\Gamma*}\left[\prod_{v\in V(\Gamma)}c_{g(v),n(v)}R(-\psi_1)^t \prod_{k=2}^{n+1} R^{-1}(\psi_k)\prod_{e\in E(\Gamma)} \frac{\Id-R^{-1}(\psi_{e'})R(-\psi_{e''})}{\psi_{e'}+\psi_{e''}}\right].
\end{equation}
Understanding how this formula gives a linear map from $V^*\otimes V^{\otimes n}$ to $H^\even(\oM_{g,n+1})$ requires some explanation. First, the covector and the $n$ vectors are fed to the external leg terms~$R(-\psi_1)^t$ and $R^{-1}(\psi_k)$, $2\leq k \leq n+1$, which are elements of $H^*(\oM_{g,n+1})\otimes \End(V^*)$ and $H^*(\oM_{g,n+1})\otimes \End(V)$, respectively. The result is an element in $H^*(\oM_{g,n+1})\otimes V^*$ and~$n$ elements in $H^*(\oM_{g,n+1})\otimes V$, the first factor of which acts by multiplication in cohomology, while the second factor is fed to the $c_{g(v),n(v)}$ sitting at the vertex $v$ to which the corresponding leg is attached.\\

Second, the edge term
$$
\frac{\Id-R^{-1}(\psi_{e'})R(-\psi_{e''}) }{\psi_{e'}+\psi_{e''}}
$$
is an element of $H^*(\oM_{g(v'),n(v')})\otimes H^*(\oM_{g(v''),n(v'')})\otimes V\otimes V^*$, where $v'$ is the vertex to which~$e'$ is attached and $v''$ is the vertex to which $e''$ is attached (indeed, the $\End(V)$-valued power series in the psi classes at the numerator is in the ideal generated by the denominator, since $R(z)R^{-1}(z)=\Id$). The first two factors act by multiplication in cohomology, the third factor is fed to one of the vector entries of the $c_{g(v'),n(v')}$ sitting at the vertex $v'$ and the fourth factor is fed to the covector entry of the $c_{g(v''),n(v'')}$ sitting at the vertex $v''$.\\

This way, all entries of the vertex terms $c_{g(v),n(v)}$ are exhausted by either a leg or an edge term, and all that is left is a product of (even) cohomology classes.

\smallskip

\begin{remark}
Note that, unlike in the analogous formula for the $R$-action on CohFTs without unit~\cite[Section~2.1]{PPZ15}, we don't have the factor $\frac{1}{|\Aut(\Gamma)|}$ in formula~\eqref{eq:R-action}. This is because stable trees don't have nontrivial automorphisms.
\end{remark}

\smallskip

\begin{theorem}
If $c_{g,n+1}$ is an F-CohFT without unit, then $(Rc)_{g,n+1}$ is an F-CohFT without unit. The resulting action is a left group action.
\end{theorem}
\begin{proof}
The proof strictly follows the ideas of the analogous proofs, found in \cite[Section~2]{PPZ15}, for the $R$-matrix action on a CohFT. The $S_n$-equivariance of $(RC)_{g,n+1}$ follows from the $S_n$-equivariance of $c_{g,n+1}$ and the definition of the $R$-matrix action. For the pullback of $(Rc)_{g,n+1}$ to a boundary divisor of curves with dual graph $\Phi$ with two vertices and one separating edge, one follows the argument in the proof of \cite[Proposition 2.3]{PPZ15} to show that such pullback is an expression similar to \eqref{eq:R-action}, the only differences being that the sum runs over all stable rooted trees $\Gamma$ which are degenerations of $\Phi$ and that the edge term assigned to the distinguished separating edge $e$ is simply $R^{-1}(\psi_{e'})R(-\psi_{e''})$. This shows that $(Rc)_{g,n+1}$ satisfies Axiom (iii) in Definition \ref{definition:F-CohFT}.\\

To show that the resulting action is a left group action, we follow again the argument in \cite[Proposition 2.4]{PPZ15}.
\end{proof}

\smallskip

\subsection{$R$-matrix action on F-CohFTs}
Consider the abelian group of $V$-valued power series of the form $T(z)=\sum_{i\geq 2} T_i z^i$. We refer to such power series as a {\it translation}. Its action on an F-CohFT without unit $c_{g,n+1}$ is given by the formula
\begin{equation}\label{eq:T-action}
(Tc)_{g,n+1}(\omega\otimes \otimes_{i=1}^n v_i): = \sum_{m\geq 0} \frac{1}{m!} \pi_{m*} c_{g,n+m+1}(\omega\otimes \otimes_{i=1}^n v_i \otimes \otimes_{k=n+2}^{n+m+1} T(\psi_k) ),
\end{equation}
where $\omega\in V^*$ and $v_i\in V$, $1\leq i \leq n$.

\smallskip

\begin{theorem}
If $c_{g,n+1}$ is an F-CohFT without unit, then $(Tc)_{g,n+1}$ is an F-CohFT without unit. The resulting action is an abelian group action. If two translations $T'$ and $T''$ are related by an $R$-matrix $R$ via the equation $T'(z)=R(z)T''(z)$, then $(T'Rc)_{g,n+1} = (RT''c)_{g,n+1}$.
\end{theorem}
\begin{proof}
The proof follows closely the proofs of \cite[Proposition 2.7, 2.8, 2.9]{PPZ15} with the obvious transpositions. 
\end{proof}

\smallskip

In order to obtain a well-defined $R$-matrix action on F-CohFTs with unit, we need to combine it with an appropriate translation. Given an F-CohFT with unit $e$ and an $R$-matrix $R$, let us define the translations $T_R'(z):=z[R(z)e-e]$ and $T_R''(z):=z[e-R^{-1}(z)e]$.

\smallskip

\begin{theorem}\label{theorem:R-action on F-CohFTs}
If $c_{g,n+1}$ is an F-CohFT with unit $e$, then $(T_R'Rc)_{g,n+1}=(RT_R''c)_{g,n+1}$ is an F-CohFT with the same unit $e$. The resulting action is a left group action of the group $G_+$ on F-CohFTs with unit.
\end{theorem}
\begin{proof}
Again, the proof follows closely the proof of \cite[Proposition 2.12]{PPZ15}.
\end{proof}

\smallskip

Following~\cite{PPZ15}, let us use the notation
$$
R.c:=T_R'Rc
$$
for the constructed action of the group $G_+$ on F-CohFTs with unit.

\smallskip

\begin{proposition}\label{proposition:R-invariance of F-TFTs}
The action of $G_+$ on F-CohFTs with unit leaves their degree $0$ part unchanged.
\end{proposition}
\begin{proof}
From equations \eqref{eq:R-action} and \eqref{eq:T-action} it is easy to see that both $(Rc)_{g,n+1} - c_{g,n+1}$ and $(Tc)_{g,n+1}-c_{g,n+1}$, seen as elements of $H^*(\oM_{g,n+1})\otimes V\otimes (V^*)^{\otimes n}$, have no cohomological degree $0$ term.
\end{proof}

\smallskip

The constructed action of the group $G_+$ on F-CohFTs (with unit) induces a $G_+$-action on the corresponding ancestor vector potentials. Let us prove that the latter coincides with the $G_+$-action on ancestor vector potentials constructed in Section~\ref{subsubsection:upper triangular group}.

\smallskip

\begin{theorem}\label{theorem:consistency of R-actions}
Consider an F-CohFT (with unit) $c_{g,n+1}\colon V^*\otimes V^{\otimes n}\to H^\even(\oM_{g,n+1})$ and choose a basis $e_1,\ldots,e_N\in V$. Consider the corresponding sequence of ancestor vector potentials $\overline{\mcF}^a$, $a\ge 0$, and an $R$-matrix $R\in G_+$. Then the sequence of ancestor vector potentials corresponding to the F-CohFT $(R.c)_{g,n+1}$ coincides with the sequence $\overline{R.\mcF}^a$, $a\ge 0$. 
\end{theorem}
\begin{proof}
It is sufficient to check the statement of the theorem infinitesimally, i.e., to prove that for any $r=\sum_{i\ge 1}r_i z^i$, $r_i\in\End(V)$, we have
\begin{align*}
&\sum_{n\geq 2}\frac{1}{n!}\sum_{\substack{1\leq\alpha_1,\ldots,\alpha_n\leq N\\ a_1,\ldots,a_n\geq 0}}\left(\int_{\oM_{0,n+1}}\left.\frac{d}{d\eps}\left[(e^{\eps r}.c)_{0,n+1}(e^\alpha\otimes\otimes_{i=1}^n e_{\alpha_i})\right]\right|_{\eps=0}\psi_1^{a}\prod_{i=1}^n\psi_{i+1}^{a_i}\right)\prod_{i=1}^n t^{\alpha_i}_{a_i}=\\
=&\left.\frac{d}{d\eps}(e^{\eps r}.\mcF)^{\alpha,a}\right|_{\eps=0},\quad 1\le\alpha\le N,\quad a\ge 0.
\end{align*}
Directly from the definition of the $R$-action on F-CohFTs, it is easy to see that
\begin{align*}
&\left.\frac{d}{d\eps}\left[(e^{\eps r}.c)_{0,n+1}(e^\alpha\otimes\otimes_{i=1}^n e_{\alpha_i})\right]\right|_{\eps=0}=\\
=&\sum_{k\ge 1}(-1)^k(r_k)^\alpha_\mu c_{0,n+1}(e^\mu\otimes\otimes_{i=1}^n e_{\alpha_i})\\
&-\sum_{i=1}^n\sum_{k\ge 1}\psi_{i+1}^k(r_k)^\mu_{\alpha_i}c_{0,n+1}\left(e^\alpha\otimes \otimes_{j=1}^{i-1}e_{\alpha_j}\otimes e_\mu\otimes\otimes_{j=i+1}^n e_{\alpha_j}\right)\\
&+\sum_{\substack{I\sqcup J=\{1,\ldots,n\}\\|I|\ge 1,\,|J|\ge 2}}\sum_{p,q\ge 0}(-1)^q(r_{p+q+1})^\mu_\nu\gl_*\left(\psi_{|I|+2}^p c_{0,|I|+2}(e^{\alpha}\otimes \otimes_{i\in I} e_{\alpha_i} \otimes e_\mu)\otimes\psi_1^q c_{0,|J|+1}(e^{\nu}\otimes \otimes_{j\in J} e_{\alpha_j})\right)\\
&+\sum_{k\ge 1}(r_k)^\mu_\un\pi_{1*}\left(\psi_{n+2}^{k+1}c_{0,n+2}(e^\alpha\otimes\otimes_{i=1}^n e_{\alpha_i}\otimes e_\mu)\right),
\end{align*}
where $\gl\colon\oM_{0,|I|+2}\times\oM_{0,|J|+1}$ is the gluing map. Multiplying the right-hand side by $\psi_1^{a}\prod_{i=1}^n\psi_{i+1}^{a_i}$, integrating over $\oM_{0,n+1}$, and taking the generating series, we obtain exactly the expression on the right-hand side of formula~\eqref{eq:infinitesimal R-action}. This completes the proof of the theorem. 
\end{proof}

\smallskip

Define a $\GL(V)$-action on an F-CohFT $c_{g,n+1}\colon V^*\otimes V^{\otimes n}\to H^\even(\oM_{g,n+1})$ by
$$
(M.c)_{g,n+1}(\omega\otimes\otimes_{i=1}^n v_i):=c_{g,n+1}(M^t\omega\otimes\otimes_{i=1}^n M^{-1}v_i),\quad M\in\GL(V),
$$
where $\omega\in V^*$ and $v_i\in V$. Clearly, if $e\in V$ is the unit of the F-CohFT $c_{g,n+1}$, then $Me\in V$ is the unit of the F-CohFT $(M.c)_{g,n+1}$. The following proposition says that this $\GL(V)$-action on F-CohFTs is consistent with the $\GL(\mbC^N)$-action on descendant vector potentials defined in Section~\ref{subsection:calibrated flat F-manifold}.

\smallskip

\begin{proposition}
Consider an F-CohFT $c_{g,n+1}\colon V^*\otimes V^{\otimes n}\to H^\even(\oM_{g,n+1})$ and choose a basis $e_1,\ldots,e_N\in V$. Consider the corresponding sequence of ancestor vector potentials $\overline{\mcF}^a$, $a\ge 0$, and an element $M\in\GL(V)=\GL(\mbC^N)$. Then the sequence of ancestor vector potentials corresponding to the F-CohFT $(M.c)_{g,n+1}$ coincides with the sequence $\overline{M.\mcF}^a$, $a\ge 0$. 
\end{proposition}
\begin{proof}
Direct computation.
\end{proof}

\smallskip

For an F-CohFT $c_{g,n+1}\colon V^*\otimes V^{\otimes n}\to H^\even(\oM_{g,n+1})$, an element $M\in\GL(V)$ and an $R$-matrix $R$ we will use the notation 
$$
MR.c:=M.(R.c).
$$

\smallskip

\subsection{Construction of F-CohFTs from semisimple flat F-manifolds}
Recall from Section \ref{section:F-CohFT} that to an F-CohFT we associated a flat F-manifold whose vector potential is given by equation~\eqref{eq:genus 0 vector potential of an F-CohFT} and only involves genus $0$ intersection numbers. Recall, moreover, from Proposition~\ref{proposition:reconstruction of F-TFTs} that the degree $0$ part of an F-CohFT is an F-TFT and, as such, it can be uniquely reconstructed from the datum of an associative unital algebra together with an element of such algebra. Such associative algebra coincides, by definition, with the one on the tangent space at the origin of the associated genus $0$ flat F-manifold, while its special element is genus $1$ information.

\smallskip

\begin{theorem}\label{theorem:F-CohFT associated to a flat F-manifold}
Consider a flat F-manifold given by a vector potential $\overline{F}=(F^1,\ldots,F^N)$, where $F^\alpha\in \mbC[[t^1,\ldots,t^N]]$, that is semisimple at the origin $t^*=0$. Let $G_0=G_0^\alpha \frac{\d}{\d t^\alpha}$, where $G_0^\alpha \in \mbC$, be an element of its tangent space at the origin. There exists an F-CohFT whose associated flat F-manifold is the one considered and whose degree $0$ part is the F-TFT defined by the associative unital algebra on the tangent space at the origin of this F-manifold together with the element $G_0$ of this algebra.
\end{theorem}
\begin{proof}
Let us start with the F-TFT reconstructed uniquely from the unital associative algebra at the origin $t^*=0$ of the F-manifold and its element $G_0$, according to Proposition \ref{proposition:reconstruction of F-TFTs}. It is defined on the $N$ dimensional $\mbC$-vector space generated by $e_\alpha=\frac{\d}{\d t^\alpha}$, $1\leq\alpha\leq N$, with unit~$e$ and structure constants $\left.\frac{\d^2 F^\alpha}{\d t^\beta \d t^\gamma}\right|_{t^*=0}$, $1\leq \alpha,\beta,\gamma\leq N$, with respect to the basis $e_1,\ldots,e_n$. Since an F-TFT is an F-CohFT, we can associate to it a flat F-manifold and a sequence of ancestor vector potentials. By definition, it coincides with the constant part of the starting flat F-manifold, and its ancestor cone will be denoted by $\mcC^\const \subset \mcH$, according to the terminology introduced in Section~\ref{section:reconstruction}.\\

Thanks to semisimplicity at the origin, according to Theorem~\ref{theorem:reconstruction}, there exists an $R$-matrix $R(z)\in G_+$ acting on $\mcC^\const$ to produce the ancestor cone $\mcC$ of the starting flat F-manifold.\\

We then make this $R$-matrix $R(z)$ act on the F-TFT via the action described in Theorem~\ref{theorem:R-action on F-CohFTs}, obtaining an F-CohFT. By Theorem~\ref{theorem:consistency of R-actions}, the flat F-manifold associated to this F-CohFT coincides with the starting F-manifold. Proposition~\ref{proposition:R-invariance of F-TFTs} ensures that the degree~$0$ part is the starting F-TFT.
\end{proof}

\smallskip

\subsection{Homogeneous F-CohFTs corresponding to homogeneous flat F-manifolds}

Here we present a construction of a family of homogeneous F-CohFTs associated to any given homogeneous flat F-manifold that is semisimple at the origin.\\

Before considering the homogeneous case, let us discuss the construction from the proof of Theorem~\ref{theorem:F-CohFT associated to a flat F-manifold} in more details. So, let us consider a flat F-manifold given by a vector potential $\overline{F}=(F^1,\ldots,F^N)$, where $F^\alpha\in \mbC[[t^1,\ldots,t^N]]$, that is semisimple at the origin $t^*=0$, and the associated ancestor cone~$\mcC$. As in the proof of Theorem~\ref{theorem:reconstruction}, we consider the canonical coordinates $u^i(t^*)$ and the matrix $H=\diag(H_1,\ldots,H_N)$ constructed in Section~\ref{subsection:metric for a flat F-manifold}. Recall that the functions $H_i$ were defined uniquely up to the rescalings $H_i\mapsto\lambda_iH_i$, where the constants $\lambda_i\in\mbC^*$ can be chosen arbitrarily. Let us make a unique choice such that $H_i|_{t^*=0}=1$ for any $i$. Consider then the matrix~$\Psi$ and the matrices $R_i$ given by Proposition~\ref{proposition:matrices R_k}. We consider these matrices as functions of the variables $t^\alpha$, $H=H(t^*)$, $\Psi=\Psi(t^*)$, $R_i=R_i(t^*)$. From the proof of Theorem~\ref{theorem:reconstruction} we know that $\mcC=\Psi^{-1}(0)R^{-1}(-z,0)\Psi(0)\mcC^\const$ or, equivalently,
$$
\mcC=\Psi^{-1}(0)R^{-1}(-z,0)\mcC^\triv_N.
$$

\smallskip

Let $V=\mbC^N$ and $e_1,\ldots,e_N\in\mbC^N$ be the standard basis in $\mbC^N$. A family of F-TFTs $c^{\triv,G_0}_{g,n+1}\colon V^*\otimes V^{\otimes n}\to H^0(\oM_{g,n+1})$, parameterized by a vector $G_0=(G_0^1,\ldots,G_0^N)\in\mbC^N$, corresponding to the trivial flat F-manifold of dimension $N$ is given by
$$
c^{\triv,G_0}_{g,n+1}(e^{i_0}\otimes\otimes_{j=1}^n e_{i_j}):=
\begin{cases}
(G_0^{i_0})^g,&\text{if $i_0=i_1=\ldots=i_n$},\\
0,&\text{otherwise}.
\end{cases}
$$
Note that $c^{\triv,G_0}_{1,1}(e^i)=G_0^i$. By the proof of Theorem~\ref{theorem:F-CohFT associated to a flat F-manifold}, a family of F-CohFTs associated to our flat F-manifold is given by
\begin{gather}\label{eq:oF-G_0 F-CohFT}
c^{\oF,G_0}:=\Psi^{-1}(0)R^{-1}(-z,0).c^{\triv,G_0}.
\end{gather}
Note that the degree zero part of $c^{\oF,G_0}_{1,1}(e^\alpha)$ is equal to $\sum_j(\Psi^{-1}(0))^\alpha_j G_0^j$.\\

Suppose now that our F-manifold is homogeneous with an Euler vector field 
$$
E=E^\alpha\frac{\d}{\d t^\alpha}=((1-q_\alpha)t^\alpha+r^\alpha)\frac{\d}{\d t^\alpha}.
$$ 
Then, by Proposition~\ref{proposition:homogeneity of H and gamma}, we have $E^\alpha\frac{\d H_i}{\d t^\alpha}=\delta_i H_i$ for some $\delta_i\in\mbC$, $1\le i\le N$. Consider then a unique sequence of matrices~$R_i$, $i\ge 1$, given by Proposition~\ref{proposition:unique R-matrix in the homogeneous case}. For any $\delta\in\mbC$ define a subspace $V_\delta\subset\mbC^N$ by
$$
V_\delta:=\{w=(w^1,\ldots,w^N)\in\mbC^N|w^i=0\text{ if } \delta_i\ne\delta\}.
$$
Let $\mathcal{D}:=\{\delta_i\}_{1\le i\le N}$. We get the decomposition $\mbC^N=\oplus_{\delta\in\mathcal{D}}V_{\delta}$.

\smallskip

\begin{theorem}\label{theorem:homogeneous F-CohFT}
For any $1\le l\le N$ and a vector $G_0\in V_{\delta_l}$ the F-CohFT 
$$
c^{\oF,G_0}=\Psi^{-1}(0)R^{-1}(-z,0).c^{\triv,G_0}
$$
is homogeneous of conformal dimension $-2\delta_l$.
\end{theorem}

\smallskip

The proof of the theorem is based on the following crucial result, which is true without the homogeneity assumption. For arbitrary vectors $w=(w^1,\ldots,w^N)\in(\mbC^*)^N$ and $G_0=(G_0^1,\ldots,G_0^N)\in\mbC^N$ define an F-TFT $c^{w,G_0}$ with the phase space $V=\mbC^N$ by
\begin{gather}\label{eq:definition of cwG0}
c^{w,G_0}_{g,n+1}(e^{i_0}\otimes\otimes_{j=1}^n e_{i_j}):=
\begin{cases}
\frac{(G_0^{i_0})^g}{(w^{i_0})^{g+n-1}},&\text{if $i_0=i_1=\ldots=i_n$},\\
0,&\text{otherwise}.
\end{cases}
\end{gather}
This F-TFT corresponds to the constant flat F-manifold with the vector potential $\left(\frac{(t^1)^2}{2w^1},\ldots,\frac{(t^N)^2}{2w^N}\right)$ and the unit $\sum_{i=1}^N w^i\frac{\d}{\d t^i}$.\\

As at the beginning of this section, consider a flat F-manifold given by a vector potential~$\oF$, $F^\alpha\in \mbC[[t^1,\ldots,t^N]]$, that is semisimple at the origin $t^*=0$, and the associated matrices~$H(t^*)$,~$\Psi(t^*)$ and~$R_i(t^*)$ such that $H_i(0)=1$. 

\smallskip

\begin{proposition}\label{proposition:crucial proposition}
For an arbitrary $G_0\in\mbC^N$, we have
\begin{gather}\label{eq:crucial proposition}
c^{\oF,G_0,\ot}=\Psi^{-1}(t^*)R^{-1}(-z,t^*).c^{\oH(t^*),H^{-1}(t^*)G_0},
\end{gather}
where $c^{\oF,G_0,\ot}$ is the formal shift of the F-CohFT $c^{\oF,G_0}$, $\ot=(t^1,\ldots,t^N)$.
\end{proposition}
\begin{proof}
Obviously, both sides of equation~\eqref{eq:crucial proposition} are equal if we set $t^\alpha=0$. From equation~\eqref{eq:formal shift of an F-CohFT} it is clear that the left-hand side of~\eqref{eq:crucial proposition} satisfies the differential equation
$$
\frac{\d}{\d t^\beta}\left(c^{\oF,G_0,\ot}_{g,n+1}(e^{\alpha_0}\otimes\otimes_{i=1}^ne_{\alpha_i})\right)=\pi_{1*}\left(c^{\oF,G_0,\ot}_{g,n+2}(e^{\alpha_0}\otimes\otimes_{i=1}^ne_{\alpha_i}\otimes e_\beta)\right).
$$
Therefore, it is sufficient to check that the right-hand side of~\eqref{eq:crucial proposition} satisfies the same differential equation, 
\begin{multline*}
\frac{\d}{\d t^\beta}\left(\left(\Psi^{-1}R^{-1}(-z).c^{\oH,H^{-1}G_0}\right)_{g,n+1}(e^{\alpha_0}\otimes\otimes_{i=1}^ne_{\alpha_i})\right)=\\
=\pi_{1*}\left(\left(\Psi^{-1}R^{-1}(-z).c^{\oH,H^{-1}G_0}\right)_{g,n+2}(e^{\alpha_0}\otimes\otimes_{i=1}^ne_{\alpha_i}\otimes e_\beta)\right),
\end{multline*}
or, equivalently,
\begin{multline}\label{eq:differential equation for RHS}
d\left(\left(\Psi^{-1}R^{-1}(-z).c^{\oH,H^{-1}G_0}\right)_{g,n+1}(e^{\alpha_0}\otimes\otimes_{i=1}^ne_{\alpha_i})\right)=\\
=\pi_{1*}\left(\left(\Psi^{-1}R^{-1}(-z).c^{\oH,H^{-1}G_0}\right)_{g,n+2}(e^{\alpha_0}\otimes\otimes_{i=1}^ne_{\alpha_i}\otimes e_\beta dt^\beta)\right).
\end{multline}

\smallskip

Recall that the unit of the F-CohFT $c^{\oH,H^{-1}G_0}$ is $\oH=\sum_{i=1}^n H_i e_i$. We have
$$
\Psi^{-1}R^{-1}(-z).c^{\oH,H^{-1}G_0}=\Psi^{-1}R^{-1}(-z)T''_{R^{-1}(-z)}c^{\oH,H^{-1}G_0},
$$
where $T''_{R^{-1}(-z)}=z\left(\oH-R(-z)\oH\right)$. Let us introduce the notation
$$
R_{\ge k}(z):=\sum_{i\ge k}R_i z^i,\quad k\ge 1.
$$
We see that
$$
\left(T''_{R^{-1}(-z)}c^{\oH,H^{-1}G_0}\right)_{g,n+1}(e^{i_0}\otimes\otimes_{j=1}^n e_{i_j})=
\begin{cases}
\Omega^k_{g,n+1},&\text{if $i_0=i_1=\ldots=i_n=k$},\\
0,&\text{otherwise},
\end{cases}
$$
where
$$
\Omega^k_{g,n}=\sum_{m\ge 0}\frac{1}{m!}\frac{(G_0^k)^g}{H_k^{2g+n+m-2}}\pi_{m*}\left(\prod_{i=n+1}^{n+m}(-\psi_i)\left(R_{\ge 1}(-\psi_i)\oH\right)^k\right)\in H^*(\oM_{g,n})\otimes\mbC[[t^1,\ldots,t^N]],
$$
and $\left(R_{\ge 1}(-\psi_i)\oH\right)^k$ denotes the $k$-th component of the vector $R_{\ge 1}(-\psi_i)\oH$.\\

Let
$$
T^N_{g,n+1}:=\left\{(\Gamma,f)\left|\begin{smallmatrix}\Gamma\in T_{g,n+1}\\f\colon V(\Gamma)\to\{1,\ldots,N\}\end{smallmatrix}\right.\right\}.
$$
We denote by $H(\Gamma)$ the set of half-edges of $\Gamma\in T_{g,n+1}$. A function $f\colon V(\Gamma)\to\{1,\ldots,N\}$ induces a function $H(\Gamma)\to\{1,\ldots,N\}$, denoted by the same letter~$f$, by $f(h):=f(v(h))$, where $h\in H(\Gamma)$ and $v(h)$ is the vertex of $\Gamma$ incident to~$h$. We denote by $l_i(\Gamma)$ the leg of $\Gamma$ marked by $i$, $1\le i\le n+1$. Let us also introduce the notations
$$
\tR(z):=R(z)\Psi,\qquad \ET(x,y):=\frac{\Id-R(-x)R^{-1}(y)}{x+y}.
$$
Then the F-CohFT $\Psi^{-1}R^{-1}(-z).c^{\oH,H^{-1}G_0}$ can be described in the following way:
\begin{multline*}
\left(\Psi^{-1}R^{-1}(-z).c^{\oH,H^{-1}G_0}\right)_{g,n+1}(e^{\alpha_0}\otimes\otimes_{i=1}^ne_{\alpha_i})=\\
=\sum_{(\Gamma,f)\in T^N_{g,n+1}}\xi_{\Gamma*}\left[\prod_{v\in V(\Gamma)}\Omega_{g(v),n(v)}^{f(v)}\tR^{-1}(\psi_1)^{\alpha_0}_{f(l_1)}\prod_{k=2}^{n+1}\tR(-\psi_k)^{f(l_k)}_{\alpha_{k-1}}\prod_{e\in E(\Gamma)}\ET(\psi_{e'},\psi_{e''})^{f(e')}_{f(e'')}\right].
\end{multline*}

\smallskip

For a pair $(\Gamma,f)\in T^N_{g,n+1}$, $v\in V(\Gamma)$, $e\in E(\Gamma)$, and $2\le k\le n+1$, let us introduce the following classes in $H^*\left(\prod_{v\in V(\Gamma)}\oM_{g(v),n(v)}\right)\otimes\mbC[[t^1,\ldots,t^N]]$:
\begin{align*}
\Cont_{\Gamma,f}^v:=&\prod_{\substack{\tv\in V(\Gamma)\\\tv\ne v}}\Omega_{g(\tv),n(\tv)}^{f(\tv)}\tR^{-1}(\psi_1)^{\alpha_0}_{f(l_1)}\prod_{k=2}^{n+1}\tR(-\psi_k)^{f(l_k)}_{\alpha_{k-1}}\prod_{e\in E(\Gamma)}\ET(\psi_{e'},\psi_{e''})^{f(e')}_{f(e'')},\\
\Cont_{\Gamma,f}^{l_1}:=&\prod_{v\in V(\Gamma)}\Omega_{g(v),n(v)}^{f(v)}\prod_{k=2}^{n+1}\tR(-\psi_k)^{f(l_k)}_{\alpha_{k-1}}\prod_{e\in E(\Gamma)}\ET(\psi_{e'},\psi_{e''})^{f(e')}_{f(e'')},\\
\Cont_{\Gamma,f}^{l_k}:=&\prod_{v\in V(\Gamma)}\Omega_{g(v),n(v)}^{f(v)}\tR^{-1}(\psi_1)^{\alpha_0}_{f(l_1)}\prod_{\substack{2\le\tk\le n+1\\\tk\ne k}}\tR(-\psi_{\tk})^{f(l_{\tk})}_{\alpha_{\tk-1}}\prod_{e\in E(\Gamma)}\ET(\psi_{e'},\psi_{e''})^{f(e')}_{f(e'')},\\
\Cont_{\Gamma,f}^e:=&\prod_{v\in V(\Gamma)}\Omega_{g(v),n(v)}^{f(v)}\tR^{-1}(\psi_1)^{\alpha_0}_{f(l_1)}\prod_{k=2}^{n+1}\tR(-\psi_k)^{f(l_k)}_{\alpha_{k-1}}\prod_{\substack{\te\in E(\Gamma)\\\te\ne e}}\ET(\psi_{e'},\psi_{e''})^{f(e')}_{f(e'')}.
\end{align*}
Then we can write
\begin{align*}
&d\left(\left(\Psi^{-1}R^{-1}(-z).c^{\oH,H^{-1}G_0}\right)_{g,n+1}(e^{\alpha_0}\otimes\otimes_{i=1}^ne_{\alpha_i})\right)=\\
=&\sum_{(\Gamma,f)\in T^N_{g,n+1}}\xi_{\Gamma*}\left[\sum_{v\in V(\Gamma)}d\Omega_{g(v),n(v)}^{f(v)}\Cont_{\Gamma,f}^v+d\tR^{-1}(\psi_1)^{\alpha_0}_{f(l_1)}\Cont_{\Gamma,f}^{l_1}+\sum_{k=2}^{n+1}d\tR(-\psi_k)^{f(l_k)}_{\alpha_{k-1}}\Cont_{\Gamma,f}^{l_k}\right.+\\
&\left.\hspace{2.7cm}+\sum_{e\in E(\Gamma)}d\ET(\psi_{e'},\psi_{e''})^{f(e')}_{f(e'')}\Cont_{\Gamma,f}^e\right].
\end{align*}

\smallskip

Equations \eqref{eq:formulas for dPsi and dGammadU},~\eqref{eq:equation for oH}, and~\eqref{eq:equation for R(z)} imply that
\begin{align*}
&d\left(zR_{\ge 1}(z)\oH\right)=[R_{\ge 2}(z),dU]\oH, && z\left(dR^{-1}(z)-[\Gamma,dU]R^{-1}(z)\right)=[R^{-1}(z),dU],\\
&d\tR(z)=z^{-1}[R(z),dU]\Psi, && d\tR^{-1}(z)=z^{-1}\Psi^{-1}[R^{-1}(z),dU],
\end{align*}
which gives the following equations:
\begin{align}
d\Omega^k_{g,n}=&\sum_{m\ge 0}\frac{2g+n+m-2}{m!}\frac{(G_0^k)^g}{H_k^{2g+n+m-1}}\left([dU,\Gamma]\oH\right)^k\pi_{m*}\left(\prod_{i=n+1}^{n+m}(-\psi_i)\left(R_{\ge 1}(-\psi_i)\oH\right)^k\right)+\label{eq:formula for dOmega}\\
&+\sum_{m\ge 0}\frac{1}{m!}\frac{(G_0^k)^g}{H_k^{2g+n+m-1}}\pi_{(m+1)*}\left(\prod_{i=n+1}^{n+m}(-\psi_i)\left(R_{\ge 1}(-\psi_i)\oH\right)^k\cdot\left([R_{\ge 2}(-\psi_{n+m+1}),dU]\oH\right)^k\right),\notag\\
d\tR^{-1}(z)=&z^{-1}\Psi^{-1}[R^{-1}(z),dU],\label{eq:formula for dtRinv}\\
d\tR(-z)=&z^{-1}[dU,R(-z)]\Psi,\label{eq:formula for dtR}\\
d\ET(x,y)=&\frac{y[R(-x),dU]R^{-1}(y)+xR(-x)[dU,R^{-1}(y)]}{x y(x+y)}.\label{eq:formula for dET}
\end{align}

\smallskip

For the right-hand side of~\eqref{eq:differential equation for RHS}, we compute
\begin{multline}
\pi_{1*}\left(\left(\Psi^{-1}R^{-1}(-z).c^{\oH,H^{-1}G_0}\right)_{g,n+2}(e^{\alpha_0}\otimes\otimes_{i=1}^ne_{\alpha_i}\otimes e_\beta dt^\beta)\right)=\label{eq:sum on RHS of differential equation}\\
=\pi_{1*}\left(\sum_{(\Gamma,f)\in T^N_{g,n+2}}\xi_{\Gamma*}\left[\tR(-\psi_{n+2})^{f(l_{n+2})}_\beta dt^\beta\Cont_{\Gamma,f}^{l_{n+2}}\right]\right).
\end{multline}
Define a subset $\tT^N_{g,n+2}\subset T^N_{g,n+2}$ by
$$
\tT^N_{g,n+2}:=\left\{(\Gamma,f)\in T^N_{g,n+2}\left|\begin{smallmatrix}g(v(l_{n+2}(\Gamma)))=0\\n(v(l_{n+2}(\Gamma)))=3\end{smallmatrix}\right.\right\}.
$$
Let us first compute the part of the sum on the right-hand side of~\eqref{eq:sum on RHS of differential equation} where $(\Gamma,f)\in \tT^N_{g,n+2}$:
\begin{align}
&\pi_{1*}\left(\sum_{(\Gamma,f)\in\tT^N_{g,n+2}}\xi_{\Gamma*}\left[\tR(-\psi_{n+2})^{f(l_{n+2})}_\beta dt^\beta\Cont_{\Gamma,f}^{l_{n+2}}\right]\right)=\label{eq:sum for contracted components,1}\\
&\hspace{2cm}=\sum_{(\Gamma,f)\in T^N_{g,n+1}}\xi_{\Gamma*}\left[\left(\Psi^{-1}dU\frac{\Id-R^{-1}(\psi_1)}{\psi_1}\right)^{\alpha_0}_{f(l_1)}\Cont_{\Gamma,f}^{l_1}+\right.\label{eq:sum for contracted components,2}\\
&\left.\hspace{5.1cm}+\sum_{k=2}^{n+1}\left(\frac{\Id-R(-\psi_k)}{\psi_k}dU\Psi\right)_{\alpha_{k-1}}^{f(l_k)}\Cont_{\Gamma,f}^{l_k}+\right.\label{eq:sum for contracted components,3}\\
&\left.\hspace{5.1cm}+\sum_{e\in E(\Gamma)}\left(\frac{\Id-R(-\psi_{e'})}{\psi_{e'}}dU\frac{\Id-R^{-1}(\psi_{e''})}{\psi_{e''}}\right)_{f(e'')}^{f(e')}\Cont_{\Gamma,f}^e\right].\label{eq:sum for contracted components,4}
\end{align}
Here, the sum in line~\eqref{eq:sum for contracted components,2} corresponds to the part of the sum in line~\eqref{eq:sum for contracted components,1} where $v(l_1(\Gamma))=v(l_{n+2}(\Gamma))$, the sum in line~\eqref{eq:sum for contracted components,3} corresponds to the part of the sum in line~\eqref{eq:sum for contracted components,1} where $v(l_k(\Gamma))=v(l_{n+2}(\Gamma))$ for $2\le k\le n+1$, and the sum in line~\eqref{eq:sum for contracted components,4} corresponds to the part of the sum in line~\eqref{eq:sum for contracted components,1} where the leg $l_{n+2}(\Gamma)$ is a unique leg incident to the vertex $v(l_{n+2}(\Gamma))$.\\

Let us now compute the part of the sum on the right-hand side of~\eqref{eq:sum on RHS of differential equation} where $(\Gamma,f)\in T^N_{g,n+2}\backslash\tT^N_{g,n+2}$. Consider the following diagram of forgetful maps:
\begin{gather*}
\xymatrix{
\oM_{g,n+m+1}\ar[d]_{\tpi_1}\ar[r]^{\tpi_m}\ar[dr]^{\pi_{m+1}} & \oM_{g,n+1}\ar[d]^{\pi_1}\\
\oM_{g,n+m}\ar[r]^{\pi_m} & \oM_{g,n}
}
\end{gather*}

\smallskip

\begin{lemma}\label{lemma:forgetful map}
Consider integers $a_1,\ldots,a_{n+1}\ge 0$ and $b_1,\ldots,b_m\ge 2$. Then we have
\begin{multline*}
\pi_{1*}\left[\prod_{j=1}^{n+1}\psi_j^{a_j}\tpi_{m*}\left(\prod_{j=1}^m\psi_{n+1+j}^{b_j}\right)\right]=\\
=\begin{cases}
\sum\limits_{i=1}^n\psi_i^{a_i-1}\prod\limits_{\substack{1\le j\le n\\j\ne i}}\psi_j^{a_j}\pi_{m*}\Big(\prod\limits_{j=1}^m\psi_{n+j}^{b_j}\Big)+\prod\limits_{j=1}^n\psi_j^{a_j}\pi_{m*}\Big(\sum\limits_{i=1}^m\psi_{n+i}^{b_i-1}\prod\limits_{j\ne i}\psi_{n+j}^{b_j}\Big),&\text{if $a_{n+1}=0$},\\
(2g+n+m-2)\prod\limits_{j=1}^n\psi_j^{a_j}\pi_{m*}\Big(\prod\limits_{j=1}^m\psi_{n+j}^{b_j}\Big),&\text{if $a_{n+1}=1$},\\
\prod\limits_{j=1}^n\psi_j^{a_j}\pi_{(m+1)*}\Big(\psi_{n+1}^{a_{n+1}}\prod\limits_{j=1}^m\psi_{n+1+j}^{b_j}\Big),&\text{if $a_{n+1}\ge 2$}.
\end{cases}
\end{multline*}
\end{lemma}
\begin{proof}
Suppose $a_{n+1}=0$, then, using~\eqref{eq:monomial in psi's and forgetful map}, we compute
\begin{align*}
&\pi_{1*}\left[\prod_{j=1}^n\psi_j^{a_j}\tpi_{m*}\left(\prod_{j=1}^m\psi_{n+1+j}^{b_j}\right)\right]=\\
=&\pi_{1*}\left[\pi_1^*\left(\prod_{j=1}^n\psi_j^{a_j}\right)\tpi_{m*}\left(\prod_{j=1}^m\psi_{n+1+j}^{b_j}\right)+\sum_{i=1}^n\pi_1^*\left(\psi_i^{a_i-1}\prod_{j\ne i}\psi_j^{a_j}\right)\delta_0^{\{i,n+1\}}\tpi_{m*}\left(\prod_{j=1}^m\psi_{n+1+j}^{b_j}\right)\right]=\\
=&\pi_{(m+1)*}\left[\pi_{m+1}^*\left(\prod_{j=1}^n\psi_j^{a_j}\right)\prod_{j=1}^m\psi_{n+1+j}^{b_j}+\sum_{i=1}^n\pi_{m+1}^*\left(\psi_i^{a_i-1}\prod_{j\ne i}\psi_j^{a_j}\right)\tpi_m^*(\delta_0^{\{i,n+1\}})\prod_{j=1}^m\psi_{n+1+j}^{b_j}\right].
\end{align*}
Since $b_1,\ldots,b_m\ge 2$, we have $\tpi_m^*\left(\delta_0^{\{i,n+1\}}\right)\prod_{j=1}^m\psi_{n+1+j}^{b_j}=\delta_0^{\{i,n+1\}}\prod_{j=1}^m\psi_{n+1+j}^{b_j}$. Therefore, we can continue the last chain of equations as follows:
\begin{align*}
&\pi_{(m+1)*}\left[\pi_{m+1}^*\left(\prod_{j=1}^n\psi_j^{a_j}\right)\prod_{j=1}^m\psi_{n+1+j}^{b_j}+\sum_{i=1}^n\pi_{m+1}^*\left(\psi_i^{a_i-1}\prod_{j\ne i}\psi_j^{a_j}\right)\delta_0^{\{i,n+1\}}\prod_{j=1}^m\psi_{n+1+j}^{b_j}\right]=\\
=&\pi_{m*}\left[\pi_m^*\left(\prod_{j=1}^n\psi_j^{a_j}\right)\sum_{i=1}^m\psi_{n+i}^{b_i-1}\prod_{j\ne i}\psi_{n+j}^{b_j}+\sum_{i=1}^n\pi_m^*\left(\psi_i^{a_i-1}\prod_{j\ne i}\psi_j^{a_j}\right)\prod_{j=1}^m\psi_{n+j}^{b_j}\right]=\\
=&\prod_{j=1}^n\psi_j^{a_j}\pi_{m*}\left(\sum_{i=1}^m\psi_{n+i}^{b_i-1}\prod_{j\ne i}\psi_{n+j}^{b_j}\right)+\sum_{i=1}^n\psi_i^{a_i-1}\prod_{j\ne i}\psi_j^{a_j}\pi_{m*}\left(\prod_{j=1}^m\psi_{n+j}^{b_j}\right).
\end{align*}

\smallskip

If $a_{n+1}\ge 1$, then, using again~\eqref{eq:monomial in psi's and forgetful map}, we obtain
\begin{align*}
\pi_{1*}\left[\prod_{j=1}^{n+1}\psi_j^{a_j}\tpi_{m*}\left(\prod_{j=1}^m\psi_{n+1+j}^{b_j}\right)\right]=&\pi_{1*}\left[\pi_1^*\left(\prod_{j=1}^n\psi_j^{a_j}\right)\psi_{n+1}^{a_{n+1}}\tpi_{m*}\left(\prod_{j=1}^m\psi_{n+1+j}^{b_j}\right)\right]=\\
=&\pi_{(m+1)*}\left[\pi_{m+1}^*\left(\prod_{j=1}^n\psi_j^{a_j}\right)\tpi_m^*(\psi_{n+1}^{a_{n+1}})\prod_{j=1}^m\psi_{n+1+j}^{b_j}\right].
\end{align*}
Noticing that $\tpi_m^*(\psi_{n+1}^{a_{n+1}})\prod_{j=1}^m\psi_{n+1+j}^{b_j}=\psi_{n+1}^{a_{n+1}}\prod_{j=1}^m\psi_{n+1+j}^{b_j}$ we get
$$
\pi_{(m+1)*}\left[\pi_{m+1}^*\left(\prod_{j=1}^n\psi_j^{a_j}\right)\psi_{n+1}^{a_{n+1}}\prod_{j=1}^m\psi_{n+1+j}^{b_j}\right]=\prod_{j=1}^n\psi_j^{a_j}\pi_{(m+1)*}\left(\psi_{n+1}^{a_{n+1}}\prod_{j=1}^m\psi_{n+1+j}^{b_j}\right).
$$
If $a_{n+1}\ge 2$, then this proves the lemma. If $a_{n+1}=1$, then we just note that 
$$
\pi_{(m+1)*}\left(\psi_{n+1}\prod_{j=1}^m\psi_{n+1+j}^{b_j}\right)=(2g+n+m-2)\pi_{m*}\left(\prod_{j=1}^m\psi_{n+j}^{b_j}\right).
$$
\end{proof}

\smallskip

We write
\begin{align}
&\pi_{1*}\left(\sum_{(\Gamma,f)\in T^N_{g,n+2}\backslash\tT^N_{g,n+2}}\xi_{\Gamma*}\left[\tR(-\psi_{n+2})^{f(l_{n+2})}_\beta dt^\beta\Cont_{\Gamma,f}^{l_{n+2}}\right]\right)=\notag\\
&\hspace{3cm}=\pi_{1*}\left(\sum_{(\Gamma,f)\in T^N_{g,n+2}\backslash\tT^N_{g,n+2}}\xi_{\Gamma*}\left[\Psi^{f(l_{n+2})}_\beta dt^\beta\Cont_{\Gamma,f}^{l_{n+2}}\right]\right)+\label{eq:crucial proposition,sum1}\\
&\hspace{3.5cm}+\pi_{1*}\left(\sum_{(\Gamma,f)\in T^N_{g,n+2}\backslash\tT^N_{g,n+2}}\xi_{\Gamma*}\left[(-\psi_{n+2})(R_1\Psi)^{f(l_{n+2})}_\beta dt^\beta\Cont_{\Gamma,f}^{l_{n+2}}\right]\right)+\label{eq:crucial proposition,sum2}\\
&\hspace{3.5cm}+\pi_{1*}\left(\sum_{(\Gamma,f)\in T^N_{g,n+2}\backslash\tT^N_{g,n+2}}\xi_{\Gamma*}\left[\tR_{\ge 2}(-\psi_{n+2})^{f(l_{n+2})}_\beta dt^\beta\Cont_{\Gamma,f}^{l_{n+2}}\right]\right).\label{eq:crucial proposition,sum3}
\end{align}
By Lemma~\ref{lemma:forgetful map}, the expression in line~\eqref{eq:crucial proposition,sum1} is equal to
\begin{align*}
&\sum_{(\Gamma,f)\in T^N_{g,n+1}}\xi_{\Gamma*}\left[\sum_{v\in V(\Gamma)}A^{f(v)}_{g(v),n(v)}\Cont_{\Gamma,f}^v+\left(\Psi^{-1}\frac{R^{-1}(\psi_1)-\Id}{\psi_1}dU\right)^{\alpha_0}_{f(l_1)}\Cont_{\Gamma,f}^{l_1}+\right.\\
&\hspace{1cm}+\sum_{k=2}^{n+1}\left(dU\frac{R(-\psi_k)-\Id}{\psi_k}\Psi\right)^{f(l_k)}_{\alpha_{k-1}}\Cont_{\Gamma,f}^{l_k}+\\
&\hspace{1cm}\left.+\sum_{e\in E(\Gamma)}\left(\frac{\ET(\psi_{e'},\psi_{e''})-\ET(\psi_{e'},0)}{\psi_{e''}}dU+dU\frac{\ET(\psi_{e'},\psi_{e''})-\ET(0,\psi_{e''})}{\psi_{e'}}\right)^{f(e')}_{f(e'')}\Cont_{\Gamma,f}^e\right],
\end{align*}
where
\begin{align*}
A^k_{g,n}=&\sum_{m\ge 0}\frac{1}{m!}\frac{(G_0^k)^g}{H_k^{2g+n+m-1}}\pi_{(m+1)*}\left(\prod_{i=n+1}^{n+m}(-\psi_i)\left(R_{\ge 1}(-\psi_i)\oH\right)^k\cdot(-1)\left(dU R_{\ge 1}(-\psi_{n+m+1})\oH\right)^k\right)=\\
=&\sum_{m\ge 0}\frac{1}{m!}\frac{(G_0^k)^g}{H_k^{2g+n+m-1}}\pi_{(m+1)*}\left(\prod_{i=n+1}^{n+m}(-\psi_i)\left(R_{\ge 1}(-\psi_i)\oH\right)^k\cdot(-1)\left(dU R_{\ge 2}(-\psi_{n+m+1})\oH\right)^k\right)+\\
&+\sum_{m\ge 0}\frac{2g+n+m-2}{m!}\frac{(G_0^k)^g}{H_k^{2g+n+m-1}}\left(dU R_1\oH\right)^k\pi_{m*}\left(\prod_{i=n+1}^{n+m}(-\psi_i)\left(R_{\ge 1}(-\psi_i)\oH\right)^k\right).
\end{align*}
Lemma~\ref{lemma:forgetful map} implies that the expressions in lines~\eqref{eq:crucial proposition,sum2} and~\eqref{eq:crucial proposition,sum3} are equal to 
$$
\sum_{(\Gamma,f)\in T^N_{g,n+1}}\xi_{\Gamma*}\left[B^{f(v)}_{g(v),n(v)}\Cont^v_{\Gamma,f}\right]\quad\text{and}\quad\sum_{(\Gamma,f)\in T^N_{g,n+1}}\xi_{\Gamma*}\left[C^{f(v)}_{g(v),n(v)}\Cont^v_{\Gamma,f}\right],
$$
respectively, where
\begin{align*}
B^k_{g,n}=&-\sum_{m\ge 0}\frac{2g+n+m-2}{m!}\frac{(G_0^k)^g}{H_k^{2g+n+m-1}}\left(R_1dU\oH\right)^k\pi_{m*}\left(\prod_{i=n+1}^{n+m}(-\psi_i)\left(R_{\ge 1}(-\psi_i)\oH\right)^k\right),\\
C^k_{g,n}=&\sum_{m\ge 0}\frac{1}{m!}\frac{(G_0^k)^g}{H_k^{2g+n+m-1}}\pi_{(m+1)*}\left(\prod_{i=n+1}^{n+m}(-\psi_i)\left(R_{\ge 1}(-\psi_i)\oH\right)^k\cdot\left(R_{\ge 2}(-\psi_{n+m+1})dU\oH\right)^k\right).
\end{align*}

\smallskip

Summarizing the above computations, we get
\begin{align*}
&\pi_{1*}\left(\left(\Psi^{-1}R^{-1}(-z).c^{\oH,H^{-1}G_0}\right)_{g,n+2}(e^{\alpha_0}\otimes\otimes_{i=1}^ne_{\alpha_i}\otimes e_\beta dt^\beta)\right)=\\
=&\sum_{(\Gamma,f)\in T^N_{g,n+1}}\xi_{\Gamma*}\left[\sum_{v\in V(\Gamma)}\mcV_{g(v),n(v)}^{f(v)}\Cont_{\Gamma,f}^v+\tcL(\psi_1)^{\alpha_0}_{f(l_1)}\Cont_{\Gamma,f}^{l_1}+\sum_{k=2}^{n+1}\mcL(\psi_k)^{f(l_k)}_{\alpha_{k-1}}\Cont_{\Gamma,f}^{l_k}\right.+\\
&\left.\hspace{2.7cm}+\sum_{e\in E(\Gamma)}\mcE(\psi_{e'},\psi_{e''})^{f(e')}_{f(e'')}\Cont_{\Gamma,f}^e\right],
\end{align*}
where
\begin{align*}
\mcV^k_{g,n}=&A^k_{g,n}+B^k_{g,n}+C^k_{g,n}\stackrel{\text{eq.~\eqref{eq:formula for dOmega}}}{=}d\Omega^k_{g,n},\\
\tcL(z)=&\Psi^{-1}dU\frac{\Id-R^{-1}(z)}{z}+\Psi^{-1}\frac{R^{-1}(z)-\Id}{z}dU=\Psi^{-1}z^{-1}[R^{-1}(z),dU]\stackrel{\text{eq.~\eqref{eq:formula for dtRinv}}}{=}d\tR^{-1}(z),\\
\mcL(z)=&\frac{\Id-R(-z)}{z}dU\Psi+dU\frac{R(-z)-\Id}{z}\Psi=z^{-1}[dU,R(-z)]\Psi\stackrel{\text{eq.~\eqref{eq:formula for dtR}}}{=}d\tR(-z),\\
\mcE(x,y)=&\frac{\Id-R(-x)}{x}dU\frac{\Id-R^{-1}(y)}{y}+\frac{\ET(x,y)-\ET(x,0)}{y}dU+dU\frac{\ET(x,y)-\ET(0,y)}{x}=\\
=&\frac{y[R(-x),dU]R^{-1}(y)+xR(-x)[dU,R^{-1}(y)]}{x y(x+y)}\stackrel{\text{eq.~\eqref{eq:formula for dET}}}{=}d\ET(x,y).
\end{align*}
We conclude that equation~\eqref{eq:differential equation for RHS} is true.
\end{proof}

\smallskip

\begin{proof}[Proof of Theorem~\ref{theorem:homogeneous F-CohFT}]
We have to check that
\begin{multline}\label{eq:homogeneity property of F-CohFT}
\Deg c^{\oF,G_0}_{g,n+1}(e^{\alpha_0}\otimes\otimes_{i=1}^ne_{\alpha_i})+\pi_{1*}c^{\oF,G_0}_{g,n+2}(e^{\alpha_0}\otimes\otimes_{i=1}^ne_{\alpha_i}\otimes r^\gamma e_\gamma)=\\
=\left(\sum_{i=1}^n q_{\alpha_i}-q_{\alpha_0}-2\delta_l g\right)c^{\oF,G_0}_{g,n+1}(e^{\alpha_0}\otimes\otimes_{i=1}^ne_{\alpha_i}).
\end{multline}
Since 
$$
\pi_{1*}c^{\oF,G_0}_{g,n+2}(e^{\alpha_0}\otimes\otimes_{i=1}^ne_{\alpha_i}\otimes r^\gamma e_\gamma)=\left.E^\alpha\frac{\d}{\d t^\alpha}c^{\oF,G_0,\ot}_{g,n+1}(e^{\alpha_0}\otimes\otimes_{i=1}^ne_{\alpha_i})\right|_{t^*=0},
$$
equation~\eqref{eq:homogeneity property of F-CohFT} follows from the equation
\begin{gather*}
\left(\Deg+E^\alpha\frac{\d}{\d t^\alpha}\right)c^{\oF,G_0,\ot}_{g,n+1}(e^{\alpha_0}\otimes\otimes_{i=1}^ne_{\alpha_i})=\left(\sum_{i=1}^n q_{\alpha_i}-q_{\alpha_0}-2\delta_l g\right)c^{\oF,G_0,\ot}_{g,n+1}(e^{\alpha_0}\otimes\otimes_{i=1}^ne_{\alpha_i}),
\end{gather*}
which, by Proposition~\ref{proposition:crucial proposition}, is equivalent to the equation
\begin{multline}\label{eq:main homogeneity}
\left(\Deg+E^\alpha\frac{\d}{\d t^\alpha}\right)\left(\Psi^{-1}R^{-1}(-z).c^{\oH,H^{-1}G_0}\right)_{g,n+1}(e^{\alpha_0}\otimes\otimes_{i=1}^ne_{\alpha_i})=\\
=\left(\sum_{i=1}^n q_{\alpha_i}-q_{\alpha_0}-2\delta_l g\right)\left(\Psi^{-1}R^{-1}(-z).c^{\oH,H^{-1}G_0}\right)_{g,n+1}(e^{\alpha_0}\otimes\otimes_{i=1}^ne_{\alpha_i}).
\end{multline}
Let us prove it.\\

From definition~\eqref{eq:definition of cwG0} and the assumption $G_0\in V_{\delta_l}$ it follows that
\begin{gather*}
\left(\Deg+E^\alpha\frac{\d}{\d t^\alpha}\right)c^{\oH,H^{-1}G_0}_{g,n+1}(e^{i_0}\otimes\otimes_{j=1}^n e_{i_j})=\left(\delta_{i_0}-\sum_{j=1}^n \delta_{i_j}-2\delta_l g\right)c^{\oH,H^{-1}G_0}_{g,n+1}(e^{i_0}\otimes\otimes_{j=1}^n e_{i_j}).
\end{gather*}
Recall that the unit of the F-CohFT $c^{\oH,H^{-1}G_0}$ is $\oH=\sum_{i=1}^n H_i e_i$. We express
$$
R^{-1}(-z).c^{\oH,H^{-1}G_0}=R^{-1}(-z)T''_{R^{-1}(-z)}c^{\oH,H^{-1}G_0},
$$
where $T''_{R^{-1}(-z)}=z(\oH-R(-z)\oH)$. Consider the definition~\eqref{eq:T-action} for the action of $T''_{R^{-1}(-z)}$ on the F-CohFT $c^{\oH,H^{-1}G_0}$. Let $\Delta:=\diag(\delta_1,\ldots,\delta_N)$. Since
$$
\left(z\frac{\d}{\d z}+E^\alpha\frac{\d}{\d t^\alpha}\right)R(z)=[\Delta,R(z)],\qquad E^\alpha\frac{\d}{\d t^\alpha}\oH=\Delta\oH,
$$ 
we have
$$
\left(z\frac{\d}{\d z}+E^\alpha\frac{\d}{\d t^\alpha}\right)\left(R(z)\oH\right)=\Delta\left(R(z)\oH\right),
$$
which implies that
$$
\left(\Deg+E^\alpha\frac{\d}{\d t^\alpha}\right)T''_{R^{-1}(-\psi_k)}=(\Id+\Delta)T''_{R^{-1}(-\psi_k)}\in V\otimes H^*(\oM_{g,n+m+1})\otimes\mbC[[t^1,\ldots,t^N]].
$$
Since the map $\pi_{m*}\colon H^*(\oM_{g,n+m+1})\to H^*(\oM_{g,n+1})$ decreases the cohomological degree by~$2m$, we obtain
\begin{multline*}
\left(\Deg+E^\alpha\frac{\d}{\d t^\alpha}\right)\left(T''_{R^{-1}(-z)}c^{\oH,H^{-1}G_0}\right)_{g,n+1}(e^{i_0}\otimes\otimes_{j=1}^n e_{i_j})=\\
=\left(\delta_{i_0}-\sum_{j=1}^n \delta_{i_j}-2\delta_l g\right)\left(T''_{R^{-1}(-z)}c^{\oH,H^{-1}G_0}\right)_{g,n+1}(e^{i_0}\otimes\otimes_{j=1}^n e_{i_j}).
\end{multline*}

\smallskip

Consider now definition~\eqref{eq:R-action} of the action of $R^{-1}(-z)$ on the F-CohFT without unit (and with coefficients in $\mbC[[t^1,\ldots,t^N]]$) $T''_{R^{-1}(-z)}c^{\oH,H^{-1}G_0}$. For the leg terms $R^{-1}(\psi_1)^t\in\End(V^*)\otimes H^*(\oM_{g,n+1})\otimes\mbC[[t^1,\ldots,t^N]]$ and $R(-\psi_k)\in \End(V)\otimes H^*(\oM_{g,n+1})\otimes\mbC[[t^1,\ldots,t^N]]$, we have
\begin{gather*}
\left(\Deg+E^\alpha\frac{\d}{\d t^\alpha}\right)R^{-1}(\psi_1)^t=[R^{-1}(\psi_1)^t,\Delta],\qquad 
\left(\Deg+E^\alpha\frac{\d}{\d t^\alpha}\right)R(-\psi_k)=[\Delta,R(-\psi_k)].
\end{gather*}
For the edge term $\frac{\Id-R(-\psi_{e'})R^{-1}(\psi_{e''})}{\psi_{e'}+\psi_{e''}}$, we compute
$$
\left(\Deg+E^\alpha\frac{\d}{\d t^\alpha}\right)\frac{\Id-R(-\psi_{e'})R^{-1}(\psi_{e''})}{\psi_{e'}+\psi_{e''}}=\left[\Delta,\frac{\Id-R(-\psi_{e'})R^{-1}(\psi_{e''})}{\psi_{e'}+\psi_{e''}}\right]-\frac{\Id-R(-\psi_{e'})R^{-1}(\psi_{e''})}{\psi_{e'}+\psi_{e''}}.
$$
Note that the map
$$
\xi_{\Gamma*}\colon H^*\left(\prod_{v\in V(\Gamma)}\oM_{g(v),n(v)}\right)\to H^*(\oM_{g,n+1})
$$
increases the cohomological degree by $2|E(\Gamma)|$. Summarizing the above computations for the action the operator $\Deg+E^\alpha\frac{\d}{\d t^\alpha}$ on the vertex, the leg, and the edge terms, we see that the contribution of each stable tree $\Gamma\in T_{g,n+1}$ in formula~\eqref{eq:R-action} to a class 
$$
\left(R^{-1}(-z)T''_{R^{-1}(-z)}c^{\oH,H^{-1}G_0}\right)_{g,n+1}(e^{i_0}\otimes\otimes_{j=1}^n e_{i_j})
$$
is an eigenvector of the operator $\Deg+E^\alpha\frac{\d}{\d t^\alpha}$ with the eigenvalue $\delta_{i_0}-\sum_{j=1}^n \delta_{i_j}-2\delta_l g$. Therefore,
\begin{multline}\label{eq:last step for homogeneity}
\left(\Deg+E^\alpha\frac{\d}{\d t^\alpha}\right)\left(R^{-1}(-z)T''_{R^{-1}(-z)}c^{\oH,H^{-1}G_0}\right)_{g,n+1}(e^{i_0}\otimes\otimes_{j=1}^n e_{i_j})=\\
=\left(\delta_{i_0}-\sum_{j=1}^n \delta_{i_j}-2\delta_l g\right)\left(R^{-1}(-z)T''_{R^{-1}(-z)}c^{\oH,H^{-1}G_0}\right)_{g,n+1}(e^{i_0}\otimes\otimes_{j=1}^n e_{i_j}).
\end{multline}

\smallskip

It remains to act by $\Psi^{-1}$ on the F-CohFT (with coefficients in $\mbC[[t^1,\ldots,t^N]]$) 
$$
R^{-1}(-z)T''_{R^{-1}(-z)}c^{\oH,H^{-1}G_0}=R^{-1}(-z).c^{\oH,H^{-1}G_0}.
$$
We have
$$
E^\alpha\frac{\d}{\d t^\alpha}\frac{\d u^i}{\d t^\beta}=\frac{\d}{\d t^\beta}\left(E^\alpha\frac{\d u^i}{\d t^\alpha}\right)-\frac{\d E^\alpha}{\d t^\beta}\frac{\d u^i}{\d t^\alpha}=q_\beta\frac{\d u^i}{\d t^\beta}\quad\Rightarrow\quad E^\alpha\frac{\d}{\d t^\alpha}\Psi^i_\beta=(\delta_i+q_\beta)\Psi^i_\beta.
$$
Together with equation~\eqref{eq:last step for homogeneity} this immediately implies equation~\eqref{eq:main homogeneity}.
\end{proof}

\smallskip

\subsubsection{Example: extended $2$-spin theory in all genera}

Let us apply the above construction to the flat F-CohFT of the extended $2$-spin theory (see Section~\ref{subsubsection:example1}):
$$
F^1(t^1,t^2)=\frac{(t^1)^2}{2},\qquad F^2(t^1,t^2)=t^1t^2-\frac{(t^2)^3}{12},\qquad E=t^1\frac{\d}{\d t^1}+\frac{1}{2}t^2\frac{\d}{\d t^2}.
$$
The unit is $\frac{\d}{\d t^1}$. The flat F-manifold is not semisimple at the origin, so we consider it around a semisimple point $(0,\tau)$, $\tau\in\mbC^*$. We have $\delta_1=0$ and $\delta_2=-\frac{1}{2}$. By Theorem~\ref{theorem:homogeneous F-CohFT}, there exist two families of homogeneous F-CohFTs with the associated flat F-manifolds given by the vector potential $\oF_{(0,\tau)}$.\\

First, for any $\lambda\in\mbC^*$ the F-CohFT
$$
c^{\oF_{(0,\tau)},(\lambda,0)}=\left.\left(\Psi^{-1}R^{-1}(-z)\right)\right|_{(t^1,t^2)=(0,\tau)}.c^{\triv,(\lambda,0)}
$$ 
is homogeneous of conformal dimension $0$. Here, the matrices $\Psi$ and $R(z)=\Id+\sum_{k\ge 1}R_k z^k$ were computed in Section~\ref{subsubsection:example1}. It is easy to see that
$$
c^{\oF_{(0,\tau)},(\lambda,0)}_{g,k+l+1}(e^1\otimes e_1^{\otimes k}\otimes e_2^{\otimes l})=\begin{cases}
0,&\text{if $l\ge 1$},\\
\lambda^g\in H^0(\oM_{g,k+1}),&\text{if $l=0$}.
\end{cases}
$$
An argument from~\cite[Section 6]{BR18} shows that the F-CohFT $c^{\oF_{(0,\tau)},(\lambda,0)}$ doesn't have a limit when $\tau\to 0$.\\

Second, for any $\lambda\in\mbC^*$ the F-CohFT
$$
c^{\oF_{(0,\tau)},(0,\lambda)}=\left.\left(\Psi^{-1}R^{-1}(-z)\right)\right|_{(t^1,t^2)=(0,\tau)}.c^{\triv,(0,\lambda)}
$$ 
is homogeneous of conformal dimension $1$. On the other hand, in~\cite[Theorem~3.9]{BR18} the authors constructed a homogeneous F-CohFT $c^{2,\ext}$, also of conformal dimension~$1$, with the associated flat F-manifold given by the vector potential $\oF$. Consider its formal shift $c^{2,\ext,(0,\tau)}$. The property (see \cite[Theorem~3.9]{BR18})
\begin{gather*}
\deg c^{2,\ext}_{g,n+1}(e^{\alpha_0}\otimes\otimes_{i=1}^ne_{\alpha_i})=\left(\sum_{i=1}^n q_{\alpha_i}-q_{\alpha_0}+g\right),\quad q_1=0,\quad q_2=\frac{1}{2},
\end{gather*}
together with the fact that $q_\alpha<1$ implies that the sum on the right-hand side of~\eqref{eq:formal shift of an F-CohFT} is finite. Therefore, the F-CohFT $c^{2,\ext,(0,\tau)}$ is well defined for any value of~$\tau$. It is interesting to compare the F-CohFTs $c^{\oF_{(0,\tau)},(0,\lambda)}$ and $c^{2,\ext,(0,\tau)}$.\\

Using the properties of the F-CohFT $c^{2,\ext}$~\cite[Theorem~3.9]{BR18}, it is easy to compute that
$$
c^{2,\ext,(0,\tau)}_{1,1}(e^\alpha)=
\begin{cases}
0,&\text{if $\alpha=1$},\\
\tau\in H^0(\oM_{1,1}),&\text{if $\alpha=2$}.
\end{cases}
$$
On the other hand, the degree zero part of $c^{\oF_{(0,\tau)},(0,\lambda)}_{1,1}(e^\alpha)$ is equal to 
$$
\lambda\left.(\Psi^{-1})^\alpha_2\right|_{(t^1,t^2)=(0,\tau)}=\begin{cases}
0,&\text{if $\alpha=1$},\\
-\frac{\lambda}{\tau},&\text{if $\alpha=2$}.
\end{cases}
$$
We see that the degree zero parts of the F-CohFTs $c^{\oF_{(0,\tau)},(0,-\tau^2)}$ and $c^{2,\ext,(0,\tau)}$ coincide. The whole F-CohFTs can not coincide because, for example, $c^{\oF_{(0,\tau)},(0,-\tau^2)}_{g,n+1}(e^1\otimes e_1^{\otimes n})=0$ for $g\ge 1$, and $c^{2,\ext,(0,\tau)}_{g,n+1}(e^1\otimes e_1^{\otimes n})=\lambda_g$. However, we expect that the F-CohFTs $c^{\oF_{(0,\tau)},(0,-\tau^2)}$ and $c^{2,\ext,(0,\tau)}$ coincide after the restriction to the moduli space of curves of compact type $\cM^\ct_{g,n+1}$ (those whose dual graph is a tree). If this is true, then it would be interesting to study whether the existence of the limit $\lim_{\tau\to 0}c^{\oF_{(0,\tau)},(0,-\tau^2)}$ gives new relations in the cohomology or Chow ring of $\cM^\ct_{g,n+1}$.\\

\smallskip

We finally remark that the partial failure, just observed for the extended $2$-spin F-CohFT, of the Givental-type theory to reconstruct an F-CohFT from its F-TFT and genus $0$ restriction is not unexpected. Indeed, a crucial difference between the $R$-matrix action on F-CohFTs and the corresponding $R$-matrix action on CohFTs~\cite[Section~2.1]{PPZ15} is that in \eqref{eq:R-action} the sum runs over stable trees only, instead of all stable graphs. This restriction seems natural, as the vertex contributions, i.e., the maps $c_{g,n+1}$, need one input and $n$ outputs, but it is clear that, in general, some parts of the full F-CohFT, supported in particular on $\oM_{g,n+1}\setminus \cM_{g,n+1}^\ct$, can be lost. One can hence try to enlarge the Givental group (for instance introducing nonseparating-edge contributions different from the $R$-matrix) and modify the action \eqref{eq:R-action} adding back stable graphs that are not trees, in the effort to recover the lost transitivity.\\

Alternatively, one can look at F-CohFTs $c_{g,n+1}$ as maps to the cohomology groups $H^*(\cM_{g,n+1}^\ct)$ instead of $H^*(\oM_{g,n+1})$ (simply by restriction) and, since the restriction of any cohomology class on $\oM_{g,n+1}$ supported on $\oM_{g,n+1} \setminus \cM_{g,n+1}^\ct$ to $\cM_{g,n+1}^\ct$ is zero, any contribution to the action~\eqref{eq:R-action} of stable graphs that are not trees becomes irrelevant.\\  

Moreover, given an F-CohFT $c_{g,n+1}$ on $\cM_{g,n+1}^\ct$, it is possible to produce a canonical F-CohFT on $\oM_{g,n+1}$ by multiplying it by the top Chern class $\lambda_g$ of the Hodge bundle on $\oM_{g,n+1}$. The result is a well defined F-CohFT on $\oM_{g,n+1}$, because the restriction of the class $\lambda_g$ to $\oM_{g,n+1} \setminus \cM_{g,n+1}^\ct$ is zero. Notice also that multiplying by $\lambda_g$ commutes with the $R$-matrix action on F-CohFTs. This is particularly relevant in view of future applications to the double ramification hierarchy, see \cite{Bur15,BR16,BR18}, which only depends on $\lambda_g \cdot c_{g,n+1}$.\\

\end{document}